\input ./style/arxiv-general.cfg
\documentclass[aop,MSNbibl,dvips]{arximspdf}
\makeatletter
   \@ifpackageloaded{graphicx}{}{\usepackage{graphicx}}
\makeatother
\usepackage{mathbh}
\usepackage{accents}

\doi{10.1214/14-AOP966}% Updated by VTEXPTS2LaTeX.exe, 21.01.2015 13:22
\volume{44}
\issue{1}
\pubyear{2016}
\firstpage{42}
\lastpage{106}
\docsubty{FLA}

\makeatletter
\newcommand{\lleft}{\left}
\newcommand{\rrvert}{\vert}
\newcommand{\rright}{\right}
\newcommand{\llvert}{\vert}
\renewcommand{\mathring}[1]{\accentset{\circ}{#1}}
\newcommand{\eqref}[1]{(\ref{#1})}
\newtheorem{them}{Theorem}[section]
\newtheorem{pro}[them]{Proposition}
\newtheorem{cor}[them]{Corollary}
\newtheorem{lem}[them]{Lemma}

\newcommand{\iint}{\int\!\!\!\int}
\newcommand{\iiint}{\int\!\!\!\int\!\!\!\int}
\newproclaim{defi}[them]{Definition}

\newproclaim{rem}[them]{Remark}
\newproclaim{ex}[them]{Example}

\newcommand{\id}{\operatorname{Id}}
\newcommand{\proj}{\operatorname{proj}}
\newcommand{\Spt}{\operatorname{spt}}
\newcommand{\graph}{\operatorname{graph}}
\newcommand{\card}{\operatorname{card}}
\newcommand{\Conv}{\operatorname{conv}}%Or {Hull} ??
\newcommand{\law}{\operatorname{law}}

\newcommand{\eps}{\varepsilon}
\renewcommand{\P}{\mathbb{P}}
\newcommand{\M}{\Pi_M} %Raum der Martingalkopplungen
\newcommand{\V}{\mathcal{V}} % Raum der Variationen
\newcommand{\R}{\mathbb{R}}
\newcommand{\Q}{\mathbb{Q}}
\newcommand{\N}{\mathbb{N}}
\newcommand{\AND}{ \mbox{and} }
\newcommand{\dd}{\mathrm{d}}
\renewcommand{\d}{\dd}
\newcommand{\E}{\mathbb{E}}
\newcommand{\I}{\mathbh{ 1}}

\newcommand{\step}[1]{\textit{#1}}

\renewcommand{\phi}{\varphi}

\newcommand{\interval}[1]{\mathring{\operatorname{conv}(\operatorname{spt}(#1))}}
\newcommand{\leqc}{\preceq_C} %for leq_convex}{\stackrel{(c)}{\leq}}
\newcommand{\leqe}{\preceq_E} %for leq_convex
\newcommand{\geqc}{\succeq_C} %for leq_convex
\newcommand{\geqe}{\succeq_E} %for leq_convex
\newcommand{\hf}{\mathrm{HF}}
\newcommand{\lc}{\mathrm{lc}}
\newcommand{\hn}{\mathrm{HN}}
\newcommand{\abs}{\mathrm{abs}}

\renewcommand{\subset}{\subseteq}
\makeatother

\begin{document}
\begin{frontmatter}

\title{On a problem of optimal transport under marginal martingale constraints}
\runtitle{Optimal martingale transport problem}

\begin{aug}
% Corresponding author: Mathias Beiglböck - mathias.beiglboeck@univie.ac.at% Updated by VTEXPTS2LaTeX.exe, 21.01.2015 16:03
%mathias.beiglboeck@univie.ac.at% Updated by VTEXPTS2LaTeX.exe,
%21.01.2015 13:22
\author[A]{\fnms{Mathias}~\snm{Beiglb\"ock}\corref{}\ead
[label=e1]{mathias.beiglboeck@univie.ac.at}\thanksref{T1}}%,
%\author[]{\fnms{}~\snm{}\ead[label=]{}}
\and
\author[B]{\fnms{Nicolas}~\snm{Juillet}\ead
[label=e2]{nicolas.juillet@math.unistra.fr}\thanksref{T2}}
\thankstext{T1}{Supported by FWF Grants P21209 and P26736.}
\thankstext{T2}{Supported in part by the Agence Nationale de la
Recherche, ANR-09-BLANC-0364-01.}
\runauthor{M. Beiglb\"ock and N. Juillet}
\affiliation{Universit\"at Wien and Universit\'e de Strasbourg et CNRS}
%\dedicated{}
\address[A]{Fakult\"at f\"ur Mathematik\\
Universit\"at Wien\\
Oskar-Morgensternplatz 1\\
1090 Wien\\
Austria\\
\printead{e1}}
\address[B]{Institut de Recherche Math\'ematique Avanc\'ee\\
%UMR 7501,
Universit\'e de Strasbourg et CNRS\\
7 rue Ren\'e Descartes\\
67000 Strasbourg\\
France\\
\printead{e2}}
\end{aug}

% HISTORY:
%
\received{\smonth{11} \syear{2012}}% Updated by VTEXPTS2LaTeX.exe,
%21.01.2015 13:22
%
\revised{\smonth{8} \syear{2014}}% Updated by VTEXPTS2LaTeX.exe,
%21.01.2015 13:22

% ABSTRACT
%
\begin{abstract}
The basic problem of optimal transportation consists in minimizing the
expected costs $\E[c(X_1,X_2)]$ by varying the joint distribution
$(X_1,X_2)$ where the marginal distributions of the random variables
$X_1$ and $X_2$ are fixed.

Inspired by recent applications in mathematical finance and connections
with the peacock problem, we study this problem under the additional
condition that $(X_i)_{i=1,2}$ is a martingale, that is, $\E[X_2|X_1]=X_1$.

We establish a variational principle for this problem which enables us
to determine optimal martingale transport plans for specific cost functions.
%We introduce a variational lemma that enables us to derive
%characteristic properties of optimal martingale transport plans for
%specific cost functions.
In particular, we identify a martingale coupling
that resembles the classic monotone quantile coupling in several
respects. In analogy with the celebrated theorem of Brenier, the
following behavior can be observed: If the initial distribution is
continuous, then this ``monotone martingale'' is supported by the
graphs of \emph{two} functions $T_1,T_2\dvtx \R\to\R$.

%The corresponding martingales transport plan is the unique minimizer
%for a variety of different cost function. (\emph{In some modest way,
%we should make a reference to Brenier-R\"uschendorf's theorem.} )

%An important tool in our considerations is the representation of
%probability measures by potential functions \`a la Chacon and Walsh.
\end{abstract}
%

% KEYWORDS
\end{frontmatter}

%
%s1 #&#
\section{Introduction}\label{sec:intro}

%s1.1 #&#
\subsection{Presentation of the martingale transport problem}
We will denote by $\mathcal{P}$ the set of probability measures on $\R
$ having
finite first moments. We are given measures $\mu, \nu\in\mathcal
{P}$, and
a (measurable) \emph{cost function} $c\dvtx \R\times\R\to\R$ which
will be
continuous in most of our applications.
%a measurable \emph{cost function}\footnote{Apart from one notable
%exception we will only be interested in continuous cost functions.} $c:
%\R\times\R\to\R$.
We assume moreover that $c(x,y)\geq a(x)+b(y)$ where $a$ (resp., $b$)
is integrable with respect to $\mu$ (resp., $\nu$). Hence if $(X, Y)$
is a joint law with marginal distributions $\law X=\mu$ and $\law Y
=\nu
$, the expectation of $c(X,Y)\geq a(X)+b(Y)$ is well defined, taking
its value in $ [ \E[a(X)]+\E[b(Y)],+\infty ]$. We will refer to
this technical hypothesis as \emph{the sufficient integrability
condition}. The basic problem of optimal transport consists in the
minimization problem
%
%e1 #&#
\begin{equation}
\label{big_cost} \mbox{Minimize}\quad %C(\mu,\nu)=\inf%_{(X_1, X_2)}
\E\bigl[c(X, Y)\bigr]\quad \mbox{subject to}\quad
\law(X)=\mu, \law(Y)=\nu,
\end{equation}
where the infimum is taken over all joint distributions. We denote the
infimum in~\eqref{big_cost} by $C(\mu,\nu)$. The joint laws on $\R
\times\R$ are usually called \emph{transport plans} after the
classical concrete problem of Monge \cite{Mo1781}: How can one
transport a heap of soil distributed according to $\mu$ to a target
distribution $\nu$? A transport plan $\pi$ prescribes that for
$(x,y)\in\R^2$ a quantity of mass $\pi(\d x\,\d y)$ is transported from
$x$ to $y$. Minimizers of the problem \eqref{big_cost} are called
{optimal transport plans}. Note that we will also use the more
probabilistic term \emph{coupling} for transport plans. Following
\cite
{Vi03}, we denote the set of all transport plans by $\Pi(\mu, \nu)$ so
that one has the alternative definition
\[
C(\mu,\nu)=\inf_{\pi\in\Pi(\mu,\nu)}\iint c(x,y) \,\d\pi(x,y).
\]
Our main interest lies in a \emph{martingale version} of the transport
problem. That is, our aim is to minimize $\E[c(X, Y)]$ over the set of
all \emph{martingale transport plans}
\[
\M(\mu, \nu)= \bigl\{\pi\in\Pi(\mu, \nu)\dvtx \pi=\law(X, Y)\mbox{ and }\E
[Y|X]=X\bigr\}.
\]
A transport plan $\pi$ is equivalently described through its
disintegration $(\pi_x)_{x\in\R}$ with respect to the initial
distribution $\mu$. The probabilistic interpretation is that
$(x,A)\mapsto\pi_x(A)$ is the transition kernel of the two-step process
$(X_i)_{i=1,2}$ where $X_1=X$ and $X_2=Y$, that is, $\pi_x(A)= \P(Y\in
A|X=x)$. In these terms, $\pi$ is an element of $\M(\mu, \nu)$, if and
only if $\int y \,\d\pi_x(y)=x$ holds $\mu$-a.s.
Hence, in this paper we study the minimization problem
%
%e2 #&#
\begin{equation}
\label{PrimalMart} \mbox{Minimize} \quad\E_\pi[c]=\iint c(x,y) \,\d\pi(x,y)\quad
\mbox{subject to}\quad \pi\in\M (\mu, \nu)
\end{equation}
for various costs. Let $C_M(\mu,\nu)$ denote the infimum $\inf\{\E
_\pi
[c]\dvtx \pi\in\M(\mu, \nu)\}$.

Our optimal transport approach permits to distinguish some special
couplings of $\M(\mu,\nu)$ that are comparable to the monotone (or
Hoeffding--Fr\'echet) coupling $\pi_\hf\in\Pi(\mu,\nu)$. Indeed, we
have developed our martingale transport theory parallel to the
classical theory and the optimizer of \eqref{PrimalMart} will enjoy
canonical properties. Nevertheless, notable differences occur between
the theories. An obvious one is the fact that $\M(\mu,\nu)$ can be
empty while $\Pi(\mu,\nu)$ always contains the element $\mu\otimes
\nu$.
The existence of a martingale transport plan is actually quite an old
topic that is present (but under different names) at least since the
study of Muirhead's inequality by Hardy, Littlewood and P\'olya \cite
{HaLiPo52}. Several articles in different fields (analysis,
combinatorics, potential theory and probability) deal with this
question in different settings, often for marginal distributions in
spaces much more general than the real line (see, e.g., \cite
{Bl51,Sh51,Mi61,CaFeMe64,St65,Do68,Ke72,FiHo80}). The interest in
finding an explicit coupling has appeared recently in the peacock
problem (see \cite{HiPrRoYo11} and the references therein): a peacock
is a stochastic process $(X_t)_{t\in I}$ such that there exists at
least one martingale $(M_t)_{t\in I}$ satisfying $\law(X_t)=\law(M_t)$
for every $t$. The problem consists in building as explicitly as
possible such a martingale $(M_t)$ from~$(X_t)$.
The martingale transport problem is maybe even closer linked to the
theory of model-independent pricing in mathematical finance.\setcounter{footnote}{2}\footnote
{We refer to the recent survey by Hobson \cite{Ho11} for a very
readable introduction to this area. Arguably, the most important tool
in model-independent finance is the Skorokhod-embedding approach; an
extensive overview is given by Ob{\l}{\'o}j in \cite{Ob04}.}
Indeed, the problem \eqref{PrimalMart} has been first studied in this
context by Hobson and Neuberger \cite{HoNe11} for the specific cost
function $c(x,y)=-|y-x|$. %\marginpar{Hier k\"onnten wir Skorokhod
%embedding und Obloj's Uebersichtspaper zu den Thema erwaehnen, das
%wuerde ihn freuen. \emph{Mach das wenn du magst!}}
The link between optimal transport and model-independent pricing has
been made explicit in \cite{BeHePe11} in a discrete time framework and
by Galichon, Henry-Labordere and Touzi \cite{GaHeTo11} in a continuous
time setup.

We note that several of the basic features of the problem \eqref
{PrimalMart} are similar to the usual optimal transport problem. This
appeals, for instance, to the weak compactness of $\Pi(\mu,\nu)$ and
$\M
(\mu,\nu)$. If $c$ is lower semicontinuous,
this carries over to the mapping $\pi\mapsto\E_\pi[c]$ for either
space of transport plans.
In particular, the infimum is attained. Note also that as in the
standard setup the problem has a natural dual formulation \cite
{BeHePe11}. However, as we already mentioned in the previous paragraph,
while there is always a transport plan which moves $\mu$ to $\nu$, the
marginal distributions need to satisfy additional assumptions to
guarantee that a \emph{martingale} transport plan exists: The set $\M
(\mu, \nu)$ is nonempty if and only if $\mu$ is smaller than $\nu$ in
the \emph{convex order} (see Definition~\ref{ConvexOrderDef}). More
details are provided in Section~\ref{sec:basic} along with a construction of a
martingale transport plan between two given marginals.

%The main topic of this article is to study the problem
%\eqref{PrimalMart} for specific choices of the cost function $c$. For
%comparison we review some relevant results in the usual setup.

%s1.2 #&#
\subsection{Summary on the classical transport problem on $\mathbb{R}$}\label
{summary}
% Recall that any probability measure $\theta\in\p(\R)$ can be written
%$G_\#\Lg$ for some monotonically increasing map $G$. Indeed let $F_
%\theta$ be the cumulative distribution function of $\theta$ defined by
%$F_\theta(t)=\theta(]-\infty,t])$ and $G_\theta$ the generalized
%inverse of $F_\theta$ defined by

%$$G_\theta(t)=\inf\{s\in\R| F_\theta(s)\geq t\}.$$ It is well-known
%that $(G_\theta)_\#\Lg=\mu$. \marginpar{\small I have erased the
%figure for now. Perhaps we should decide later how important the
%quantile functions are and were they should be introduced.}
%
%\begin{figure}[ht]
%\begin{center}
%\includegraphics[width=10cm]{cumulative}
%\caption{Illustration.}\label{cumulative}
%\end{center}
%\end{figure}

A cornerstone in the modern theory of optimal transportation is
Brenier's theorem (or Brenier--Rachev--R\"uschendorf theorem); see
\cite
{Br91,RuRa90}. It treats the optimal transport problem in the
particular case $c(x,y)= |y-x|^2$, where $|\cdot |$ denotes the Euclidean
norm on $\R^n$. This is simply problem \eqref{big_cost} when $\mu$ and
$\nu$ are interpreted as measures on $\R^n$. Under appropriate
regularity conditions on $\mu$, the optimal transport $\pi\in\Pi
(\mu,\nu
)$ is unique and supported by the graph of a function $T\dvtx \R^n\to\R^n$
that is the gradient of some convex function. In particular, the
optimal transport is realized by a mapping. Note that in dimension one
the gradient of a convex function is simply a monotonically increasing
function so that the optimal coupling is the usual monotone coupling.
This fact can be directly proved without too many difficulties (see,
e.g., \cite{Ju11}) but nevertheless it is interesting as one of
the rare cases where an optimal transport plan can be so easily
understood. Moreover, even without any assumption on $\mu$, the
monotone coupling is the unique optimal transport plan. In this paper,
we will see that similar results are valid in the martingale case, for
example, the uniqueness of the minimizer or the fact that the optimal
coupling is concentrated on a special set comparable to the graph of a
monotone mapping.%(uniqueness, optimal transport concentrated on a
%special set

We present the classical (nonmartingale) optimal transport problem on
the real line that will serve as a guideline to our paper. The results
are given for an arbitrary strictly convex cost. Any cost of this type
activates the same theory, which again is characteristic of dimension one.
%\marginpar{Do we take $\mu, \nu$ compactly supported for simplicity?}

%th1.1 #&#
\begin{them}\label{classic}
Let $\mu, \nu$ be probability measures and $c$ a cost function defined
by $c(x,y)= h(y-x)$, where $h\dvtx \R\to\R$ is a strictly convex function.
We assume that $c$ satisfies the sufficient integrability condition
with respect to $\mu$ and $\nu$ and that $C(\mu,\nu)<\infty$. The
following statements are equivalent:
\begin{longlist}[(1)]
\item[(1)] The measure $\pi$ is optimal.
\item[(2)] The transport preserves the order, that is, there is a set
$\Gamma$ with $\pi(\Gamma)=1$ such that whenever $(x,y), (x',y')\in
\Gamma$, if $x<x'$ one has also $y\leq y'$.
%\item If $(x,y)\in\Spt(\pi)$ then $\pi(]x,+\infty[\times]-\infty,
%y[)=0$.
%\item The measure $\pi$ is the distribution of the random vector $(G_
%\mu(U),G_\nu(U))$ where $U$ is uniform on $[0,1]$, i.e.\ it is the
%increasing rearrangement of its marginals, alias the quantile coupling.
\end{longlist}
\end{them}

We have the two following corollaries.

%co1.2 #&#
\begin{cor}\label{MonCou}
For given measures $\mu$ and $\nu$, if $C(\mu,\nu)$ is finite then
there exists a unique minimizer to the transport problem \eqref
{big_cost} and it is the monotone (Hoeffding--Fr\'echet) coupling $\pi
_{\hf}$.
\end{cor}

One has in fact $\pi_\hf=(G_\mu\otimes G_\nu)_\#\lambda_{[0,1]}$ where
$\lambda$ is the Lebesgue measure and $G_\mu$ and $G_\nu$ are the
\emph
{quantile functions} of $\mu$ and $\nu$, that is, the nondecreasing
and left-continuous functions obtained from the cumulative distribution
functions $F_\mu$ and $F_\nu$ as a generalized inverse by the formula
$G(s)=\inf\{t\in\R\dvtx s\leq F(t) \}$.\footnote{Note that the function $G$
may take infinite values at the boundary of its domain $[0,1]$.}
This observation is the reason why the coupling $\pi_\hf$ is also known
under the alternative name \emph{quantile coupling}.

For the following corollary, we recall that a measure $\mu$ is said to
be continuous if $\mu(\{x\})=0$ for every $x\in\R$.

%co1.3 #&#
\begin{cor}\label{MonMap}
Under the assumptions of Corollary~\ref{MonCou}, if $\mu$ is continuous
then the optimal transport plan $\pi_{\hf}$ is concentrated on the
graph of an increasing mapping $T\dvtx \R\to\R$. Moreover, $T_\# \mu=\nu$.
\end{cor}

It is straightforward to see that $T=G_\nu\circ F_\mu$. This formula
determines $T$, $\mu$-a.s.

\subsubsection*{Quadratic costs in the martingale setting}
While $c(x,y )=(y-x)^2$ is arguably the most important cost function in
the theory of optimal transport, we stress that it plays a rather
different role in the martingale setup. Assume that $\law(X)=\mu$ and
$\law(Y)=\nu$ are linked by a martingale coupling $\pi$ and posses
second moments. Then
\[
\E[X Y]=\E \bigl[\E[X Y|X] \bigr]=\E\bigl[X^2\bigr],
\]
hence we have the Pythagorean relation
\[
\int(y-x)^2 \,\d\pi(x,y)= \E\bigl[(Y-X)^2\bigr]= \E
\bigl[Y^2\bigr]-\E\bigl[X^2\bigr].
\]
Thus, the cost associated to $\pi$ depends only on the marginal
distributions, that is, not on the particular choice of $\pi\in\M
(\mu
,\nu)$.

We record the following consequence: Let $c$ be a cost function and
assume that
\[
\tilde c(x,y)= c(x,y)+ p\cdot(y-x)^2 + q\cdot(y-x)
\]
for some real constants $p$ and $q$. Then in problem \eqref{PrimalMart}
the minimizers are the same for the costs $c$ and $\tilde c$. In
particular, if $c(x,y)=h(y-x)$, we do not expect that monotonicity or
convexity properties of the function $h$ are relevant for the structure
of the optimizer.

%s1.3 #&#
\subsection{A new coupling: The monotone martingale coupling, main
results}
In this section, we will discuss a particular coupling which may be
viewed as a martingale analogue to the monotone (Hoeffding--Fr\'echet)
coupling. Notable similarities are that it is canonical with respect to
the convex order as well as that it is optimal for a range of different
cost functions.

%\subsubsection{Monotonicity}
%We will base our definition on property (2) in Theorem~\ref{classic}.

%de1.4 #&#
\begin{defi}\label{defi_monoto}
%\marginpar{Or ``$x'>x$ implies $y'\notin]y^-,y^+[$''?.}
A martingale transport plan $\pi$ on $\R\times\R$ is \emph
{left-monotone} or simply \emph{monotone} if there exists a Borel set
$\Gamma\subseteq\R\times\R$ with $\pi(\Gamma)=1$ such that whenever
$(x,y^-), (x,y^+), (x',y')\in\Gamma$ we cannot have (see Figure~\ref
{fig:forbidden} where this situation is represented)
%
%e3 #&#
\begin{equation}
\label{BadConfiguration}x < x' \quad\mbox{and}\quad y^-<y'<y^+.
\end{equation}
Respectively, $\pi$ is said to be \emph{right-monotone} if there exists
$\Gamma$ such that if $(x,y^-)$, $(x,y^+)$ and $(x',y')$ are elements
of $\Gamma$ then we do not have
\[
x >x' \quad\mbox{and}\quad y^-<y'<y^+.
\]
We will refer to the set $\Gamma$ as the \emph{monotonicity set} of
$\pi$.
\end{defi}

In this paper, we will only state the results for (left-)monotone
couplings. The corresponding results for right-monotone couplings can
be deduced easily. We illustrate the forbidden situation \eqref
{BadConfiguration} in Figure~\ref{fig:forbidden}. Note that the top
line represents the measure $\mu$ while $\nu$ is distributed on the
bottom line; this convention will also be used in the subsequent pictures.

%f1 #&#
\begin{figure}

\includegraphics{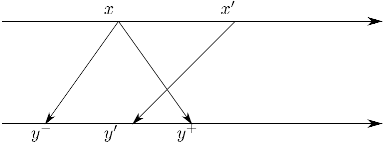}

\caption{The forbidden mapping.}\label{fig:forbidden}
\end{figure}

%\begin{figure}[ht]
%\begin{center}
% \def\svgwidth{7.6cm}
%%\input{forbidden.pdf_tex}
%\begingroup
% \makeatletter
% \providecommand\color[2][]{%
% \errmessage{(Inkscape) Color is used for the text in Inkscape, but
%the package 'color.sty' is not loaded}
% \renewcommand\color[2][]{}%
% }
% \providecommand\transparent[1]{%
% \errmessage{(Inkscape) Transparency is used (nonzero) for the text
%in Inkscape, but the package 'transparent.sty' is not loaded}
% \renewcommand\transparent[1]{}%
% }
% \providecommand\rotatebox[2]{#2}
% \ifx\svgwidth\undefined
% \setlength{\unitlength}{417.75200007pt}
% \else
% \setlength{\unitlength}{\svgwidth}
% \fi
% \global\let\svgwidth\undefined
% \makeatother
% \begin{picture}(1,0.40155403)%
% \put(0,0){\includegraphics[width=\unitlength]{forbidden.pdf}}%
% \put(0.26810165,0.34167251){\color[rgb]{0,0,0}\makebox(0,0)[lb]{
%\smash{$x$}}}%
% \put(0.57450353,0.34167251){\color[rgb]{0,0,0}\makebox(0,0)[lb]{
%\smash{$x'$}}}%
% \put(0.07660047,0.01612051){\color[rgb]{0,0,0}\makebox(0,0)[lb]{
%\smash{$y^-$}}}%
% \put(0.45960283,0.01612051){\color[rgb]{0,0,0}\makebox(0,0)[lb]{
%\smash{$y^+$}}}%
% \put(0.26810165,0.01612051){\color[rgb]{0,0,0}\makebox(0,0)[lb]{
%\smash{$y'$}}}%
% \end{picture}%
%\endgroup
%
%\end{center}
%\end{figure}

The next theorem is proved in Section~\ref{sec:unique}.

%th1.5 #&#
\begin{them}
Let $\mu, \nu$ be probability measures in convex order. Then there
exists a unique (left-)monotone transport plan in $\Pi_M(\mu,\nu)$. We
denote this coupling by $\pi_\lc$ and call it \emph
{left-curtain}\footnote{This name is explained in some detail before
Theorem~\ref{lc_defi}.} coupling.
\end{them}

Of course, one does not expect that a martingale is concentrated on the
graph of a deterministic mapping $T$; this holds only in the trivial
case when $\mu= \nu$ and $T(x) \equiv x$. Rather we have the
following result.

%co1.6 #&#
\begin{cor}\label{FamousCoro}
Let $\mu, \nu$ be probability measures in convex order and assume that
$\mu$ is continuous. Then there exist a Borel set $S\subseteq\R$ and
two measurable functions $T_1, T_2\dvtx S\to\R$ such that:
\begin{longlist}[(1)]
\item[(1)]$\pi_\lc$ is concentrated on the graphs of $T_1$ and $T_2$.
\item[(2)] For all $x\in\R, T_1(x)\leq x \leq T_2(x)$.
\item[(3)]For all $x<x'\in\R$, $T_2(x) < T_2(x')$ and $T_1(x')\notin
\,]T_1(x), T_2(x)[$.
\end{longlist}
\end{cor}

The following picture (Figure~\ref{fig:curtain}) illustrates the
coupling $\pi_\lc$ in a specific case.
The measures $\mu$ and $\nu$ are Gaussian distributions having the same
mean, the variance of $\nu$ being greater than the variance of $\mu$.
There exist two points at which the density of $\mu$ (w.r.t. Lebesgue
measure) equals the density of $\nu$. Denote the smaller of these
points by $x_0$. Then we have $T_1(x)=T_2(x)=x$ for $x<x_0$. For
$x>x_0$, the map $T_1$ is strictly decreasing and $T_2$ is strictly increasing.

%% Creator: Inkscape 0.48.0, www.inkscape.org
%% Accompanies image file 'shadow.pdf' (pdf, eps, ps)
%% \input{<filename>.pdf_tex}
%% instead of
%% \includegraphics{<filename>.pdf}
%% \def\svgwidth{<desired width>}
%% \input{<filename>.pdf_tex}
%% instead of
%% \includegraphics[width=<desired width>]{<filename>.pdf}
%% \usepackage{import}
%% \import{<path to file>}{<filename>.pdf_tex}
%% \graphicspath{{<path to file>/}}
%% http://tug.ctan.org/tex-archive/info/svg-inkscape

%f2 #&#
\begin{figure}

\includegraphics{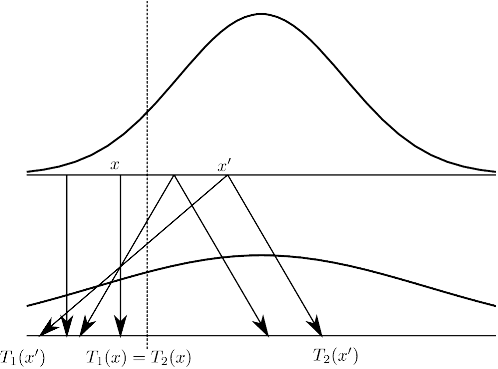}

\caption{Scheme of the left-curtain $\pi_\lc$ coupling between two
Gaussian measures.}\label{fig:curtain}
\end{figure}

%\begin{figure}[ht]
%\begin{center}
% \def\svgwidth{10cm}
%%\input{curtain.pdf_tex}
%
%\begingroup
% \makeatletter
% \providecommand\color[2][]{%
% \errmessage{(Inkscape) Color is used for the text in Inkscape, but
%the package 'color.sty' is not loaded}
% \renewcommand\color[2][]{}%
% }
% \providecommand\transparent[1]{%
% \errmessage{(Inkscape) Transparency is used (nonzero) for the text
%in Inkscape, but the package 'transparent.sty' is not loaded}
% \renewcommand\transparent[1]{}%
% }
% \providecommand\rotatebox[2]{#2}
% \ifx\svgwidth\undefined
% \setlength{\unitlength}{591.159375pt}
% \else
% \setlength{\unitlength}{\svgwidth}
% \fi
% \global\let\svgwidth\undefined
% \makeatother
% \begin{picture}(1,0.8098391)%
% \put(0,0){\includegraphics[width=\unitlength]{curtain.pdf}}%
% \put(0.22039393,0.46960497){\color[rgb]{0,0,0}\makebox(0,0)[lb]{
%\smash{$x$}}}%
% \put(0.43691761,0.46573847){\color[rgb]{0,0,0}\makebox(0,0)[lb]{
%\smash{$x'$}}}%
% \put(0.63024232,0.0829555){\color[rgb]{0,0,0}\makebox(0,0)[lb]{
%\smash{$T_2(x')$}}}%
% \put(-0.00192947,0.07908904){\color[rgb]{0,0,0}\makebox(0,0)[lb]{
%\smash{$T_1(x')$}}}%
% \put(0.16046328,0.07908904){\color[rgb]{0,0,0}\makebox(0,0)[lb]{
%\smash{ $T_1(x)
%=T_2(x)$}}}%
% \end{picture}%
%\endgroup
%\caption{Scheme of the left-curtain $\pi_\lc$ coupling between two
%Gaussian measures.}\label{fig:curtain}
%\end{center}
%\end{figure}

The subsequent result states that the transport plan $\pi_\lc$ is
optimal for a variety of different cost functions. (See Theorem~\ref
{OptForExp} below.)

%th1.7 #&#
\begin{them}[($\pi_\lc$ is optimal)]\label{MonIsOpt}
Let $\mu,\nu$ be probability measures in convex order.
Assume that $c(x,y)=h(y-x)$ for some differentiable function $h$ whose
derivative is strictly convex and that $c$ satisfies the sufficient
integrability condition. If $C_M(\mu,\nu)<\infty$, then $\pi_\lc$ is
the unique optimizer.
\end{them}

Natural examples of cost functions to which the result applies are
given by $c(x,y)=(y-x)^3$ and $c(x,y)=\exp(y-x)$.

We discuss a further characteristic property of the transport plan $\pi
_\lc$. For a real number $t$ and $\pi\in\Pi(\mu,\nu)$, consider the
measure
\[
\nu^\pi_t:=\proj^y_\# \pi|_{]{-}\infty, t]\times\R},
\]
where $\proj^y\dvtx (a,b)\in\R^2\mapsto b\in\R$. Loosely speaking, the mass
$\mu|_{]{-}\infty, t]}$ is moved to $\nu_t^\pi$ by the transport plan
$\pi
$. It is intuitively clear (and not hard to verify) that a transport
plan $\pi\in\Pi(\mu, \nu)$ is uniquely determined by the family
$(\nu
_t^\pi)_{t\in\R}$.

Using this notation, the classic monotone transport plan $\pi_\hf$ is
characterized by the fact that for each $t$, the measure $\nu^{\pi
_\hf
}_t$ is as \emph{left} as possible. More precisely, for every $t$ the
measure $\nu^{\pi_\hf}_t$ is minimal with respect to the first-order
stochastic dominance in the family
\[
\bigl\{\nu^\pi_t\dvtx \pi\in\Pi(\mu,\nu)\bigr\}.
\]

We have the following, analogous characterization for the monotone
martingale coupling $\pi_\lc$. This is in fact the way we will formally
define $\pi_\lc$ in Theorem~\ref{lc_defi}.

%th1.8 #&#
\begin{them}[($\pi_\lc$ is canonical with respect to the convex
order)]\label{cano}
For every real number $t$, the measure $\nu^{\pi_\lc}_t$ is minimal
with respect to the convex order (i.e., second-order stochastic
dominance) in the family
\[
\bigl\{\nu^\pi_t\dvtx \pi\in\M(\mu,\nu)\bigr\}.
\]
\end{them}

The next theorem summarizes the properties of $\pi_\lc$.

%th1.9 #&#
\begin{them} \label{synth} Let $\mu, \nu$ be probability measures in
convex order.
Let $h\dvtx \R\to\R$ be a differentiable function such that $h'$ is strictly
convex and assume that the cost function $c\dvtx (x,y)\mapsto h(y-x)$
satisfies the sufficient integrability condition.

We assume moreover $C_M(\mu,\nu)<+\infty$. Let $\pi$ be a martingale
coupling in $\Pi_M(\mu,\nu)$. The following statements are equivalent:
\begin{itemize}
\item The coupling $\pi$ is monotone.
\item The coupling $\pi$ is optimal.
\item The coupling $\pi$ is the left-curtain coupling $\pi_\lc$: for
every $(\pi',t)\in\Pi_M(\mu,\nu)\times\R$, the measure $\nu^\pi
_t$ is
smaller than $\nu^{\pi'}_t$ in the convex order.
\end{itemize}
\end{them}

Note that Theorem~\ref{synth} is a consequence of the other results
stated above.

%s1.4 #&#
\subsection{A ``variational principle'' for the martingale transport problem}

An important basic tool in optimal transport is the notion of \emph
{$c$-cyclical monotonicity} (see~\cite{Vi09}, Chapter~4) which links
the optimality of transport plans to properties of the support of the
transport plan.
A parallel statement holds true in the present setup and plays a
fundamental role in our considerations. Heuristically, we expect that
if $\pi\in\M(\mu, \nu)$ is optimal, then it will prescribe optimal
movements for single particles.
To make this precise, we use the following notion.

%de1.10 #&#
\begin{defi}\label{def:competitor}
Let $\alpha$ be a measure on $\R\times\R$ with finite first moment in
the second variable.
We say that $\alpha'$, a measure on the same space, is a \emph
{competitor} of $\alpha$ if $\alpha'$ has the same marginals as
$\alpha
$ and for $(\proj^x_\#\alpha)$-a.e. $x\in\R$
\[
\int y \,\d\alpha_x(y)=\int y \,\d\alpha_x'(y),
\]
where $(\alpha_x)_{x\in\R}$ and $(\alpha'_x)_{x\in\R}$ are
disintegrations of the measures with respect to $\proj^x_\#\alpha$.
\end{defi}

We can now formulate a ``variational principle'' for the martingale
transport problem.

%le1.11 #&#
\begin{lem}[(Variational lemma)]\label{GlobalLocal}
Assume that $\mu, \nu$ are probability measures in convex order and
that $c\dvtx \R^2\to\R$ is a Borel measurable cost function satisfying the
sufficient integrability condition. Assume that $\pi\in\M(\mu,\nu
)$ is
an optimal martingale transport plan which leads to finite costs.
Then there exists a Borel set $\Gamma$ with $\pi(\Gamma)=1$ such that
the following holds:

If $\alpha$ is a measure on $\R\times\R$ with $|\operatorname
{spt}(\alpha)|<\infty$
and $\operatorname{spt}(\alpha) \subseteq\Gamma$, then we have
$\int c \,\d\alpha\leq
\int c \,\d\alpha'$ for every competitor $\alpha'$ of $\alpha$.
\end{lem}

Indeed, under the additional assumption that the cost function $c$ is
continuous and bounded we can prove that the condition given in the
variational lemma is not only necessary but also sufficient to
guarantee that a measure is optimal; see Lemma~\ref{LocalGlobal} in
Appendix \ref{appA}.

The variational Lemma~\ref{GlobalLocal} is one of the key ingredients
in our investigation of the monotone martingale transport plan $\pi
_\lc
$ introduced above.
Moreover, it turns out to be very useful if one seeks to derive results
on the optimizers for various specific cost functions.
Assuming for simplicity that $\mu$ is continuous, Lemma~\ref
{GlobalLocal} allows us to derive the following results:
\begin{longlist}[(1)]
%\item If $c(x,y)=(y-x)^3$, then $\card(\Spt\pi_x)\leq2$, $
%\mu(x)$-a.s.
%
\item[(1)] If $c(x,y)=(y-x)^4$, then $\card(\Spt\pi_x)\leq3$, $\mu(x)$-a.s.
\item[(2)] Assume that $c(x,y)= h(y-x)$ for some continuously differentiable
function $h$ and that the derivative $h'$ intersects every affine
function at most in $k\in\N$ points. Then $\card(\Spt\pi_x)\leq k$,
$\mu(x)$-a.s. for the optimizing $\pi$. (See Theorem~\ref{LessThanK},
and also Theorem~\ref{ClassicLessThanK} for a similar result which
appeals to the classical transport problem.)
\item[(3)] If $c(x,y)=-|y-x|$, then there is a unique optimizer $\pi\in\M
(\mu,\nu)$. Moreover, $\card(\Spt\pi_x)\leq2$, $\mu(x)$-a.s. (This
was first shown in \cite{HoNe11}; see Theorem~\ref{HobsonNeuberger}.)
\item[(4)] If $c(x,y)=|y-x|$, then there is a unique optimizer $\pi\in\M
(\mu
,\nu)$. Moreover, $\card(\Spt\pi_x)\leq3$ and $\card(\Spt\pi_x
\setminus\{x\} )\leq2$, $\mu(x)$-a.s. (see Theorem~\ref{RHN}).
\end{longlist}
Having financial applications in mind, the cost functions
$c(x,y)=|y-x|$ and $c(x,y)=-|y-x|$ are particularly relevant, we refer
to the work of Hobson and Neuberger \cite{HoNe11}.%
%\marginpar{Referenzproblem mit HoNe12?}

%s1.5 #&#
\subsection{Organization of the paper}

We will start with a warm up section (Section~\ref{sec:basic}) in which
we derive some basic properties and explain a procedure that allows to
find a martingale coupling for two given measures in convex order.
Then, in Section~\ref{sec:gamma}, we establish the variational Lemma~\ref{GlobalLocal} which will play a crucial role throughout the paper.
In Section~\ref{sec:curtain}, we introduce and study the shadow
projection, which permits us to introduce the left-curtain transport
plan $\pi_\lc$. We define it in Theorem~\ref{lc_defi} through its
canonical property with respect to the convex order, we explain the
name ``left-curtain'' and prove that it is monotone in Theorem~\ref
{them:mono}. The particular properties of the transport plan $\pi_\lc$
are established in Sections~\ref{sec:unique} and \ref{sec:optimal}. In
Section~\ref{sec:general}, we present results related to other costs
and other couplings. Finally, in the \hyperref[app]{Appendix}, we present a converse to
the variational Lemma~\ref{GlobalLocal}. We also provide an alternative
derivation of Lemma~\ref{GlobalLocal} which is longer than argument
presented in Section~\ref{sec:gamma} but has the advantage to be
constructive and self-contained. %\marginpar{N: something about the
%other appendix?}

%s2 #&#
\section{Construction of a martingale transport plan for
measures}\label
{sec:basic}

In this section, we extend the martingale optimal transport problem to
general finite measures with finite first moment and we define the
convex order on this space. We prove that there exists a martingale
transport plan between two measures in convex order and give a very
short description of the duality theory linked to our optimization problem.
%s2.1 #&#
\subsection{Basic notions}\label{ParBasicNotions}

Denote by $\mathcal{M}$ the set of finite measures on $\R$ having
finite first
moment. We consider it with the usual topology, that is, we say
%%marginpar{vor HANIN [154] cit\'e p213 Villani1}
that a sequence $(\nu_n)_n$ converges weakly in $\mathcal{M}$ to an
element $\nu
\in\mathcal{M}$ if:
\begin{longlist}[(1)]
\item[(1)]$(\nu_n)_n$ converges weakly in the usual sense, that is, using
continuous bounded functions as test functions;
\item[(2)] the sequence $\int|x| \,\d\nu_n$ converges to $\int|x| \,\d\nu$.
\end{longlist}
Note that this is the same as adding all functions that grow at most
linearly in $\pm\infty$ to the set $\mathcal C_b$ of continuous and
bounded test functions.

The reason we are interested in the space $\mathcal{M}$ is that we
will need to
consider also transport plans between measures $\mu, \nu\in\mathcal
{M}$ which
have (the same) mass $k$, where $k$ is possibly different from $1$.
In direct generalization of the earlier definition, the set of
transport plans $\Pi(\mu, \nu)$ then consists of all Borel measures
$\pi
$ on $\R\times\R$ satisfying $\proj^x_\# \pi=\mu$, $\proj^y_\#
\pi=\nu$.
As a consequence of Prohorov's theorem, the set $\Pi(\mu, \nu)$ is
compact; see, for example, \cite{Vi09}, Lemma~4.4, for details. If $c$
is a continuous (or lower semicontinuous) cost function satisfying the
sufficient integrability condition with respect to $\mu$ and $\nu$,
then the cost functional
\[
\pi\in\Pi(\mu,\nu) \mapsto\int c \,\d\pi\in\,]{-}\infty,+\infty]
\]
is lower semicontinuous w.r.t. the weak topology (\cite{Vi09}, Lemma~4.3). It follows that the infimum in the classic transport
problem is attained.

We proceed analogously in the martingale setup. If $\mu$ and $\nu$ are
not necessarily probabilities, we define $\M(\mu, \nu) $ to consist of
all transport plans $\pi$ such that the disintegration in probability
measures $(\pi_x)_{x\in\R}$ w.r.t. $\mu$ satisfies
\[
\int y \,\d\pi_x(y)=x
\]
for $\mu$-almost every $x$.
Then
$\pi\in\Pi(\mu, \nu) $ is a martingale measure if and only if
%
%e4 #&#
\begin{equation}
\label{MartTest}\int\rho(x) (y-x) \,\d\pi(x,y)=0
\end{equation}
%
%\marginpar{ich w\"urde vielleicht $\Delta$ durch $\Lambda$ ersetzen: $
%\Delta$ is auch die Diagonale im Kapitel 7.}
for all bounded measurable functions $\rho\dvtx \R\to\R$.
To see whether $\pi$ is a martingale measure, it is of course enough to
test \eqref{MartTest} for a sufficiently rich class of functions,
for example, for all functions of the form $\rho=\I_{]{-}\infty, x]},
x\in\R$ or for all continuous bounded functions (see \cite{BeHePe11}, Lemma~2.3).

Hence, the set $\Pi_M(\mu, \nu) $ is compact in the weak topology (see
\cite{BeHePe11}, Proposition~2.4). Precisely as in the usual setup it
follows that the value of the minimization problem \eqref{PrimalMart}
is attained provided that the set $\M(\mu, \nu) $ is nonempty.

Of course, it is a fundamental question on which conditions martingale
transport plans exist. In the usual optimal transport setup, the
problem is simple enough: the properly renormalized product measure
$\frac{1}{\mu(\R)} \mu\otimes\nu$ witnesses that $\Pi(\mu, \nu)$ is
nonempty. As mentioned in the \hyperref[sec:intro]{Introduction}, the proper notion which
guarantees existence of a martingale transport plan is the convex
order. As it plays a crucial role throughout the paper, we will discuss
it in some detail.

%s2.2 #&#
\subsection{The convex order of measures}
Let us start with the definition.

%de2.1 #&#
\begin{defi}\label{ConvexOrderDef}%\marginpar{Do we really need $\m$
%and the first moment in the definition?}
Two measures $\mu$ and $\nu$ are said to be in convex order\footnote
{The convex order is also called Choquet order or second-order
stochastic dominance.} if:
\begin{longlist}[(1)]
\item[(1)] they have finite mass and finite first moments, that is, lie in
$\mathcal{M}$,
\item[(2)] for convex functions $\varphi$ defined on $\R$, $\int\varphi
\,\dd
\mu\leq\int\varphi\,\dd\nu$.
\end{longlist}
In that case, we will write $\mu\leqc\nu$.
\end{defi}

Note that if $\mu\leqc\nu$, then one can apply (2) to all affine
functions. Using the particular choices $\phi(x)\equiv1$ and $\phi
(x)\equiv-1$, one obtains that $\mu$ and $\nu$ have the same total
mass and considering the functions $\phi(x)\equiv x$ and $\phi
(x)\equiv
-x$ one finds that $\mu$ and $\nu$ have the same barycenter.\footnote
{The \emph{barycenter} or \emph{mean} of a measure $\mu$ is $\frac{1}{\mu
(\R)}\int x \,\d\mu(x)$. }

It is useful to know that it is sufficient to test hypothesis (2)
against suitable subclasses of the convex functions. For instance,
measures $\mu, \nu$ having the same finite mass and the same first
moments are in convex order if and only if
\[
\int(x-k)_+ \,\d\mu(x)\leq\int(x-k)_+ \,\d\nu(x)
\]
for all real $k$. This follows from simple approximation arguments (see
\cite{HiRo12} and also Section \ref{pot_func}) using monotone
convergence. In particular, it is sufficient to check (2) for
positive convex functions with finite asymptotic slope in $-\infty$ and
$+\infty$. %This can be seen as a simple application of the monotone
%convergence theorem

We give some examples of measures in convex order.

%ex2.2 #&#
\begin{ex}\label{delta_a}
If $\delta$ is an atom of mass $\alpha>0$ at the point $x$, then
$\delta
\leqc\nu$ simply means that $\nu$ has mass $\alpha$ and barycenter $x$.
%% Actually $x\mapsto1$ is convex and concave and $x\mapsto x$ too.
%The converse is a consequence of Jensen's inequality.
\end{ex}

%ex2.3 #&#
\begin{ex}\label{delta_b}
If $\mu_i\leqc\nu_i$ for $i=1, \ldots, n$ then $\sum_{i=1}^n\mu
_i\leqc
\sum_{i=1}^n\nu_i$.
\end{ex}

%ex2.4 #&#
\begin{ex} \label{ex_order}
If two measures $\mu$ and $\mu'$ have the same barycenter and the same
mass, $\mu$ is concentrated on $[a,b]$ and $\mu'$ is concentrated on
$\R
\setminus\,]a,b[$ then $\mu\leqc\mu'$. Indeed it can be proved for convex
functions $\varphi$ defined on $\R$ that
\[
\int\varphi\,\dd\mu\leq\int\psi\,\dd\mu= \int\psi\,\dd\mu'\leq \int
\varphi\,\dd\mu',
\]
where $\psi$ is the linear function satisfying $\psi=\phi$ in $a$
and $b$.
\end{ex}

%ex2.5 #&#
\begin{ex} \label{ex_order2}
If two measures $\mu$ and $\mu'$ have the same barycenter and the same
mass, $\mu-(\mu\wedge\mu')$ is concentrated on $[a,b]$ and $\mu
'-(\mu
\wedge\mu')$ is concentrated on $\R\setminus\,]a,b[$ then we have
$\mu
\leqc\mu'$. To see this, apply Example~\ref{ex_order} to the two
reduced measures and note that adding $\mu\wedge\mu'$ preserves the order.
\end{ex}

The following result formally states the connection between the convex
order and the existence of martingale transport plans.

%th2.6 #&#
\begin{them}\label{BasicExistence}
Let $\mu, \nu\in\mathcal M$. The condition $\mu\leqc\nu$ is
necessary and sufficient for the existence of a martingale transport
plan in $\M(\mu, \nu)$.
\end{them}

It is a simple consequence of Jensen's inequality that the condition
$\mu\leqc\nu$ is necessary to have $\M(\mu,\nu)\neq\varnothing$:
if $\pi
$ is a martingale transport plan and $\phi$ is convex then
\begin{eqnarray*}
\int\phi(y) \,\d\nu(y)&=& \int\phi(y) \,\d\pi (x,y)
\\
&=&\iint\phi(y) \,\d\pi_x(y) \,\d\mu(x)\geq\int\phi(x) \,\d\mu(x).
\end{eqnarray*}
The fact that the condition is also sufficient is well known and goes
back at least to a paper by Strassen \cite{St65}. Nevertheless, we
think that it is worthwhile to describe a procedure which allows to
obtain a martingale transport plan. This is what we do in the next subsection.

%s2.3 #&#
\subsection{Construction of a martingale transport}\label
{ConsMatPlan} %
%\marginpar{I thing 2.3 could be better written}

We fix finite measures $\mu,\nu$ having finite first moments and
satisfying $\mu\leqc\nu$; our aim is to show that $\M(\mu, \nu)$ is
nonempty. The desired result will first be given in the case where $\mu
$ is concentrated on finitely many points. The construction in
Proposition~\ref{SimpleBasicExistence} will rely on the elementary fact
(related to Example~\ref{delta_b}) that $\pi_1\in\M(\mu_1,\nu_1),
\pi
_2\in\M(\mu_2,\nu_2)$ implies that $\pi_1+\pi_2\in\M(\mu_1+\mu
_2,\nu
_1+\nu_2)$.

%pr2.7 #&#
\begin{pro}\label{SimpleBasicExistence}
Assume that $\mu= \sum_{i=1}^n \delta_i$, where each $\delta_i$ is an
atomic measure. If $\nu$ satisfies $\mu\leqc\nu$, then $\M(\mu,
\nu) $
is nonempty.
\end{pro}

First, note that by Example~\ref{delta_a} this proposition is clear if
$n=1$. The general case will be established by induction. To perform
the inductive step, we need to understand how to couple a single atom,
say $\delta:=\delta_1$, with a properly chosen portion $\nu'$ of
$\nu$
so that the other atoms ($\sum_{i=2}^n\delta_i$) are smaller than
$\nu
-\nu'$ in convex order. Assume that $\delta$ has mass $\alpha$ and is
concentrated on $x$. Recalling Example~\ref{delta_a}, we should pick
$\nu'$ so that it has mass $\alpha$ and barycenter $x$. Clearly, it
also needs to satisfy $\nu'\leq\nu$, where $\leq$ refers to the usual
pointwise order of measures.

As $\delta$ is a part of $\mu$ and $\mu\leqc\nu$, we can introduce the
measure $\tilde\mu=\mu-\delta$ which has mass $t=\nu(\R)-\alpha$.
Obviously, we then have $\delta+\tilde\mu\leqc\nu$. We are looking for
the measure $\nu'$ among the measures $\{\nu_s\dvtx s\in[0,t]\}$ obtained
as the restriction of $\nu$ between two quantiles $s$ and $s'=s+\alpha
$. More precisely, we consider $\nu_s=G_\#\lambda_{[s,s+\alpha]}$ where
$G\dvtx [0,t+\alpha]\to\R$ is the quantile function of $\nu$, and
$\lambda
_{[s,s']}$ is the Lebesgue measure restricted to $[s,s']$. In Section
\ref{summary}, we have discussed quantile functions only for
probability measures but of course the notion carries over to measures
in $\mathcal{M}$. %Recall that quantile functions have been introduced
%in
%Section \ref{summary} for probability measures. The definition also
%holds for measures in $\m$ and may take infinite values at the
%boundary of the domain of $G$.
For completeness, note that $\nu=G_\#\lambda_{[0,t+\alpha]}$.

The barycenter $B(s,\nu)$ of $\nu_s$ depends continuously on the
parameter $s\in[0,t]$ and we claim that
%
%e5 #&#
\begin{equation}
\label{SmallerLarger}B(0, \nu)\leq x,\qquad B(t, \nu) \geq x.
\end{equation}
This is a consequence of the convex order relation $(\delta+\tilde\mu
)\leqc\nu$ applied to the convex and nonnegative functions $u\mapsto
(u-G(\alpha))_-$ and $u\mapsto(u-G(t))_+$. For instance,
\begin{eqnarray*}
\int u-G(t) \,\dd\delta(u)&\leq&\int\bigl(u-G(t)\bigr)_+ \,\dd\delta(u)\leq\int
\bigl(u-G(t)\bigr)_+ \,\dd\nu(u)\\
&=& \int u-G(t) \,\dd\nu_t(u).
\end{eqnarray*}
By the intermediate value theorem, the continuity of $s\mapsto B(s,\nu
)$ implies that there exists some $s\in[0,t]$ such that $\nu_s$ has
barycenter $x$. Moreover, if $B(s,\nu)=B(s',\nu)$, the measures $\nu_s$
and $\nu_{s'}$ are equal so that there exists a unique measure with
barycenter $x$. We denote it by $\nu'$.

This discussion leads us to the following lemma.

%le2.8 #&#
\begin{lem}%\label{BasicExistence}
\label{AtomShadowSplitting}
Let $\mu$ be of the form $\mu= \tilde\mu+\delta$, where $\delta$
is an
atom and assume that $\mu\leqc\nu$. Then there exists a unique
splitting of the measure $\nu$ into two positive measures $\nu'$ and
$\tilde\nu=\nu-\nu'$ in such a way that:
\begin{longlist}[(1)]
\item[(1)]$\delta\leq c\nu'$,
\item[(2)]$\tilde\nu(I)=0$ where $I=\interval{\nu'}$ is the interior of the
smallest interval containing the support of $\nu'$.
\end{longlist}
Moreover, the measures $\tilde\mu$ and $\tilde\nu$ satisfy $\tilde
\mu
\leqc\tilde\nu$.
\end{lem}

\begin{pf}
Having already constructed $\nu'$ (and $I$, i.e., $]G(s),G(s+\alpha
)[$) in the paragraph above Lemma~\ref{AtomShadowSplitting} it remains
to show (2): $\tilde\mu$ is smaller than $\tilde\nu$ in the convex order.
Let $\varphi$ be a nonnegative convex function which satisfies
\[
\limsup_{|x|\to+\infty}\bigl|\varphi(x)/x\bigr|<+\infty.
\]
We will prove that $\int\varphi\,\dd\tilde\mu\leq\int\varphi\,\dd
\tilde\nu$. To this end, we introduce a new function $\psi$ which
equals $\varphi$ on $\R\setminus I$ and is linear on $I$. The function
$\psi$ can be chosen to be convex and satisfy $\psi\geq\varphi$. (Note
that this is possible also in the case where $I$ is unbounded.) The
functions $\varphi$ and $\psi$ coincide on the border of $I$. We have
\[
\int\varphi\,\dd\tilde\mu\leq\int\psi\,\dd\tilde\mu= \int\psi \,\dd \mu- \int\psi\,\dd
\delta.
\]
But as $\psi$ is linear on $I$, one has $\int\psi\,\d\delta=\int
\psi
\,\d\nu'$ and because $\mu\leqc\nu$ one has $\int\psi\,\d\mu\leq
\int
\psi\,\dd\nu$. It follows that
\[
\int\varphi\,\d\tilde\mu\leq\int\psi\,\d\nu- \int\psi\,\d\nu' =\int\psi
\,\d\tilde\nu=\int\varphi\,\d\tilde\nu.
\]
The last equality is due to the fact that $\tilde{\nu}$ is concentrated
on $\R\setminus I$.
We have thus established our claim that $\tilde\mu\leqc\tilde\nu$.
\end{pf}
\begin{pf*}{Proof of Proposition~\ref{SimpleBasicExistence}}
In the first step, we apply Lemma~\ref{AtomShadowSplitting} to the
measures $\delta= \delta_1$ and $\tilde\mu= \sum_{i=2}^n \delta
_i$ to
obtain a splitting $\nu=\hat\nu_1+ \tilde\nu$ that satisfies
$\delta
_1\leqc\hat\nu_1$ and $\tilde\mu\leqc\tilde\nu$.
Trivially, $\M(\delta_1, \hat\nu_1)$ consists of a single element
$\pi_1$.

In the next step, we repeat the procedure with $\tilde\mu$ and
$\tilde
\nu$ in the place of $\mu, \nu$ and continue until the $n$th step where
$\delta_n$ can be martingale-transported to the remaining part of $\nu$
because the convex order relation $\delta_n\leqc(\nu-\sum_{i=1}^{n-1}\hat{\nu}_i)$ is satisfied in Example~\ref{delta_a}. Hence,
we have obtained recursively a sequence $(\hat\nu_i)_{i=1}^n$ such that
$\delta_i\leqc\hat\nu_i$ and $\hat\nu_1+\cdots+\hat\nu_n= \nu
$. We
have constructed $n$ martingale transport plans $\pi_1, \ldots, \pi_n$
where $\pi_i$ is the unique element of $\Pi_M(\delta_i,\hat\nu_i)$.
Thus, $\pi_1+ \cdots+ \pi_n$ is an element of $\M(\mu, \nu)$.
\end{pf*}

To extend Proposition~\ref{SimpleBasicExistence} to the case of general
$\mu\in\mathcal M$, we need the following simple and straightforward
fact that will also be useful in Section~\ref{sec:curtain}.

%le2.9 #&#
\begin{lem}[(Approximation of a measure in the convex order)]\label
{approx_measure}
Assume $\gamma\in\mathcal{M}$. There exists a sequence $(\gamma
^{(n)})_n$ of
finitely supported measures such that $\gamma^{(n+1)}\geqc\gamma
^{(n)}$, the sequence $(\gamma^{(n)})_n$ converges weakly to $\gamma$
in $\mathcal{M}$ and $\gamma^{(n)}\leqc\gamma$ holds for every $n$.
\end{lem}
\begin{pf}
To any partition $\mathcal{J}$ of $\R$ into finitely many intervals,
we can associate some $\gamma_\mathcal{J}$ smaller than $\gamma$ in the
convex order. We simply replace $\gamma=\sum_{I\in\mathcal
{J}}\gamma
|_I$ by $\gamma_\mathcal{J}=\sum_{\mathcal{J}}\delta_I$ where
$\delta
_I$ is an atom with the same mass and same barycenter as $\gamma|_I$.
Note that if $\mathcal{J}'$ is finer than $\mathcal{J}$ (the intervals
of $\mathcal{J}$ are broken in subintervals) then $\gamma_{\mathcal
{J}}\leqc\gamma_{\mathcal{J}'}$.
For $k,N\in\N$, we consider the partition %\marginpar{I think this
%proof may have too many details compared to what we demand from the
%reader in the rest of the article.}
%
\[
\mathcal{J}_{k,N}= \Biggl(\bigcup_{i=-2^kN}^{(2^k-1)N}
\biggl]\frac
{i}{2^k},\frac{i+1}{2^k} \biggr] \Biggr)\,\cup\,]N,+\infty[\, \cup\,]{-}\infty,-N],
\]
and set $\gamma_{k,N}=\gamma_{\mathcal{J}_{k,N}}$. We have $\gamma
_{k,N}\leqc\gamma_{k+1,N}$ and $\gamma_{k,N}\leqc\gamma_{k,N+1}$.
Write $\gamma^{(n)}$ for $\gamma_{n,n}$. Let $f$ be a continuous
function that grows less than linearly in $\pm\infty$. There exist
$a,b>0$ such that $|f(x)|\leq a|x|+b$. Let $\varepsilon>0$ and $N$ be
such that $\int_{|x|\geq N} a|x|+b \,\d\gamma(x)\leq\eps/3$. The
function $f$ is uniformly continuous on $[-N,N]$. Thus, there exists
$\omega$ such that if $x,y\in[-N,N]$ and $|x-y|\leq\omega$ we have
$|f(x)-f(y)|\leq\eps/3$. Let $k$ be such that $1/2^k\leq\omega$. For
$n\geq\max\{k,N\}$, we have
\begin{eqnarray*}
 \bigl|\gamma(f)-\gamma^{(n)}(f)\bigr|&\leq& \biggl\llvert \int
_{-N}^N f \,\d \gamma - \int_{-N}^N
f \,\d\gamma^{(n)} \biggr\rrvert
\\
&&{}+\biggl\llvert \int_{|x|\geq N}f \,\d\gamma\biggr\rrvert +
\biggl\llvert \int_{|x|\geq
N}f \,\d\gamma^{(n)}\biggr\rrvert
\leq\frac{\eps}3+\frac{\eps}3+\frac
{\eps}3.
\end{eqnarray*}
The first two estimates are a consequence of our preparations: To see
this, note that
\[
\biggl\llvert \int_{|x|\geq N}f \,\d\gamma^{(n)}\biggr
\rrvert \leq\int_{|x|\geq
N}a|x|+b \,\d\gamma^{(n)}\leq\int
_{|x|\geq N}a|x|+b \,\d\gamma,
\]
where the convexity of $x\mapsto a|x|+b$ and $\gamma^{(n)}|_{\{|x|\geq
N\}}\leqc\gamma|_{\{|x|\geq N\}}$ are used.
\end{pf}

We are now finally in the position to complete the proof of Theorem~\ref
{BasicExistence}.
\begin{pf*}{Proof of {sufficiency} in Theorem~\ref{BasicExistence}}
Pick a sequence of finitely supported measures $(\mu_n)_{n\geq1}$
satisfying $\mu_n \leqc\nu$ such that $\mu_n$ converges to $\mu$
weakly. (By Lemma~\ref{approx_measure}, the sequence could be chosen to
be increasing in the convex order, but we do not need this here.)
We have already solved the problem of transporting a discrete
distribution. Pick martingale measures $(\pi_n)_{n\geq1}$ which
transport $\mu_n $ to $\nu$ for each $n$. To be able to pass to a
limit, we note that the
set
\[
\Omega:=\M(\mu, \nu)\cup\bigcup_{n=1}^\infty
\M(\mu_n, \nu)
\]
is compact. Hence, the sequence $(\pi_n)_{n\geq1}$ has an accumulation
point $\pi$ in $\Omega$ and of course $\pi$ is as desired: Its
marginals are $\mu$ and $\nu$ and it is a martingale transport plan.
\end{pf*}

We have thus seen a self-contained proof to Theorem~\ref
{BasicExistence}. Of course, the reader may object that the martingale
established in the course of the proof was in no sense canonical and
that the derivation was not constructive since we have invoked a
compactness argument to prove the existence in the case of a general
measure $\mu$. In Section~\ref{sec:curtain}, we will be concerned with a modification
of the above ideas which does not suffer from these shortfalls.

%s2.4 #&#
\subsection{A dual problem}\label{dual_pro}
We mention that the martingale transport problem \eqref{PrimalMart}
admits a dual formulation. In analogy to the dual part of the optimal
transport problem, one may consider %\marginpar{Mathias rewrite!}
\[
\mbox{Maximize} \quad \int\phi\,\d\mu+ \int\psi\,\d\nu,%\mbox{for }
%\begin{array}{l}
%\phi\in L^1(\mu), \psi\in L^1(\nu), \mbox{s.t.\ }
%\exists\Delta\in{\mathcal C}_b(\R),\\
%\phi(x)+\psi(y)+\Delta(x) (y-x) \leq c(x,y ),
%\end{array}
\]
where one maximizes over all functions $\phi\in L^1(\mu), \psi\in
L^1(\nu)$ such that there exists $ \Delta\in{\mathcal C}_b(\R)$ satisfying
%
%e6 #&#
\begin{equation}
\label{DualCond} c(x,y)\geq\phi(x)+\psi(y)+\Delta(x) (y-x)
\end{equation}
for all $x,y\in\R$.
Denote the corresponding supremal value by $D$.
The inequality $D\leq C_M(\mu, \nu)$ then follows by integrating
\eqref
{DualCond} against $\pi\in\M(\mu, \nu)$. In the case of lower
semicontinuous costs $c$, the duality relation $D= C_M(\mu, \nu)$ is
established in \cite{BeHePe11}, Theorem~1.1. We also note that the dual
part of the problem appears naturally in mathematical finance where it
has a canonical interpretation in terms of \emph{replication}. We refer
to \cite{BeHePe11} for more details on this topic.

Duality results for a continuous time martingale transport problem are
obtained by Galichon, Henry-Labordere, Touzi \cite{GaHeTo11} and
Dolinsky and Soner \cite{DoSo12}.
%s3 #&#
\section{A short proof of the variational lemma}\label{sec:gamma}
The aim of this section is to establish the \emph{variational lemma},
Lemma~\ref{GlobalLocal}.
That is, for a given optimal martingale transport plan $\pi$ we want to
construct a Borel set $\Gamma$, $\pi(\Gamma)=1$ such that the
following holds:
if $\alpha$ is a measure on $\R\times\R$ with $|\operatorname
{spt}(\alpha)|<\infty$
and $\operatorname{spt}(\alpha) \subseteq\Gamma$ then we have $\int
c \,\d\alpha
\leq\int c \,\d\alpha'$ for every competitor $\alpha'$ of $\alpha$.

As mentioned above, this result can be viewed as a substitute for the
characterization of optimality through the notion of $c$-cyclical
monotonicity in the classical setup. Under mild regularity assumptions,
it is not too hard to show that a transport plan $\pi$ which is optimal
for the (usual) transport problem is $c$-cyclically monotone; we refer
to \cite{Vi09}, Theorem~5.10. %\marginpar{I need to check whether this
%Exercise is the correct direction of the implication.}
However, this approach does not translate effortlessly to the
martingale case. Roughly speaking, the main problem in the present
setup is that the martingale condition makes manipulation of transport
plans a relatively delicate issue.

Instead, we give here a proof of Lemma~\ref{GlobalLocal} that is based
on certain measure theoretic tools: It requires a general duality
theorem of Kellerer (\cite{Ke84}, Lemma~1.8(a), Corollary~2.18), which in
turn requires Choquet's capacability theorem \cite{Ch59}.\hskip.2pt\footnote{This
approach is inspired by \cite{BeGoMaSc08} where $c$-cyclical
monotonicity is linked to optimality with the help of Kellerer's
result.} See the \hyperref[app]{Appendix} for an alternative and constructive proof of
the variational lemma.

%
%In the classical theory of optimal transport it is fairly straight
%forward to prove that optimality implies cyclical monotonicity. The
%basic idea is that nonoptimality which is detected on a pointwise
%level can be lifted to the global setup using continuity of the cost
%function (see \cite[Exercise 08.15]{Vi03}.)

%A different approach is to accumulate counterexamples on the point
%wise level using measure theoretic tools. This second approach has the
%advantage that it does not rely on the continuity of the cost function
%but the disadvantage that it requires a rather abstract duality
%theorem of Kellerer \cite[Theorem~08.15]{Ke84} which in turn is based
%on Choquet's capacability theorem [Ch1956].
%
%It turns out that a direct as well as an abstract, measure-theoretic
%\marginpar{\small Perhaps we should tell this story below when we
%actually prove the theorem and not here. However I think that it makes
%sense to tell this story at some point.} approach can be used to
%tackle Lemma~\ref{GlobalLocal}. %However, we find that (opposed to the
%classical setting) the elementary approach
%In contrast to the classical setup we encountered an unexpected amount
%of technical obstacles in the elementary case. We have thus decided to
%present the direct approach in the appendix.

The crucial ingredient is the following result.

%th3.1 #&#
\begin{them}\label{KellDich}
Let $(Z, \zeta)$ be a Polish probability space and $M\subseteq Z^n$.
Then either of the following holds true:
\begin{longlist}[(1)]
\item[(1)] There exist subsets $(M_i)_i$ of $Z^n$ such that $\zeta(\proj^i
M_i)=0$ for $i=1,\ldots,n$ and
\[
M\subseteq\bigcup_{i=1}^n
M_i.
\]
\item[(2)] There exists a measure $\gamma$ on $Z^n$ such that $\gamma(M)>0$
and $\proj^i_{\#}\gamma\leq\zeta$ for $i=1,\ldots, n$.
\end{longlist}
\end{them}

We refer to \cite{BeGoMaSc08}, Proposition~2.1, for a detailed proof of
Theorem~\ref{KellDich} from Kellerer's result. % \cite[Theorem~08.15]{Ke84}.

\begin{pf*}{Proof of Lemma~\ref{GlobalLocal}}
Fix a number $n\in\N$. We want to construct a set $\Gamma_n$ for which
the optimality property holds for all $\alpha$ satisfying
$|\operatorname{spt}\alpha
|\leq n$. This set $\Gamma_n$ will satisfy $\pi(\Gamma_n)=1$. Clearly,
$\Gamma=\bigcap_{n\in\N}\Gamma_n$ is then as required to establish
the lemma.

For a fixed $n\in\N$, define a Borel set $M$ by
\[
M:= \lleft\{(x_i, y_i)_{i=1}^n
\dvtx \exists\alpha\mbox{ s.t. } %
\begin{array} {ll} \mbox{(1) $\alpha$
is a measure on $\R\times\R$, }
\\
\mbox{(2) $\operatorname{spt}\alpha\subseteq\bigl\{(x_i,y_i)
\dvtx i=1,\ldots , n\bigr\}$, and }
\\
\mbox{(3) $\exists$ competitor $\alpha'$ satisfying $\int c \,\d
\alpha' < \int c \,\d\alpha$} \end{array} %
 \rright
\}. %
\]
We then apply Theorem~\ref{KellDich} to the space $(Z, \zeta)=(\R
^2,\pi
)$ and the set $M$.%\marginpar{Do we mean $\R\times\R$?}

If we are in case (1), let $N$ be $\bigcup_{i=1}^n \proj^i(M_i)$ so
that $\pi(N)=0$ and
$M\subseteq(N\times Z^{n-1})\cup\cdots\cup(Z^{n-1}\times
N)=Z^n\setminus{(Z\setminus N)^n}$. We can then simply define $\Gamma
_n:= Z\setminus N=\R^2 \setminus N$ to obtain a set which does not
support any nonoptimal $\alpha$ with $|\operatorname{spt}\alpha|\leq
n$. Moreover,
$\pi(\Gamma_n)=1$ as we want, hence the proof is complete.

It remains to show that case (2) cannot occur. Striving for a
contradiction, we assume that there is a measure $\gamma$ such that
$\gamma(M)>0$ and $\proj^i_\#\gamma\leq\pi$ for $i=1,\ldots, n$.
Restricting $\gamma$ to $M$, we may of course assume that $\gamma(\R
\times\R\setminus M)=0$. Rescaling $\gamma$ if necessary, we may also
assume that $\proj^i_\#\gamma\leq\frac{1}n\pi$.

Consider the measure $\omega=\sum_{i=1}^n \proj^i_\#\gamma$ on $\R^2$.
It is smaller than $\pi$ and has positive mass. In particular $\mu
_\omega= \proj^x_\# \omega\leq\mu$. We will find a competitor
$\omega
'$ (recall Definition~\ref{def:competitor}) such that $\omega'$ leads
to smaller costs than $\omega$, that is,
\[
\int c(x,y) \,\d\omega'<\int c(x,y) \,\d\omega.
\]
If such a measure $\omega'$ exists then the measure $\pi-\omega
+\omega
'$ is a martingale transport plan which leads to smaller costs than
$\pi
$, contradicting the optimality of $\pi$. It remains to explain how
$\omega'$ is obtained. For each $p= ((x_1, y_1), \ldots,\break (x_n,
y_n) )\in(\R\times\R)^n$, let $\alpha_p$ be the measure which is
uniformly distributed on the set $\{(x_1, y_1), \ldots,  (x_n, y_n)\}$. Then
\[
\omega= \int_{p \in(\R\times\R)^n} \alpha_p \,\d\gamma(p).
\]
%
%where $\alpha_p$ has mass $1$ and is concentrated on $n$ points of $X
%\times Y$, the coordinates of $p$.
%For each $p=\big((x_1, y_1), \ldots, (x_n, y_n)\big)$ let $X_0=\{x_1,
%\ldots, x_n\}$, $Y_0=\{y_1, \ldots, y_n\}$, $\mu_0= \sum_{i=1}^n{
%\delta_{x_i}}, \nu_0= \sum_{i=1}^n{\delta_{y_i}}$.
For each $p\in(\R\times\R)^n$, let $\alpha'_p$ be an optimizer of
the problem
\[
\mbox{Minimize}\quad \int_{(x,y)\in\R\times\R} c(x,y) \,\d\beta(x,y)\qquad \mbox {$
\beta$ competitor of $\alpha_p$.}
\]
%
%subject to $\beta$ being a competitor of $\alpha_p$.
We emphasize that $\alpha'_p$ exists and can be taken to depend
measurably on $p$. This follows, for instance, by calculating $\alpha'_p$
using the simplex algorithm.\footnote{It is well known that the optimal
transport problem for finite spaces falls into the realm of linear
programming; see, for instance, \cite{Vi03}, page~23. The same holds true
in the martingale case.}

As $\gamma$ is concentrated on $M$, for $\gamma$-almost all points $p$
the measure $\alpha'_p$ satisfies
\[
\int_{(x,y)\in\R\times\R} c(x,y) \,\d\alpha'_p(x,y)<
\int_{(x,y)\in\R
\times\R} c(x,y) \,\d\alpha_p(x,y).
\]
(Note that $\alpha'_p$ is in general not concentrated on the same set
as $\alpha_p$.) Then $\omega'$ defined by
\[
\omega'=\int_{p \in(\R\times\R)^n} \alpha'_p
\,\d\gamma(p)
\]
satisfies the above conditions as required. For instance, we have
\begin{eqnarray*}
\int_{\R\times\R} c \,\d\omega' &=&\int
_{p \in
(\R
\times\R)^n} \int_{(x,y)\in\R\times\R} c(x,y) \,\d
\alpha'_p(x,y) \,\d \gamma(p)
\\
& <& \int_{p \in(\R\times\R)^n} \int_{(x,y)\in\R\times\R} c(x,y) \,\d
\alpha_p(x,y) \,\d\gamma(p)= \int_{\R\times\R}c \,\d\omega.
\end{eqnarray*}
The other properties are checked analogously.
\end{pf*}
We note that the just given proof of Lemma~\ref{GlobalLocal} is likely
to extend to more general setups. In particular, we expect that the
result remains valid if martingale transport plans between higher
dimensional spaces and with a finite number of time steps [i.e.,
$(X_i)_{i=1}^n$ rather then just $X_1=X$ and $X_2= Y$] are considered.

%We note that the just given proof of Lemma~\ref{GlobalLocal} also
%works in more general setups: the result remains valid if martingale
%transport plans between higher dimensional spaces are considered. In a
%different direction, on may state versions which appeal to more than
%two prescribed marginals, i.e.\ martingales $(X_i)_{i=1}^n$ with more
%than just two steps $X_1=X$ and $X_2= Y$.

Subsequently,
Lemma~\ref{GlobalLocal} will several times be applied in conjunction
with the following technical assertion. %Intuitively it asserts that if
%certain structures appear in $\Gamma$...
Given $\Gamma\subseteq\R^2$ we will use the notation $\Gamma_x$ for
$\{y\in\R\dvtx (x,y)\in\Gamma\}$.

% Write $\Gamma_x$ for $\{y: (x,y)\in\Gamma\}$. Denote by $A$ the set
%of continuity points of $\mu$. If there are most countably many $a\in
%A$ such that $\Gamma_a$ then we may remove those and hence assume that
%there are none at all. On the other hand, if there are uncountably
%many such points, than the corresponding transports necessarily
%accumulate in the following sense:

%le3.2 #&#
\begin{lem}\label{AccumulatingBadPoints}
Let $k$ be a positive integer and $\Gamma\subseteq\R^2$. Assume also
that there are uncountably many $a\in\R$ satisfying $|\Gamma_a|\geq k$.

There exist $a$ and $b_1< \cdots< b_k\in\Gamma_a$ such that for every
$\eps>0$ one may find $a'>a$ and $b_1'< \cdots< b_k'\in\Gamma_{a'}$ with
\[
\max\bigl(\bigl|a-a'\bigr|, \bigl|b_1-b'_1\bigr|,
\ldots, \bigl|b_k-b'_k\bigr|\bigr)< \eps.
\]
Moreover, one may also find $ a''<a$ and $b_1''< \cdots< b_k''\in
\Gamma_{a''}$ with
\[
\max\bigl(\bigl|a-a''\bigr|,\bigl |b_1-b''_1\bigr|,
\ldots, \bigl|b_k-b''_k\bigr|\bigr)<
\eps.
\]
\end{lem}

\begin{pf}
Write $A$ for the set of all $a $ such that $|\Gamma_a|\geq k$ and pick
for each $a\in A$ distinct elements $b_1^a, \ldots, b_k^a\in\Gamma_a$.
Set $\Gamma_A=\{(a, b_1^a, \ldots, b_k^a)\dvtx a\in A\}$.
We call $(a, b_1^a, \ldots, b_k^a)\in\Gamma_A$ a \emph
{right-accumulation} point if for every $\eps>0$ there exists $a'\in
]a, a+\eps[$ such that $|b_i^{a}- b_i^{a'}| < \eps$ for every $i$. We
call it \emph{right-isolated} otherwise. If $p$ belongs to the set of
right-isolated points $I_r\subset\Gamma_A$, then there exists some
$\eps_p>0$ such that
\[
\bigl[\{p\}+ \bigl(]0,\eps_p[\,\times\,]{-}\eps_p,
\eps_p[^k \bigr) \bigr] \cap\Gamma_A=
\varnothing,
\]
where $+$ refers to the Minkowski sum of sets.

Assume for contradiction that the set $I_r$ is uncountable. Then there
exists some $\zeta>0$ such that $K=\{p\in I_r\dvtx \eps_p>\zeta\}$ is
uncountable. Given $p_1, p_2\in K$, we have $p_2\notin p_1+
((0,\zeta)\times(-\zeta,\zeta)^k )$. Since $p_1$ and $p_2 $ have
different first coordinates, this implies
\[
\bigl[\{p_1\} + \bigl(]0,\zeta/2[\,\times\,]{-}\zeta/2,\zeta/2
[^k \bigr) \bigr] \cap \bigl[\{p_2\} + \bigl(]0,
\zeta/2[\,\times\,]{-}\zeta/2,\zeta /2[^k \bigr) \bigr]= \varnothing.
\]
This is a contradiction since there cannot be uncountably many disjoint
open sets in $\R^{k+1}$.

It follows that all but countably many elements of $A$ are
right-accumulation points. Arguing the same way with left replacing
right we obtain the desired conclusion.
\end{pf}

%s4 #&#
\section{Existence of a monotone martingale transport plan: The
left-curtain transport plan} \label{sec:curtain}
A short way to prove that there exists some monotone martingale
transport plan would be to take a minimizer of problem \eqref
{PrimalMart} for $c(x,y)=h(y-x)$ where $h$ is chosen appropriately.
Then one may apply Lemma~\ref{GlobalLocal} to prove that this minimizer
is monotone. This kind of argument will be encountered in Sections~\ref
{sec:optimal} and \ref{sec:general} below. Here, however, we find it
useful to give a construction which yields more insight in the
structure of the martingale transport plan. In particular, it will also
allow us to prove the uniqueness of a monotone martingale transport
plan in Section~\ref{sec:unique} and it will not require any
assumptions on $\mu$ and $\nu$.

For our argument, we reconsider the construction used in Proposition~\ref{SimpleBasicExistence} and decide to transport the atoms $\delta_i$
of $\mu=\sum_{i}\delta_i$ to $\nu$ in a particular order, starting with
the left-most atom and continuing to the right. It turns out that one
can characterize the martingale coupling that we obtain in terms of an
\emph{extended convex order} and \emph{shadow} introduced below (see
Definition~\ref{defi_ext} and Lemma~\ref{MinimalImage}). These notions
enable us to adapt the construction directly to the continuous case,
thus making the approximation procedure used in Section \ref
{ConsMatPlan} obsolete.

%s4.1 #&#
\subsection{Potential functions}\label{pot_func}

An important tool in this section will be the so-called \emph{potential
functions}. For each $\mu\in\mathcal{M}$, we define the potential
function $u_\mu
\dvtx \R\to\R$ by
\[
u_\mu(x)=\int_{-\infty}^\infty|y-x| \,\d\mu(y)
\]
for $x\in\R$. Set $k=\mu(\R)$ and $m=\frac{1}{k}\int x \,\d\mu$.

%pr4.1 #&#
\begin{pro}\label{meas2pot}
If $\mu$ is in $\mathcal{M}$ and $k=\mu(\R), m=\frac{1}{k}\int x
\,\d\mu$, then
$u_\mu$ has the following properties:
\begin{longlist}[(ii)]
\item[(i)]$u_\mu$ is convex,
%for all $x\in\R$ we have $u_\mu(x)\geq k |x-m|$% it is a consequence
%of the 2 others
%
\item[(ii)]$\lim_{x\to-\infty} u_\mu(x)-k|x-m|=0$ and $\lim_{x\to
+\infty}
u_\mu(x)-k|x-m|=0$.
\end{longlist}
Conversely, if $f$ is a function satisfying these properties for some
numbers $m\in\R$ and $k\in[0,+\infty[$, then there exists a unique
measure $\mu\in\mathcal{M}$ such that $f=u_\mu$. The measure $\mu$
is one-half
the second derivative $f''$ in the sense of distributions.
\end{pro}

\begin{pf}
See, for instance, the proof of Proposition~2.1 in \cite{HiRo12}.
%%CALLS ARE A LITTLE DIFFERENT. LETS LEAVE IT ! %The three properties
%are easy to see, thus we will only give the argument for the converse
%statement.
%Without loss of generality we assume that $m=0$ and that $k=1$.
%Set $$C(x):=\frac{g(x)+x}2.$$ To motivate this definition we note that
%in case that $g=u_\mu$ this amounts to
%$$ C(x)= 1/2 \int|y-x| - (y-x) \,\d\mu(y)=\int(y-x)_+\,\d\mu(y).$$
%We further let $F_\eps(x)= [C(x-\eps)-C(x)]/\eps.$ By the assumed
%properties the limit $F=\lim_{\eps\to0} F_\eps$ exists and is a
%distribution function. The corresponding measure $\mu$ satisfies $u_
%\mu=g$.
\end{pf}

Let us list some relevant properties of potential functions.

%pr4.2 #&#
\begin{pro} \label{convergence}
Let $\mu$ and $\nu$ be in $\mathcal{M}$.
\begin{itemize}
\item If $\mu$ and $\nu$ have the same mass, $\mu\leqc\nu$ is
equivalent to $ u_\mu\leq u_\nu$.
\item We have $ \mu\leq\nu$ if and only if $u_\mu$ has smaller
curvature than $u_\nu$. More precisely, $\mu\leq\nu$ if and only if
$u_\nu-u_\mu$ is convex.
\item A sequence of measures $(\mu_n)_n$ in $\mathcal{M}$ with mass
$k$ and mean
$m$ converges weakly in $\mathcal{M}$ to some $\mu$ if and only if
$(u_{\mu
_n})_n$ converges pointwise to the potential function of some $\mu'\in
\mathcal{M}
$. In that case, $\mu=\mu'$.
\end{itemize}
\end{pro}

\begin{pf}
For the first property, see \cite{HiPrRoYo11}, Exercise 1.7, for the
third \cite{HiRo12}, Proposition~2.3. The second property is a
consequence Proposition~\ref{meas2pot}. Namely, $2\mu$ and $2\nu$ are
the second derivatives of $u_\mu$ and $u_\nu$.
\end{pf}
%
%Si les $\mu$ sont de masse 1 et de barycentre 0 alors $\mu'$ aussi. En
%effet le potentiel u de $\mu'$ doit {\chr"C3}{\chr"AA}tre $1$
%Lipschitzien donc sa pente {\chr"C3}{\chr"A0} l'infinie est \leq1.
%C'est aussi \geq1. Donc c'est 1. Par cons{\chr"C3}{\chr"A9}quent c'est
%forc{\chr"C3}{\chr"A9}ment centr{\chr"C3}{\chr"A9} comme les mu_n.

We will need the following generalization of the convex order.

%de4.3 #&#
\begin{defi}[(Extended convex order on $\mathcal{M}$)] \label{defi_ext}
Let $\mu$ and $\nu$ be measures in~$\mathcal{M}$. We write $\mu
\leqe\nu$ and
say that $\nu$ is greater than $\mu$ in the \emph{extended convex
order} if for any \emph{nonnegative} convex function $\varphi\dvtx\R\to
\R$
we have
\[
\int\varphi\,\d\mu\leq\int\varphi\,\d\nu.
\]
\end{defi}

The partial order $\leqc$ on $\mathcal{M}$ is extended by the order
$\leqe$ in
the sense that $\leqe$ keeps the old relations and gives rise to new
ones. By definition, if $\mu\leqc\nu$ then we have $\mu\leqe\nu$
(since nonnegative convex functions are convex). But if $\mu\leq\nu$,
we will also have $\mu\leqe\nu$ (as nonnegative convex functions are
nonnegative). Note that in this second case the two measures may have
neither the same mass nor the same barycenter.

As $x\mapsto1$ is a convex function, a trivial consequence of $\mu
\leqe\nu$ is $\mu(\R)\leq\nu(\R)$. More precisely, let us prove that
if the two measures have the same mass, $\mu\leqe\nu$ is equivalent to
$\mu\leqc\nu$. Indeed if $\mu\leqe\nu$, for a convex function
$\varphi
\dvtx \R\to\R$ and any (negative) constant $y$, the convex function
$\varphi
_y\dvtx x\mapsto\varphi(x)\vee y$ satisfies $\int\varphi_y \,\dd\mu\leq
\int
\varphi_y \,\dd\nu$ because $\int\varphi_y-y \,\dd\mu\leq\int
\varphi
_y-y \,\dd\nu$. Letting $y$ go to $-\infty$ we obtain $\int\varphi
\,\dd
\mu\leq\int\varphi\,\dd\nu$. Hence, $\mu\leqc\nu$.

%Let us give now the following simple characterization of the extended
%convex order in terms of $\leqc$:
In terms of $\leqc$, the extend convex order can be characterized as follows.

%pr4.4 #&#
\begin{pro}\label{ExtendedOrderChar}
Assume that $\mu\leqe\nu$. Then there exists a measure $\theta\leq
\nu
$ such that $\mu\leqc\theta$.
\end{pro}

Of course, the converse statement is true as well: If there exists
$\theta$ such that $\mu\leqc\theta$ and $\theta\leq\nu$, then we have
also $\mu\leqe\nu$.
\begin{pf*}{Proof of Proposition \ref{ExtendedOrderChar}}
Let $\mu$ and $\nu$ satisfy $\mu\leqe\nu$. We can assume that $\nu
$ is
a probability measure and denote by $k$ and $m$ the mass,
respectively, the mean of $\mu$. We define a measure $\theta\leq\nu$
as follows. Consider the quantile function $G_\nu$ of $\nu$. Recall
that $\lambda$ is the Lebesgue measure on $\R$. For a parameter
$\zeta
\in[0,k]$, we denote by $\lambda^\zeta$ the restriction of $\lambda
$ to
$[0,1]\setminus[\zeta,\zeta+(1-k)]$. This measure has mass $k$ as well
as does $\theta=(G_\nu)_\#\lambda^\zeta$. We now pick $\zeta$ such that
$\theta$ has mean $m$. To see that this can be done, we will apply the
intermediate value theorem in the same fashion as in the discussion
preceding Lemma~\ref{AtomShadowSplitting}:
To see that $m$ is indeed an intermediate value between the means of
$\theta$ obtained for $\zeta=0$ and $\zeta=k$, we consider the
nonnegative and convex functions $x\mapsto(x-G_\nu(1-k))_+$ and
$x\mapsto(G_\nu(k)-x)_+$ and integrate them against $\mu$ and $\nu$ in
the same way as we did above to obtain the inequalities in~\eqref
{SmallerLarger}. Clearly, the mean of $\theta$ depends continuously on
$\zeta$, and hence the intermediate value theorem yields the existence
of the desired $\zeta$.

% and hence the intermediate value theorem is applicable. Actually one
%can consider the nonnegative and convex functions $x\mapsto(x-G_
%\nu(1-k))_+$ and $x\mapsto(G_\nu(k)-x)_+$ and integrate them against $
%\mu$ and $\nu$ in a similar way as we did for the inequalities in
%\eqref{SmallerLarger}. It follows from the convex order relation that
%$m$ is indeed an intermediate value between the means of $\theta$ for $
%\zeta=0$ and $\zeta=k$.

We are now given two measures $\mu$ and $\theta$ of the same mass and
the same mean. Consider a convex function $\varphi$. We want to prove
that its integral with respect to $\mu$ is smaller than the one with
respect to $\theta$. For that, we can assume without loss of generality
$\varphi(G_\nu(\zeta))=\varphi(G_\nu(\zeta+(1-k)))=0$.
Then
\begin{eqnarray*}
\int\varphi\,\dd\mu(x)&\leq&\int\varphi_+(x) \,\dd\mu(x)
\\
&\leq& \int \varphi_+(x) \,\dd\nu(x)=\int\varphi_+(x) \,\dd\theta(x)=\int \varphi(x)
\,\dd\theta(x).
\end{eqnarray*}
This completes the proof.
\end{pf*}

%s4.2 #&#
\subsection{Maximal and minimal elements}
For $\mu\leqe\nu$, let $F_\mu^\nu$ be the set of measures $\eta$ such
that $\mu\leqc\eta$ and $\eta\leq\nu$. Note that the measures in
$F_\mu
^\nu$ have the same mass and the same barycenter as $\mu$. In the next
lemmas, we consider the partially ordered set $(F_\mu^\nu,\leqc)$ and
show that it has both a maximal and a minimal element.

%le4.5 #&#
\begin{lem}\label{MaximalImage}
For $\mu\leqe\nu$, the set $F_\mu^\nu$ has an element which is maximal
w.r.t. the convex order, that is, there exists $T^\nu(\mu)$ such that:
\begin{longlist}[(iii)]
\item[(i)]$T^\nu(\mu)\leq\nu$.
\item[(ii)]$\mu\leqc T^\nu(\mu)$.
\item[(iii)] If $\eta$ is another measure satisfying \textup{(i)} and \textup{(ii)} then we have
$\eta\leqc T^\nu(\mu)$.
\end{longlist}
\end{lem}

\begin{pf}
Consider the measure $\theta$ defined as in the proof of Proposition~\ref{ExtendedOrderChar} and let $\eta$ be another measure in $F_\mu
^\nu
$. We know that $\theta$ is concentrated outside an open interval $I$
and that it coincides with $\nu$ on ${\R\setminus\bar{I}}$ so that
$\theta|_{\R\setminus\bar{I}}\geq\eta|_{\R\setminus\bar{I}}$. Thus,
$\eta-(\eta\wedge\theta)$ is concentrated on $\bar I$ whereas
$\theta
-(\eta\wedge\theta)$ is concentrated on $\overline{\R\setminus
I}$. It
follows from Example~\ref{ex_order2} that $\eta\leqc\theta$.
\end{pf}

The existence of a minimal element is more involved and will play an
important role subsequently.

%le4.6 #&#
\begin{lem}[(Shadow embedding)]\label{MinimalImage}
Let $\mu,\nu\in\mathcal{M}$ and assume $\mu\leqe\nu$. Then there
exists a
measure $S^\nu(\mu)$, called the \emph{shadow} of $\mu$ in $\nu$,
such that:
\begin{longlist}[(iii)]
\item[(i)]$S^\nu(\mu)\leq\nu$.
\item[(ii)]$\mu\leqc S^\nu(\mu)$.
\item[(iii)] If $\eta$ is another measure satisfying \textup{(i)} and \textup{(ii)}, then we
have $S^\nu(\mu)\leqc\eta$.
\end{longlist}
As a consequence of \textup{(iii)}, the measure $S^\nu(\mu)$ is uniquely
determined. Moreover, it satisfies the following property:
\begin{longlist}[(iii$'$)]
\item[(iii$'$)] If $\eta$ is a measure such that $\eta\leq\nu$
and $\mu\leqe\eta$, then we have $S^\nu(\mu)\leqe\eta$.
\end{longlist}
\end{lem}

Note that if $\mu\leqc\nu$, that is, if $\mu$ and $\nu$ have the same
mass, then the shadow $S^\nu(\mu)$ is just $\nu$ itself because this is
the only measure $\eta$ with mass $\mu(\R)=\nu(\R)$ that satisfies
$\eta
\leq\nu$.
\begin{pf*}{Proof of Lemma~\ref{MinimalImage}}
First observe that (iii$'$) follows from Proposition~\ref
{ExtendedOrderChar} applied to $\mu$ and $\eta$.

We write $k$ (resp., $m$) for the mass (resp., the mean) of $\mu$.
The principal strategy of our proof is to rewrite the problem in terms
of potential functions. Set $f=u_\mu$ and $g=u_\nu$.

The task is to find a convex function $h$ (corresponding to $u_{S^\nu
(\mu)}$) such that:
\begin{longlist}[(1)]
\item[(1)]$h-g$ is concave, that is, $h''\leq g''$ in a weak sense.
\item[(2)]$f\leq h$ and $\lim_{|x|\to\infty} h(x)-k |x-m|=0$.
\item[(3)] We have $h\leq h_2$ for all functions $h_2$ in the set
\[
U_F=\bigl\{\mbox{$h$ is convex and satisfies (1) and (2)}\bigr\}=
\{h=u_\eta \dvtx \eta \in F\}.
\]
%
%(see just below for the definition of $U_F$).
\end{longlist}
We note that by
Proposition~\ref{ExtendedOrderChar} there exist functions satisfying
conditions (1) and~(2). Hence, the sets $F=\{\eta\dvtx \mu\leqc\eta,
\eta\leq
\nu\}$ and $U_F$ % $$U_F=\{\mbox{$h$ is convex and satisfies (1) and
%(2)}\}=\{h=u_\eta:\eta\in F\}$$
are not empty. Looking for a function which also satisfies the third
property we define
%
%e7 #&#
\begin{equation}
\label{bar} \bar{h}=\inf_{h\in U_F}h.
\end{equation}
If this function is convex, which we shall show below, it will satisfy
the three required conditions. Conditions (2) and (3) are clear; let us
briefly prove (1): Every function $h\in U_F$ is ``less convex'' than
$g$, that is, the function $h-g$ is concave. Hence, $\bar{h}-g=(\inf_{h\in U_F} h)-g=\inf_{h\in U_F} (h-g)$ is also concave.

The convexity of $\bar{h}$ will be proved if we can establish that its
epigraph $\mathcal{E}(\bar{h})$ is convex, that is, that every segment
of $\R^2$ with both ends in $\mathcal{E}(\bar{h})$ is included in this
set. This will be the case if $U_F$ is stable under the following
operation: take $h_1,h_2$ in $U_F$ and let $h_{\min}$ be the convex
hull of $x\mapsto\min(h_1(x),h_2(x))$. %Indeed notice that the convex
%hull of $\mathcal{E}(h_1)\cup\mathcal{E}(h_2)$ is $\mathcal{E}(h_{
%\min})$.
More precisely,
\[
h_{\min}(x)=\inf_{ab\geq0, (a,b)\neq(0,0)}\frac{bh_1(x-a)+ah_2(x+b)}{a+b}.
\]
Since $\lim_{|c|\to\infty} (h_1-h_2)(x+c)=0$, this infimum is in fact
a minimum. Condition~(2) holds for $h_{\min}$. It remains to prove that
$h_{\min}-g$ is concave.

We use a nonusual but clear characterization of concavity: A real
function is concave if and only if it has locally an upper tangent in
every point. More precisely, $f$ is concave if for every $x\in\R$,
there exists an affine function $l$ with $l(x)=f(x)$ and $l\geq f$ in a
neighborhood of $x$. With respect to the definition of $h_{\min}$,
there are two kinds of real $x$. A point $x$ such that $h_{\min}(x)$
equals $h_i(x)$ for some $i\in\{1,2\}$ is of the first kind. In this
case, the property is true because $h_i\geq h_{\min}$ so that
$h_i-g\geq h_{\min}-g$ where the first function is concave. These
relations even hold globally. In the other case, there exist $a, b$
with $ab>0$ such that $h_{\min}(x)=\frac{bh_1(x-a)+ah_2(x+b)}{a+b}$.
Without loss of generality, we may assume $a>0$ and $b>0$. As $h_{\min
}$ both is convex and its graph is
\textit{below} the cord $[(x-a,h_1(x-a)),(x+b,h_2(x+b))]$ we can
conclude that it is affine on $[x-a,x+b]$. Hence, $h_{\min}-g$ is
concave in a neighborhood of~$x$. Summing up, the property holds for
the two kinds of real~$x$. Finally, $h_{\min}-g$ is concave and
$h_{\min
}\in U_F$. Hence, $\bar{h}$ is convex and satisfies conditions (1)--(3).
\end{pf*}

%\marginpar{It may be not clear for the reader that the hypothesis of
%Lemma~2.8 are the same as $\mu\leqe\nu$.}
Note that in Lemma~\ref{AtomShadowSplitting} we have implicitly
encountered the shadow in the case where the starting distribution
consists of an atom.
%
%ex4.7 #&#
\begin{ex}[(Shadow of an atom)]\label{one_atom}
Let $\delta$ be an atom of mass $\alpha$ at a point $x$. Assume that
$\delta\leqe\nu$.
Then $S^\nu(\delta)$ is the restriction of $\nu$ between two quantiles,
that is, it is $\nu'=(G_\nu)_\#\lambda_{[s,s']}$ where $s'-s=\alpha$
and the barycenter of $\nu'$ is $x$. Indeed, for another measure $\eta
\in\mathcal{M}$ with $\delta\leqc\eta$ and $\eta\leq\nu$,
applying the
observation from Example~\ref{ex_order2} to $\nu'$ and $\eta$ we obtain
$\nu'\leqc\eta$.
\end{ex}

%s4.3 #&#
\subsection{Associativity of shadows}
In this section, we will establish the following associativity property
of the shadow.

%th4.8 #&#
\begin{them}[(Shadow of a sum)]\label{two_measures}
Let $\gamma_1, \gamma_2$ and $ \nu$ be elements of $\mathcal{M}$
and assume that
$\mu= \gamma_1+\gamma_2 \leqe\nu$.
Then we have $\gamma_2\leqe\nu-S^\nu(\gamma_1)$ and
\[
S^\nu(\gamma_1+\gamma_2)=S^\nu(
\gamma_1)+S^{\nu-S^\nu(\gamma
_1)}(\gamma_2).
\]
\end{them}

In Figure~\ref{fig:shadow}, we can see the shadow of $\mu=\gamma
_1+\gamma_2$ in $\nu$ for two different ways of labeling the $\gamma
_i$'s. In both cases, $\nu_1:=S^\nu(\gamma_1)$ is simply $\gamma
_1$. On
the left part of the figure $S^{\nu-\nu_1}(\gamma_2)$ is quite
intuitive while on the right part it is deduced from the associativity
of the shadow projection. Of course, it has to be $S^\nu(\mu)-\nu_1$.

%% Creator: Inkscape 0.48.0, www.inkscape.org
%% Accompanies image file 'shadow.pdf' (pdf, eps, ps)
%% \input{<filename>.pdf_tex}
%% instead of
%% \includegraphics{<filename>.pdf}
%% \def\svgwidth{<desired width>}
%% \input{<filename>.pdf_tex}
%% instead of
%% \includegraphics[width=<desired width>]{<filename>.pdf}
%% \usepackage{import}
%% \import{<path to file>}{<filename>.pdf_tex}
%% \graphicspath{{<path to file>/}}
%% http://tug.ctan.org/tex-archive/info/svg-inkscape

%f3 #&#
\begin{figure}

\includegraphics{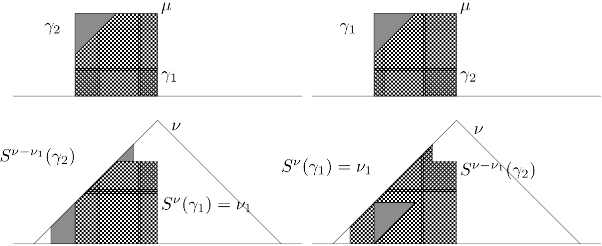}

\caption{Shadow of $\mu=\gamma_1+\gamma_2$ in $\nu$.}\label{fig:shadow}
\end{figure}

Our proof of Theorem~\ref{two_measures} will rely on approximations of
$\mu$ by atomic measures and we need several auxiliary results. In our
argument, we will require a certain continuity property of the mapping
$\nu\mapsto S^\nu(\delta)$ stated in Lemma~\ref{atom_shadow
continuity}. We will derive it now with the help of the Kantorovich metric.

%pr4.9 #&#
\begin{pro}[(Metric on $\mathcal{M}$)]
The function $W$ defined on $\mathcal{M}$ by
%
%e8 #&#
\begin{equation}
\label{kanto} W(\nu,\hat\nu)= \cases{ +\infty,&\quad $\mbox{if }\nu(\R)\neq\hat\nu(
\R)$,\vspace*{2pt}
\cr
\displaystyle\sup_{f} \biggl(\int f \,\d\nu-\int f \,\d
\hat\nu \biggr),&\quad $\mbox{otherwise,}$}
\end{equation}
where the supremum is taken over all $1$-Lipschitz functions $f\dvtx \R\to
\R
$ is a metric with values in $[0,+\infty]$. For $k>0$, the associated
topology on the subspaces of measure of mass $k$ coincides with the
weak topology introduced in Section
\ref{ParBasicNotions}.
\end{pro}

In the case where $\nu$, $\hat\nu$ are probability measures, $W(\nu
,\hat
\nu)$ is the classical Kantorovich metric (also called $1$-Wasserstein
distance, or transport distance). We state here two useful relations
that are well known (and straightforward) in the case of probability
measures and extended to finite measures through normalization. If $\nu
(\R)=\hat{\nu}(\R)$, we have
\[
W(\nu,\hat{\nu})=\|F_\nu-F_{\hat\nu}\|_{1}=
\|G_\nu-G_{\hat\nu
}\|_1,
\]
where $F_\nu$, $F_{\hat{\nu}}$ and $G_\nu$, $G_{\hat{\nu}}$ are the
cumulative distribution functions and the quantile functions of $\nu$
and $\hat{\nu}$, respectively.
The norm $\|\cdot\|_1$ refers to the $L^1$-norm for the Lebesgue measure
on $\R$, respectively, $[0,\nu(\R)]$. Recall that $\nu=(G_\nu)_\#
\lambda$ and $\hat{\nu}=(G_{\hat{\nu}})_\#\lambda$.

Let us now fix some notation in preparation to Lemma~\ref{atom_shadow
continuity}. First, let $\nu$ and $\hat{\nu}$ be of mass 1. We also fix
a quantity $\alpha\leq1$ and set $t= 1-\alpha$. As in the discussion
preceding Lemma~\ref{AtomShadowSplitting}, we consider for $s\in[0,t]$
the restriction $\nu_s=(G_\nu)_\#\lambda_{[s,s+\alpha]}$ of $\nu$
between the quantiles $s$ and $s+\alpha$. We adopt the same convention
for $\hat\nu$. Note that the barycenter of $\nu_s$ can be written
%
%e9 #&#
\begin{equation}
\label{bary} B(s,\nu)=\frac{1}\alpha\int_\R x
\,\d\nu_s(t) \quad\mbox{or}\quad B(s,\nu )=\frac{1}\alpha\int
_0^{\alpha}G_\nu(s+t) \,\d\lambda(t).
\end{equation}
Indeed, the function $t\in[0,\alpha]\mapsto G_\nu(s+t)$ is simply
$G_{\nu_s}$ and $\nu_s=(G_{\nu_s})_\#\lambda_{[0,\alpha]}$.

Together with \eqref{kanto} applied to the functions $f\dvtx x\mapsto\pm
x$, the first formula for the barycenter implies
\[
\bigl|B(s,\nu)-B(s,\hat\nu)\bigr|\leq\frac{1}{\alpha}W(\nu_s,\hat
\nu_s).
\]
Moreover, we can prove that
\[
W(\nu_{r},\nu_s)=\alpha\bigl|B(r,\nu)-B(s,\nu)\bigr|
\]
without difficulty by using $W(\nu_r,\nu_s)=\|G_{\nu_r}-G_{\nu_s}\|_1$
and the fact that $G_{\nu_s}$ and $G_{\nu_r}$ are equal to the
nondecreasing function $G_\nu$ up to translation.
Another simple property is
\[
W(\nu_s,\hat\nu_s)\leq W(\nu,\hat\nu).
\]
Again this can be seen as a consequence of the representation of $W$ by
quantile functions: We have $W(\nu_s,\hat{\nu}_s)=\|G_{\nu
_s}-G_{\hat
{\nu}_s}\|_1\leq\|G_{\nu}-G_{\hat{\nu}}\|_1$.

Let $x$ be an element of $\R$ and consider the subset of measures $\nu
\in\mathcal{P}$ such that $B(0,\nu)\leq x\leq B(t,\nu)$.
These are exactly the measures such that there exists $s\in\R$
satisfying $B(s,\nu)=x$; for such $\nu$ the shadow $S^\nu(\delta
)=\nu
_{s}$ is well defined.

%Let $x$ be an element of $\R$ and consider the subset of measures $\nu
%\in\p$ such that $B(0,\nu)\leq x\leq B(t,\nu)$. These are exactly the
%measures such that there exists $s_x\in\R$ satisfying $B(s_x,\nu)=x$.
%We denote by $S_{x,l}(\nu)$ the measure $R_{s_x}(\nu)$. Note first of
%all that the map $S_{x,l}$ is well defined. In fact if two
%restrictions of $\nu$ for the quantiles $s$ and $s'$ and mass $m$ have
%the same barycenter, we have necessarily $G_\nu(s^+)=G_\nu(s'^+)$ and
%$G_\nu((s+l)^-)=G_\nu((s'+l)^-)$, which can only happen if $\nu$ has
%atoms (of mass greater than $|s'-s|$) in those values of $G_\nu$. In
%this case the restrictions $R_s(\nu)$ and $R_{s'}(\nu)$ coincide and
%we denote this measure by $S_{x,l}(\nu)$. Later in this part it will
%be minimal shadow of $l\cdot\delta_x$ in $\nu$ and denoted by $S^\nu(l
%\cdot\delta_x)$.

%le4.10 #&#
\begin{lem}\label{atom_shadow continuity} Let $\delta= \alpha\delta
_x$ be an atom of mass $\alpha< 1$.
The map $\nu\mapsto S^\nu(\delta)$ is continuous on its domain of
definition inside the probability measures.
\end{lem}

\begin{pf}
%Consider $\nu\in\mathcal M$ with barycenter $x$ and $s_x$ such that
%%$S_{x,m}(\nu)=R_{s_x}(\nu)$
%$ S^\nu( \delta)=\nu_{s_x}$.
Let $\nu, \hat\nu$ be probability measures in $\mathcal{M}$ and
assume that $
S^{\nu}( \delta),\break  S^{\hat\nu}( \delta)$ exist. Let $r,s$ be such that
$\nu_r=S^{\nu}( \delta)$ and $\hat\nu_s=S^{\hat\nu}( \delta)$. Of
course, both measures have the same barycenter. Then
\begin{eqnarray*}
W \bigl(S^\nu(\delta),S^{\hat\nu}(\delta) \bigr)&=& W(
\nu_{r},\hat \nu_{s})
\\
&\leq& W(\nu_{r}, \nu_{s})+W( \nu_{s},\hat
\nu_{s})
\\
&=& \alpha\bigl|B(r,\nu)-B(s,\nu)\bigr|+W( \nu_{s},\hat\nu_{s})
\\
&=& \alpha\bigl|B(s,\hat{\nu})-B(s,\nu)\bigr|+W( \nu_{s},\hat
\nu_{s})
\\
&\leq& W(\nu_s,\hat\nu_s)+W( \nu_{s},\hat
\nu_{s})\leq2W(\nu,\hat \nu).%
%\qedhere
\end{eqnarray*}
%
%Since the last quantity tends to $0$ as $n$ goes to infinity, $S^\nu(
%\delta)$ is continuous in $\nu$.
\upqed\end{pf}

%le4.11 #&#
\begin{lem}\label{monoton_lem}
Let $\delta$ be an atom and assume $\delta\leqe\eta$, where $\eta
\leq
\nu$. Then we have
\[
\eta-S^\eta(\delta)\leq\nu-S^\nu(\delta).
\]
\end{lem}
\begin{pf}
First note that $S^\eta(\delta)\leq\eta\leq\nu$. Hence, $\delta
\leqe
\nu$ and $S^\nu(\delta)$ is well defined. As explained in Example~\ref
{one_atom}, there exists an interval $Q\subseteq[0,\nu(\R)]$ such that
$S^\nu(\delta)$ equals ${G_\nu}_\#\lambda_Q$. The same is true for
$\delta, \eta, G_\eta$ and some interval of $[0,\eta(\R)]$ but we will
represent the ``quantile coordinates'' of $S^\eta(\delta)$ under
$\eta$
in a slightly different way. Indeed, $S^\eta(\delta)$ is the
restriction of $\eta$ to a real interval plus possibly some atomic
parts of $\eta$ at the ends of this interval. In any case, it is
smaller than $\eta$ and $\nu$. Thus, we can parameterize it with a
subinterval $Q'$ of $[0,\nu(\R)]$ such that $S^\eta(\delta)=({G_\nu
}_\#
\lambda_{Q'})\wedge\eta$. Note that the length of $Q'$ is greater than
the length of $Q$ which equals the mass of $\delta$. The measures
$S^\nu
(\delta)$ and $S^\eta(\delta)$ have the same mass and the same
barycenter and both are smaller than $\nu$.

We prove by contradiction that $Q\subseteq Q'$. By symmetry, it is
enough to prove $b'\geq b$ where we denote $Q$ and $Q'$ by $[a,b]$ and
$[a',b']$, respectively. If it were not the case, $S^\eta(\delta)$
would be stochastically strictly smaller than $S^\nu(\delta)$, which is
the right-most measure that stays smaller than quantile $b$, has the
same mass as $\delta$ and is smaller than $\nu$. In particular, the
barycenters would be strictly ordered (see the discussion before Lemma~\ref{AtomShadowSplitting} for a similar and more detailed argument).
This is a contradiction since the barycenters coincide by the
definition of the shadow. Finally,
\[
\eta-S^\eta(\delta)= \eta-\bigl[({G_\nu}_\#
\lambda_{Q'})\wedge\eta\bigr] \leq\nu -{G_\nu}_\#
\lambda_{Q'}\leq\nu-{G_\nu}_\#\lambda_{Q}.
\]
Here, we used the fact that for three measures $\alpha, \beta, \gamma$
satisfying the relations $\alpha\leq\gamma$ and $\beta\leq\gamma$,
the measure $\gamma-\alpha$ is greater than the positive part of
$\beta
-\alpha$, which is $\beta-(\alpha\wedge\beta)$.
\end{pf}

%le4.12 #&#
\begin{lem}[(Shadow of one atom and one measure)]\label{atom_plus_measure}
Consider now $\delta+\gamma$ where $\delta$ is an atom. Assume
$(\delta
+\gamma)\leqe\nu$. Then we have $\gamma\leqe\nu-S^\nu(\delta)$ and
%
%e10 #&#
\begin{equation}
\label{decompo1} S^\nu(\delta+\gamma)=S^\nu(
\delta)+S^{\nu-S^\nu(\delta)}(\gamma).
\end{equation}
\end{lem}

\begin{pf}
We first prove that $\gamma$ is smaller than $\nu':=\nu-S^\nu
(\delta)$
in the extended order. Note that there exists an interval $I$ such that
$S^\nu(\delta)$ is concentrated on $\bar I$ and $\nu'(I)=0$. Let
$\varphi$ be a nonnegative convex function which satisfies $\limsup_{|x|\to+\infty}|\varphi(x)/x|<+\infty$. We will prove $\int
\varphi
\,\dd\gamma\leq\int\varphi\,\dd\nu'$. For that, we introduce $\psi$
which equals $\varphi$ on $\R\setminus I$ and is linear on $I$. We can
assume that $\psi$ is convex and $\psi\geq\varphi$ (even if $I$ is
unbounded). Note that $\varphi$ and $\psi$ coincide on the border of
$I$. We have
\[
\int\varphi\,\dd\gamma\leq\int\psi\,\dd\gamma\leq\int\psi\,\dd \nu - \int\psi\,\dd
\delta.
\]
But $\int\psi\,\dd\delta=\int\psi\,\d S^\nu(\delta)$ because $\psi$
is linear on $I$. % and this quantity is greater than $\int\varphi\,\d
%S^\nu(\delta)$.
Moreover, $\int\psi\,\d\nu'=\int\varphi\,\d\nu'$ because $\nu'$ is
concentrated on $\R\setminus I$. It follows that
\[
\int\varphi\,\d\gamma\leq\int\psi\,\d\nu- \int\psi\,\d\delta\leq \int\varphi\,\d
\nu'.
\]
As in the case of the usual convex order, it is of course sufficient to
test against convex functions of linear growth, hence $\gamma\leqe\nu'$.

It remains to establish (\ref{decompo1}). It is clear (see, e.g., Example~\ref{delta_b}) that both sides of the equation are
greater than $\delta+\gamma$ in the convex order and $\leq\!\nu$. Hence,
by the definition of the shadow it follows $S^\nu(\delta+\gamma
)\leqc
S^\nu(\delta)+S^{\nu-S^\nu(\delta)}(\gamma)$. The other
inequality is
shown as follows: we will prove that for $\eta\geqc\delta+\gamma$ and
satisfying $\eta\leq\nu$ we have $S^\nu(\delta)+S^{\nu-S^\nu
(\delta
)}(\gamma)\leqc\eta$. In fact, if $\eta\geqc\delta+\gamma$ then
$S^\eta(\delta)\leq\eta$ and $S^{\eta-S^\eta(\delta)}(\gamma
)\leq\eta
-S^\eta(\delta)$ so that, since measures in the convex order have the
same mass,
\[
\eta=S^\eta(\delta)+S^{\eta-S^\eta(\delta)}(\gamma).
\]
(Note that we have already proved that all terms exist in this
decomposition since $\geqe$ extends $\geqc$.)
%because the convex order is a special case of the extended order.)
But it follows from $\eta\leq\nu$ and $\eta-S^\eta(\delta)\leq
\nu
-S^\nu(\delta)$ (proved in Lemma~\ref{monoton_lem}) that $F_\gamma
^\eta
\subset F_\gamma^\nu$ and $F_\gamma^{\eta-S^\eta(\delta
)}\subseteq
F_\gamma^{\nu-S^\nu(\delta)}$ so that $S^\eta(\delta)\geqc S^\nu
(\delta
)$ and $S^{\eta-S^\eta(\delta)}(\gamma)\geqc S^{\nu-S^\nu(\delta
)}(\gamma)$. As in Example~\ref{delta_b}, the compatibility of sum and
convex order completes the proof.
\end{pf}

%le4.13 #&#
\begin{lem}[(Shadow of finitely many atoms)]\label{many_atoms}
Let $(\delta_i)_{i\in\N}$ be a family of atoms at point $x_i$ and of
mass $\alpha_i\in[0,+\infty[$ (where we allow the weight $\alpha_i$ to
be~$0$). For every $n\geq1$, let $\mu_n=\delta_1+\cdots+\delta_n$ and
assume that $\mu_n\leqe\nu$. The sequence $(\nu_n)_{n\in\N}$ defined
by $\nu_n=S^\nu(\mu_n)$ satisfies the following recurrence relation:
\begin{itemize}
\item$\nu_0=0$,
\item$\nu_{n}=\nu_{n-1}+S^{\nu-\nu_{n-1}}(\delta_n)$ for every
$n\geq1$.
\end{itemize}
\end{lem}

\begin{pf}
The lemma is proved by induction. The basis holds with $\nu_{1}=S^\nu
(\delta_1)$. Fix $n\geq1$ and assume that the recurrence relation
holds until $n$. Let $(\mu_i)_i$, $\nu$ and $(\nu_i)_i$ be as in the
statement of the lemma. %So the relation $\nu_n=\nu_{n-1}+S^{\nu-
%\nu_{n-1}}(\delta_n)$ is satisfied.
Denote $\sum_{i=2}^{n+1}\delta_i$ by $\mu'_n$ and more generally
$\sum_{i=2}^{i+1}\delta_i$ by $\mu'_i$. As $\mu_{n+1}\leqe\nu$, we can
apply Lemma~\ref{atom_plus_measure} to the decomposition $\mu
_{n+1}=\delta_1+\mu'_n$. So $\mu'_n\leqe\nu-\nu_1$ and
%
%e11 #&#
\begin{equation}
\label{induc} S^\nu(\mu_{n+1})=S^\nu(
\delta_1)+S^{\nu'}\bigl(\mu'_n
\bigr),
\end{equation}
where we denoted $\nu-\nu_1$ by $\nu'$. But because of the inductive
hypothesis applied to $\mu'_n$ and $\nu'$, the shadow $S^{\nu'}(\mu
'_n)$ is $\nu'_n=\nu'_{n-1}+S^{\nu'-\nu'_{n-1}}(\delta_{n+1})$ where
the measures $\nu'_i$ denote the shadows of $\mu'_i$ in $\nu'$. Note
also that $\nu_n=\nu_1+\nu'_{n-1}$ by Lemma~\ref{atom_plus_measure}.
Starting from (\ref{induc}), we now have
\[
\nu_{n+1}=\nu_1+\nu'_n=
\nu_1+\nu'_{n-1}+S^{\nu'-\nu
'_{n-1}}(\delta
_{n+1})=\nu_n+S^{\nu'-\nu'_{n-1}}(\delta_{n+1}).
\]
But ${\nu'-\nu'_{n-1}}=(\nu_1+\nu')-(\nu_1+\nu'_{n-1})={\nu-\nu_n}$.
This completes the proof.
\end{pf}

%re4.14 #&#
\begin{rem}%[Associativity]
An important consequence of the lemma above is that $\nu_n-\nu_k$ is
the shadow of $\mu_n-\mu_k$ in $\nu-S^\nu(\mu_k)$. Even though the
above construction is of inductive nature, when permuting the $n$ first
atoms, the measure $\nu_n=\sum_{i=1}^n \nu_i-\nu_{i-1}$ is always the
same: it simply equals $S^\nu(\mu_n)$. The same assertions apply to
Proposition~\ref{many_measures} below.
\end{rem}

%pr4.15 #&#
\begin{pro}\label{limit_measure}
Assume that $(\mu_n)_n$ is increasing in the convex order and $\mu
_n\leqc\mu\leqe\nu$ for every $n\in\N$. Then both $(\mu_n)_n$ and
$(S^\nu(\mu_n))_n$ converge in $\mathcal{M}$. If we call $\mu
_\infty$,
respectively, $S_\infty$ the limits, then the measure $S_\infty$ is the
shadow of $\mu_\infty$ in $\nu$.
\end{pro}

\begin{pf}
%Notice that we can assume $\mu_n\leqc\nu$. Actually for measures $
%\mu$ greater than $\mu_0$ in the convex order $\mu\leqe\nu$ is
%equivalent to $\mu\leqc T^\nu(\mu_0)$ where $T^\nu(\mu_0)$ is defined
%in Lemma~\ref{MaximalImage}. So we can assume $\nu=T^\nu(\mu_0)$ and
%write $\mu_n\leqc\nu$ instead of $\mu_n\leqe\nu$.
First note that the assumptions imply $u_{\mu_0}\leq u_{\mu_1}\leq
\cdots\leq u_{\mu_n}$ and $u_{\mu_n}\leq u_\mu$. The limit
$u_\infty
:=\lim_{n\in\N} u_{\mu_n}$ exists because for every $x\in\R$,
$(u_{\mu
_n}(x))_n$ is increasing and bounded from above. Of course, the limit
$u_\infty$ is a convex function and since $u_\mu$ is an upper bound it
has the correct asymptotic behavior. Therefore, $u_\infty$ is a
potential function and by Proposition~\ref{meas2pot} it is the
potential function of some $\mu_\infty\in\mathcal{M}$ with the same
mass and
mean as $\mu$ and the $\mu_n$'s.

On the other hand, for $n\in\N$ we consider the set $F_{\mu_n}^\nu$ of
measures $\eta_n$ satisfying $\mu_n\leqc\eta_n$ and $\eta_n\leq
\nu$.
(We are using the notation of the proof of Lemma~\ref{MinimalImage}.)
The measure $S^\nu(\mu_n)$ is the smallest element of $F_{\mu_n}^\nu$
with respect to the convex order. The family $F_{\mu_n}^\nu$ is
decreasing in $n$ and it is not difficult to see that $F^\nu_\mu
\subseteq\cap F^\nu_{\mu_n}$ so that it is not empty. Hence
$S^\nu
(\mu_n)$ is increasing in the convex order and it is bounded from above
by $S^\nu(\mu)$. Exactly for the same reasons as for the sequence
$(\mu
_n)_n$, it converges to some $S_\infty$ in $\mathcal{M}$. We now have to
conclude that $S^\nu(\mu_\infty)=S_\infty$. We will in fact prove that
$S_\infty\leqc S^\nu(\mu_\infty)$ and $S^\nu(\mu_\infty)\leqc
S_\infty$.

For every $n$, we have $\mu_n\leqc\mu_\infty\leqc S^\nu(\mu
_\infty)$
and $S^\nu(\mu_\infty)\leq\nu$. Thus, $S^\nu(\mu_n)\leqc S^\nu
(\mu
_\infty)$. By Proposition~\ref{convergence}, we have $S_\infty\leqc
S^\nu(\mu_\infty)$. Conversely, using again Proposition~\ref
{convergence}, the relation $\mu_n\leqc S^\nu(\mu_n)$ yields $\mu
_\infty
\leqc S_\infty$ as $n$ goes to $+\infty$. But $S_\infty\leq\nu$ [the
limit of a converging sequence $(u_\nu-u_{S^\nu(\mu_n)})_n$ is convex].
Hence, $S^\nu(\mu_\infty)\leqc S_\infty$.
\end{pf}

%le4.16 #&
\begin{lem}[(Shadow of one measure and one atom)]\label{measure_plus_atom}
Consider now $\gamma+\delta$ where $\delta$ is an atom. Assume
$(\gamma
+\delta)\leqe\nu$. Then we have $\delta\leqe S^\nu(\gamma+\delta
)-S^\nu
(\gamma)$ and
%
%e12 #&#
\begin{equation}
\label{decompo} S^\nu(\gamma+\delta)=S^\nu(
\gamma)+S^{\nu-S^\nu(\gamma)}(\delta).
\end{equation}
\end{lem}
\begin{pf}
If $\gamma$ is the sum of finitely many atoms, the result follows from
Lemma~\ref{many_atoms}. Let us consider an approximating sequence
$(\gamma^{(n)})_n$ of $\gamma$ as in Lemma~\ref{approx_measure}. We can
write the decomposition of the shadow of $\gamma^{(n)}+\delta$ in
$\nu$
as in the statement of the lemma and apply Proposition~\ref
{limit_measure} to the sequence $(S^\nu(\gamma^{(n)}))_n$. It follows
that the limit exists and equals $S^\nu(\gamma)$. Write $\nu^{(n)}$ for
$S^\nu(\gamma^{(n)})$ and $\nu^{(\infty)}$ for $S^\nu(\gamma)$.
For the
same reasons as above, the shadows of $\gamma^{(n)}+\delta$ converge to
$S^\nu(\gamma+\delta)$.

We still have to show that $S^{\nu-\nu^{(n)}}(\delta)$ converges to
$S^{\nu-\nu^{(\infty)}}(\delta)$.
%We know that $\nu^{(n)}$ converges to $\nu^{(\infty)}$ in $\m$ so $\nu-
%\nu^{(n)}$ tends to $\nu-\nu^{(\infty)}$ and all these measures are
%bounded by $\nu$ (in particular they have a density smaller than $1$
%with respect to $\nu$). We also know from Example~\ref{one_atom} that
%$S^{\nu-\nu^{(n)}}(\delta)$ is the restriction of $\nu-\nu^{(n)}$ to
%the (uniquely determined) ``quantile interval'' with the correct mass
%and barycenter. Rescaling masses if necessary, the continuity Lemma~\ref{atom_shadow continuity} implies that $S^{\nu-\nu^{(n)}}(\delta)$
%converges to $S^{\nu-\nu^{(\infty)}}(\delta)$.
We know that $\nu^{(n)}$ converges to $\nu^{(\infty)}$ in $\mathcal
{M}$ so $\nu
-\nu^{(n)}$ tends to $\nu-\nu^{(\infty)}$ and all these measures are
bounded by $\nu$. We also know that $S^{\nu-\nu^{(n)}}(\delta)$ is the
restriction of $\nu-\nu^{(n)}$ to the (uniquely determined) ``quantile
interval'' with the correct mass and barycenter. Rescaling masses if
necessary, the continuity Lemma~\ref{atom_shadow continuity} implies
that $S^{\nu-\nu^{(n)}}(\delta)$ converges to $S^{\nu-\nu^{(\infty
)}}(\delta)$.
\end{pf}

We are now finally in the position to prove the desired associativity
property of the shadow mapping.
\begin{pf*}{Proof of Theorem~\ref{two_measures}}
If $\gamma_2$ is the sum of finitely many atoms, the property holds
since by Lemma~\ref{measure_plus_atom} it is possible to construct
recursively $S^\nu(\gamma_1+\gamma_2)$ using a decomposition with one
atom from $\gamma_2$ and the rest of $\gamma_1+\gamma_2$ as the second
measure. Let us consider a sequence $(\gamma^{(n)}_2)_n$ of measures
consisting of finitely many atoms that weakly converge to $\gamma_2$
and satisfy $\gamma^{(n)}_2\leqc\gamma_2$. Moreover, we may assume
that $(\gamma^{(n)}_2)_n$ is increasing in the convex order as in Lemma~\ref{approx_measure}.

We can write the decomposition of the shadow of $\gamma_1+\gamma
_2^{(n)}$ in $\nu$ as in the statement of the theorem and apply
Proposition~\ref{limit_measure} to the sequence $(S^{\nu-S^\nu
(\gamma
_1)}(\gamma^{(n)}_2))_n$. We obtain that the limit exists and equals
$S^{\nu-S^\nu(\gamma_1)}(\gamma_2)$. For the same reasons, the shadow
of $\gamma_1+\gamma^{(n)}_2$
converges to $S^\nu(\gamma_1+\gamma_2)$. This completes the proof.
\end{pf*}

Before we define the left-curtain transport plan, it seems worthwhile
to record the following result.

%pr4.17 #&#
\begin{pro}[(Shadow of the sum of finitely many measures)]\label{many_measures}
Let $(\gamma_i)_i$ be a family of measures (that possibly vanish
identically). Let $\mu_n=\gamma_1+\cdots+\gamma_n$. Assume also that
$\mu_n\leqe\nu$ for every $n\geq1$. The sequence $(\nu_n)_{n\in\N}$
defined by $\nu_n=S^\nu(\mu_n)$ satisfies the following recurrence relation:
\begin{itemize}
\item$\nu_0=0$,
\item$\nu_{n}-\nu_{n-1}=S^{\nu-\nu_{n-1}}(\gamma_n)$.
\end{itemize}
\end{pro}

\begin{pf}
The statement is the same as Lemma~\ref{many_atoms} except that we do
not require the measures $\gamma_i$ to be atoms. Lemma~\ref{many_atoms}
relies on Lemma~\ref{atom_plus_measure} which characterizes the shadow
of $\gamma_1+\gamma_2$ under the assumption that $\gamma_1$ is an atom.
Substituting it with Theorem~\ref{two_measures} the present claim
follows verbatim.
\end{pf}
Let us now formally define the left-curtain coupling $\pi_\lc$ that has
been discussed in the \hyperref[sec:intro]{Introduction} and whose properties will be derived
in the sequel. We baptize it the ``left-curtain transport plan''
because it projects shadow measures as a curtain that one closes
starting from the left-hand side.

Note that given measures $\mu\leq\mu' \leqe\nu$, Theorem~\ref
{two_measures} implies that $S^\nu(\mu)\leq S^\nu(\mu')$. This property
is essential for the definition of $\pi_\lc$.

%Roughly speaking the mass of $\mu$ contained in $]x,x+\,\dd x]$ is
%transported to $S^\nu(\mu_{]{-}\infty,x+\,\dd x]})-S^\nu(\mu_{]-
%\infty,x]})$, that is the shadow of $\mu_{]x,x+\,\dd x]}$ on $\nu-S^\nu(
%\mu_{]{-}\infty,x]})$. Note that both measures have the same barycenter,
%which explains that $\pi_\lc$ is a martingale transport. M: I did not
%find this extremely intuitive. Sorry!

%In the two following theorems and Section~\ref{sec:unique} we will
%consider measures of the type $\proj^y_\#(\pi|_{I\times\R})$ where $I
%\subseteq\R$ and $\pi\in\Pi_M(\mu,\nu)$. Although this expression may
%seem intricate, it simply refers to the second marginal, denoted by $
%\nu^\pi_I$, of the martingale transport plan $\pi|_{I\times\R}$
%obtained by restriction of $\pi$. Hence $\mu_I\leqc\nu^\pi_I$ because
%$\pi|_{I\times\R}\in\Pi_M(\mu_I,\nu^\pi_I)$. Of course $\nu^\pi_I\leq
%\nu$. Recall also that the notation $\nu^\pi_x$ for $\nu^\pi_{]-
%\infty,x]}$ was\marginpar{M: ``were'' $\to$ ``was''!!} already
%introduced in the introduction.

%th4.18 #&#
\begin{them}[(Definition of ${\pi_\lc}$)]\label{lc_defi} Assume that
$\mu
\leqc\nu$. There is a unique probability measure $\pi_\lc$ on $\R
\times
\R$ which transports $\mu|_{]{-}\infty,x]}$ to\break $S^\nu(\mu|_{]{-}\infty
,x]})$, that is, satisfies $\proj^x_\#({\pi_\lc|}_{]{-}\infty
,x]\times\R
})=\mu|_{]{-}\infty,x]}$ and\break $\proj^y_\#({\pi_\lc|}_{]{-}\infty
,x]\times\R
})=S^\nu(\mu|_{]{-}\infty,x]})$ for all $x\in\R$. Moreover, $\pi
_\lc$ is
a martingale transport plan which takes $\mu$ to $\nu$, that is,
${\pi
_\lc}\in\Pi_M(\mu,\nu) $.
\end{them}
\begin{pf}%\marginpar{Uniqueness: $\Pi$-system and classe monotone.
%Existence??: fonction de r{\chr"C3}{\chr"A9}partition multivalu{
%\chr"C3}{\chr"A9}e puis v{\chr"C3}{\chr"A9}rifier que c'est bon pour
%$A$ diff{\chr"C3}{\chr"A9}rent}
Plainly, the condition given in the statement prescribes the value of
\[
\pi_\lc\bigl(]{-}\infty,x]\times A\bigr)= S^\nu(\mu|_{]{-}\infty,x]})
(A)
\]
for $x\in\R$ and every Borel set $A\subseteq\R$, thus giving rise to
a unique measure on the product space. Here we use that, by Theorem~\ref
{two_measures}, $S^\nu(\mu|_{]{-}\infty,x]})\leq S^\nu(\mu
|_{]{-}\infty
,x']})$ whenever $x\leq x'$.

Clearly, the first marginal of $\pi_\lc$ equals $\mu$. By construction,
%%\marginpar{N:``By construction'' is wrong. The second marginal is
%given for $\R\times A$. We need to take a limit $x\to\infty$.}
the second marginal satisfies $\proj^y_\# \pi_\lc\leq\nu$. Since
$\mu
$ and $\nu$ have the same mass, this implies $\proj^y_\# \pi_\lc=
\nu$
as required.

%The second marginal of $\pi$ equals $\lim_{x\to+\infty} \proj^y_\#(
%\pi_{]{-}\infty,x]\times\R})=\sup_{x\in\R} \proj^y_\#(\pi_{]{-}\infty,x]
%\times\R})$. Therefore $\proj^y_\#\pi_\lc=\lim_{x\to+\infty} S^\nu(
%\mu_{]{-}\infty,x]})\leq\nu$. Since the mass of $S^\nu(\mu_{]-
%\infty,x]})$ tends to $1$ we have $\proj^y_\#\pi_\lc=\nu$.

To establish the martingale property, we show that property \eqref
{MartTest} holds for any function $\rho=\I_{]{-}\infty,x']}, x'\in\R$.
Indeed, we have
\begin{eqnarray*}
\int(y-x)\rho(x) \,\dd\pi_\lc(x,y)&=&\int y \,\dd S^\nu(
\mu|_{]{-}\infty
,x']}) (y)-\int x \,\dd\mu|_{]{-}\infty,x']}(x)\\
&=&0.
\end{eqnarray*}
\upqed\end{pf}

%\begin{rem} Below (in Theorems \ref{them:mono} and \ref{mono_unique})
%we will see that $\pi_\lc$ is the unique (left-)monotone martingale
%transport plan in $\Pi_M(\mu, \nu)$.
%Similarly, the unique right-monotone (or ``right-curtain'') martingale
%transport can be defined by requiring that for each $x\in\R$, the
%interval $[x, + \infty[$ is transported to $S^\nu(\mu_{|[x,+\infty[})$.
%
%Of course this procedure to construct martingale transport plans could
%be further generalized, replacing $(]{-}\infty, x])_x$ by appropriate
%families of ordered sets.
%\end{rem}

%re4.19 #&#
\begin{rem}
The family of intervals $(]{-}\infty,x])_{x\in\R}$ is totally ordered
with respect to $\subseteq$ and it spans the $\sigma$-field of Borel
measurable sets. In the proof of Theorem~\ref{lc_defi}, we used these
properties to show that there is a unique martingale transport plan
which transports $\mu|_{]{-}\infty,x]}$ to $S^\nu(\mu|_{]{-}\infty,x]})$.
This construction can be applied to more general families of sets: Let
$I$ be some index set and $(C_\iota)_{\iota\in I}$ a family of Borel
sets that both is totally ordered with respect to $\subseteq$ and spans
the $\sigma$-field of Borel sets. %and satisfies $\bigcup_{k\in\N} C_{
%\iota_k}=\R$ for some sequence of $\iota_k\in I$.
Then a measure $\pi\in\Pi_M(\mu,\nu)$ is defined uniquely by the
relations $\pi(C_\iota\times A)= S^\nu(\mu|_{C_\iota})(A)$ for all
indices $\iota\in I$ and Borel sets $A\subseteq\R$.
\end{rem}

%ex4.20 #&#
\begin{ex}
In the case of a finitely supported measure $\mu=\sum_{i=1}^n \delta
_i$, it follows that if the ordering is done so that the support of
$\delta_i$ is $\{x_i\}$ with $x_1\leq\cdots\leq x_n$, then the $\pi
_\lc
$-coupling is $\pi_\lc=\sum_{i=1}^n \tilde\delta_{i}\otimes S^{\nu
-\nu
_{i-1}}(\delta_i)$ where $\tilde\delta_i=\delta_i/\delta_i(x_i)$ are
the properly renormalized versions of $\delta_i$ and the measures $\nu
_i$ are $S^\nu(\mu_i)$ with $\mu_i=\delta_1+\cdots+\delta_i$ as in
Lemma~\ref{many_atoms}.
\end{ex}

%th4.21 #&#
\begin{them} \label{them:mono}
The martingale $\pi_\lc$ is left-monotone in the sense of Definition~\ref{defi_monoto}.
\end{them}

\begin{pf}
Note that $\pi_\lc$ is simultaneously a minimizer for all cost
functions of the form $c_{s,t}(x,y)= \I_{]{-}\infty,s]}(x)|y-t|$,
where $s,t$ are real numbers. Indeed, if $\pi$ is an arbitrary
martingale transport plan then
\begin{eqnarray*}
\iint c_{s,t}(x,y) \,\d\pi(x,y)&=&\iint_{]{-}\infty,s]\times\R}|y-t| \,\d \pi (x,y)\\
&=&
\int|y-t| \,\d\bigl(\proj^y_\#\pi|_{]{-}\infty,s]\times\R}\bigr) (y).
\end{eqnarray*}
Setting $\nu_s^\pi=\proj^y_\#\pi|_{]{-}\infty,s]\times\R}$ we have
$\nu
_s^\pi\leq\nu$ and $\mu|_{]{-}\infty,s]}\leqc\nu_s^\pi$ which implies
$S^\nu(\mu|_{]{-}\infty,s]})\leqc\nu_s^\pi$. Therefore,
\[
\int|y-t| \,\d S^\nu(\mu|_{]{-}\infty,s]}) (y)\leq\int|y-t| \,\d\nu
_s^\pi(y),
\]
where equality holds for all $s,t \in\R$ if (and only if) $\pi= \pi
_\lc
$. %because $S^\nu(\mu_{]{-}\infty,s]})$ is $\proj^y_\#{\pi_\lc}_{]-
%\infty,s]\times\R}$ for every $s\in\R$.
%But for $\pi_e$ the mass $\proj^y_\#(\pi|_{]{-}\infty,s]\times\R})$ is
%exactly $S^\nu_{\mu_{]{-}\infty,s]}}(y)$ and as $y\mapsto|y-t|$ is
%convex $\pi_e$ is a minimizer for the cost $c_{s,t}$.

Applying Lemma~\ref{GlobalLocal} to the costs $c_{s,t}$ for $s,t\in\Q$,
we obtain a Borel set $\Gamma_{s,t}$ of $\pi_\lc$-measure $1$. Set
$\Gamma=\bigcap_{s,t\in\Q}\Gamma_{s,t}$. We claim that a configuration
as in \eqref{BadConfiguration} cannot appear in $\Gamma$. Indeed, if
$(x,y^-), (x,y^+)$ and $(x',y')$ are in $\Gamma$ and satisfy $x<x'$ and
$y^-<y'<y^+$, they are also in $\Gamma_{s,t}$ where $(s,t)$ satisfies
$s\in\,]x,x'[$ and $t\in\,]y',y^+[$. Let $\lambda\in\,]0,1[$ be such that
$y'=\lambda y^+ + (1-\lambda)y^-$. The measure $\alpha=\lambda\delta
_{(x,y^+)}+(1-\lambda)\delta_{(x,y^-)}+\delta_{(x',y')}$ is
concentrated on $\Gamma$ but the competitor $\alpha'=\lambda\delta
_{(x',y^+)}+(1-\lambda)\delta_{(x',y^-)}+\delta_{(x,y')}$ leads to a
lower global cost. This yields the desired contradiction.
\end{pf}

%s5 #&#
\section{Uniqueness of the monotone martingale transport}\label{sec:unique}

In this section, we establish that the left-curtain coupling $\pi_\lc$
is the unique monotone martingale coupling. Our proof of this result is
specific to the present setup. We will also explain a more classical
argument that is often invoked in the optimal transport theory to
establish some uniqueness property. This so-called \emph{half sum}
argument will be used several times subsequently but requires the
initial distribution $\mu$ to be continuous. %\ref{sec:general}.
%\marginpar{repr\'eciser les articulations de ce cette section.}

We start with two preliminary lemmas which are required to derive the
main result of this part, Theorem~\ref{mono_unique}.

%le5.1 #&
\begin{lem}\label{interv}
If $\mu\leqc\nu$, then one of the following statements holds true:
\begin{itemize}
\item we have $\mu(]a,+\infty[)>0$ and $\nu(]a,+\infty[)>0$ for every
$a$; %and $\nu(]a,+\infty[)>0$
%\item There exists $a\in\R$ such that $\mu(]a,+\infty[)=0, \nu([a,+
%\infty)>0$ and $\nu(\{a\})\geq\mu(\{a\})$. FALSE
%
\item the number $a=\sup(\Spt(\mu))$ is finite and $\nu(]a,+\infty[)>0$;
\item the number $a=\sup(\Spt(\mu))$ is finite and $\nu(]a,+\infty
[)=0$. Moreover,\break $\nu(\{a\})\geq\mu(\{a\})$.
\end{itemize}
The corresponding result for intervals of the form $]{-}\infty,b[$ is
true as well.
\end{lem}

\begin{pf} Integrating the convex function $x\mapsto(x-a')^+$ for
different values of $a'$ we obtain $\sup(\Spt(\mu))\leq\sup(\Spt
(\nu
))$. Therefore, the first case corresponds to $\sup(\Spt(\mu))=\sup
(\Spt
(\nu))=+\infty$, the second to $\sup(\Spt(\mu))<\sup(\Spt(\nu
))$ and
the third to
$\sup(\Spt(\mu))=\sup(\Spt(\nu))<+\infty$.

Let us prove that in the third case we also have $\mu(\{a\})\leq\nu
(\{
a\})$. If $\mu(\{a\})=0$ we are done. If $\mu(\{a\})>0$, the
conditional transport measure $\pi_a$ must be the static transport
because it is a martingale transport plan and $\sup(\Spt(\nu))= a$.
This completes the proof.
\end{pf}
%
%We recall from Section~4 that that if $\mu\leqe\nu$ then $F_\mu^\nu$
%denotes the set of measures $\eta$ such that $\mu\leqc\eta$ and $\eta
%\leq\nu$. As a consequence of Lemma~\ref{interv} we have the
%following:
%\begin{lem}\label{interv2}
%Let $\mu, \nu$ and $\nu_2$ be finite measures and assume that $\mu
%\leqe\nu$. If there exist $\eta\in F^\nu_\mu$ and $d\in\R$ such that $
%\eta$ is concentrated on $]{-}\infty,d]$ and $\nu_2$ is concentrated on
%$[d,+\infty[$ then the shadows $S^{\nu+\nu_2}(\mu)$ and $S^\nu(\mu)$
%are equal. Both are concentrated on $]{-}\infty,d]$ and $S^\nu(\mu)(\{d
%\})\leq\eta(\{d\})$.
%\end{lem}
%\begin{pf}
%It is clear that $S^{\nu+\nu_2}(\mu)\leqc S^\nu(\mu)$. We have also
%$S^{\nu+\nu_2}(\mu)\leqc\eta$ so that we can apply Lemma~\ref{interv}
%to this pair. We are clearly not in the first situation because the
%support of $\eta$ is bounded on the right. Hence the assertion of the
%second or the third case applies. In either case $\eta$ and $S^{\nu+
%\nu_2}(\mu)$ are concentrated on $]{-}\infty,d]$. The shadow is smaller
%than $(\nu+\nu_2)|_{]{-}\infty,d]}=\nu+\nu_2(\{d\})\delta_d$ and
%considering carefully the third case, $S^{\nu+\nu_2}(\mu)\leq\nu$.
%Finally we obtain $S^{\nu+\nu_2}(\mu)=S^\nu(\mu)$.
%\end{pf}

For $u,v\in\R, u< v$ let $g_{u,v}$ be defined by
%
%e13 #&#
\begin{equation}
g_{u,v}(x)= \cases{ v-x,&\quad $\mbox{if $x\in[u,v]$,}$\vspace*{2pt}
\cr
0,&\quad $\mbox{otherwise.}$}
\end{equation}

%le5.2 #&#
\begin{lem}\label{yahoo}
Let $\sigma$ be a nontrivial signed measure of mass $0$ and denote its
Hahn decomposition by $\sigma=\sigma^+-\sigma^-$. There exist $a\in
\Spt
(\sigma^+)$ and $b>a$ such that $\int g_{a,b}(x) \,\d\sigma(x)>0$.
\end{lem}

\begin{pf}
First, notice that $u\mapsto\int g_{u,u+1}(x) \,\d\sigma(x)$ does not
vanish identically. %\footnote{Is that clear? Yes: Consider the
%functions $P:u\mapsto\int g_{u,u+1}(x) \,\d\Delta^+(x)$ and $M:u\mapsto
%\int g_{u,u+1}(x) \,\d\Delta^-(x)$. These functions are the convolutions
%of the ``density'' $g_{0,1}$ with the measures $\Delta^+ $ resp.\ $
%\Delta^-$. Forming this convolutions is $1-1$ and as the measures $
%\Delta^+$ and $\Delta^-$ are different, the functions $P$ and $M$ are
%different. }
Since, by Fubini's theorem,
\[
\iint g_{u,u+1}(x) \,\d\sigma(x) \,\d u=0
\]
there exists $u\in\R$ such that $\int g_{u,u+1}(x) \,\d\sigma(x)>0$. The
set $\Spt(\sigma^+\cap[u,u+1[)$ cannot be empty, so let $a=\min
(\Spt
(\sigma^+\cap[u,u+1])$. It follows that
\[
0<\int g_{u,u+1} \,\d\sigma\leq\int g_{a,u+1} \,\d\sigma.
\]
%
%\qedhere
\upqed\end{pf}

%th5.3 #&#
\begin{them}[(Uniqueness of the monotone martingale coupling)]\label
{mono_unique}
Let $\pi$ be a monotone martingale transport plan and $\mu=\proj^x_\#
\pi
$ and $\nu=\proj^y_\#\pi$. Then $\pi$ is
the left-curtain coupling $\pi_\lc$ from $\mu$ to $\nu$.
\end{them}

\begin{pf}
Let $\pi$ be left-monotone with monotonicity set $\Gamma$ as in
Definition~\ref{defi_monoto} and let $\pi_\lc$ be the left-curtain
transport plan between $\mu$ and $\nu$. We consider the target measures
$\nu^\pi_x$ and $\nu^{\pi_\lc}_x$ obtained when transporting the
$\mu
$-mass of $]{-}\infty,x]$ into $\nu$, that is,
\[
\nu^{\pi}_x=\proj^y_\#\pi|_{]{-}\infty,x]\times\R}
\]
and
\[
\nu^{\pi_\lc}_x=S^\nu(\mu|_{]{-}\infty,x]})=
\proj^y_\#{\pi_\lc }|_{]{-}\infty
,x]\times\R}.
\]
If $\nu^\pi_x=\nu^{\pi_\lc}_x$ for every $x$, then $\pi=\pi_\lc
$ by the
definition of the curtain-coupling in Theorem~\ref{lc_defi}.

Assume for contradiction that there exists some $x$ with $\nu^\pi
_x\neq
\nu^{\pi_\lc}_x$. This means in particular that $\sigma_x=(\nu
^{\pi_\lc
}_x-\nu^\pi_x)\neq0$. The shadow property implies that $\nu^{\pi
_\lc
}_x\leqc\nu^\pi_x$. By Lemma~\ref{yahoo}, we can pick $u\in\Spt
(\sigma
_x^+)$ and $v>u$ such that
\[
\int g_{u,v} \,\d\sigma_x>0.
\]
As $u\in\Spt\sigma_x^+$, $\sigma^+_x\leq\nu-\nu_x^\pi=\proj
^y_\#\pi
|_{]x,+\infty[\times\R}$, and $\pi(\Gamma)=1$, there is a sequence
$(x_n',u_n)_n$ such that: %\marginpar{$x_n'$ und $ u'$ sind vielleicht
%schlechte bezeichnungen weil die Striche in dem Beweis fuer Schatten
%stehen.}
%
\begin{itemize}
\item$x'_n>x$,
\item$(x'_n,u_n)\in\Gamma$,
\item$u_n\rightarrow u$.
\end{itemize}
By the monotonicity property of $\Gamma$, for every $t\leq x$ and
$n\in
\N$, the set $\Gamma_t$ defined by $\{y\in\R\dvtx (t,y)\in\Gamma\}$ cannot
intersect $]{-}\infty,u_n[$ \emph{and} $]u_n,+\infty[$. Hence, for
$t\leq x$,
%
%e14 #&#
\begin{equation}
\label{vide} \Gamma_t\,\cap\,]{-}\infty,u[\, =\varnothing\quad\mbox{or}\quad
\Gamma_t\,\cap\, ]u,+\infty[\, =\varnothing.
\end{equation}
This remark will be important in the sequel of the proof.

We distinguish two cases depending on the respective positions of $u$
and $x$.
\begin{enumerate}[(1)]
\item[(1)] First case: $u<x$. Note that we have
\[
\nu^\pi_x-\nu^\pi_u=
\proj^y_\#\pi|_{]u,x]\times\R}
\]
and
\[
\nu^{\pi_\lc}_x-\nu^{\pi_\lc}_u=
\proj^y_\#{\pi_\lc }|_{]u,x]\times\R
}=S^{\nu-\nu^{\pi_\lc}_u}(
\mu|_{]u,x]}).
\]
As a consequence of \eqref{vide} and of the fact that $\pi$ is a
martingale transport plan, $\pi$ transports the mass of $]{-}\infty,u]$
to $]{-}\infty,u]$ and the mass of $]u,x]$ to $[u,+\infty[$. We show
below that the same applies to $\pi_\lc$, more precisely that $\nu
^{\pi
_\lc}_u\leqc\nu^\pi_u$ %(on the left side of $u$)
and $(\nu^{\pi_\lc}_x-\nu^{\pi_\lc}_u)\leqc(\nu^\pi_x-\nu^\pi
_u)$. %on
%the right side of $u$. I think ``on the right side of $u$'' does not
%have a real meaning.
%
\begin{itemize}
\item The measure $\nu^{\pi_\lc}_u$ is the shadow of $\mu
|_{]{-}\infty
,u]}$ in $\nu$. We have also $\mu|_{]{-}\infty,u]}\leqc\nu^\pi_u$
and $\nu
^\pi_u\leq\nu$ so that $\nu^{\pi_\lc}_u\leqc\nu^\pi_u$. We
apply now
Lemma~\ref{interv} and obtain that $\nu_u^{\pi_{\lc}}$ is concentrated
on $]{-}\infty,u]$ and $\nu^{\pi_\lc}_u(\{u\})\leq\nu^{\pi}_u(\{u\})$.
%%in any case.

\item We have $\pi|_{]u,x]\times\R}\in\Pi_M(\mu_{]u,x]},\eta)$ where
$\eta:=\proj^y_\#\pi|_{]u,x]\times\R}=\nu^\pi_x-\nu^\pi_u$ is
concentrated on $[u,+\infty[$. More precisely, we have
\[
\eta\leq\bigl(\nu-\nu^\pi_u\bigr)|_{[u,+\infty[}\leq
\bigl(\nu-\nu^{\pi_{\lc
}}_u\bigr)|_{[u,+\infty[}\leq\nu-
\nu^{\pi_{\lc}}_u
\]
because $\nu^{\pi_\lc}_u$ and $\nu^{\pi}_u$ are concentrated on
$]{-}\infty,u]$ and $\nu^{\pi_\lc}_u(\{u\})\leq\nu^{\pi}_u(\{u\})$
as we
have seen above. Moreover, we have $\mu|_{]u,x]}\leqc\eta$. Hence,
\[
\nu^{\pi_\lc}_x-\nu^{\pi_\lc}_u=S^{\nu-\nu^{\pi_\lc}_u}(
\mu _{]u,x]})\leqc\eta=\nu^\pi_x-
\nu^\pi_u.
\]
\end{itemize}
Note that $g_{u,v}$ is convex on $[u,+\infty[$ so that $\int g_{u,v}
\,\d
(\nu^{\pi_\lc}_x-\nu^{\pi_\lc}_u)\leq\break  \int g_{u,v} \,\d(\nu^\pi
_x-\nu^\pi
_u)$. Moreover, we have $\int g_{u,v} \,\d\nu^{\pi_\lc}_u\leq\int
g_{u,v} \,\d\nu^\pi_u$ because $\nu^{\pi_\lc}_u(\{u\})\leq\nu
^{\pi}_u(\{
u\})$. Summing these inequalities, we obtain $\int g_{u,v} \,\d\nu^{\pi
_\lc}_x\leq\int g_{u,v} \,\d\nu^{\pi}_x$, which is a contradiction to
$\int g_{u,v} \,\d\sigma_x>0$.

\item[(2)] Second case: $x\leq u$. The measure $\pi$ cannot transport mass
from $]{-}\infty,x]$ to $]u,+\infty[$. Indeed, because of the martingale
property it then would also transport mass to the set $]{-}\infty,u[$,
contradicting \eqref{vide}. Thus, $\nu^{\pi}_x$ is concentrated on
$]{-}\infty,u]$. But we have $\nu^{\pi_\lc}_x\leqc\nu^{\pi}_x$ so that
considering Lemma~\ref{interv}, $\int g_{u,v} \,\d\nu^{\pi_\lc
}_x\leq\int
g_{u,v} \,\d\nu^{\pi}_x$ holds (even in the third case of this lemma
where $a=u$). This contradicts $\int g_{u,v} \,\d\nu^\pi_x>0$.\quad\qed
\end{enumerate}
\noqed\end{pf}
% as it is possible to transport the mass of $]u,x]$ into $\nu_{[u,+
%\infty[}-\sigma_u(\{u\})\delta_u$, for instance with the shadow
%projection $S^{\nu-\nu^\pi_u}(\mu_{]u,x]})=\nu^\pi_x-\nu^\pi_u$, the
%shadow of $\mu_{]u,x]}$ in $\nu-S^\nu(\mu_{]{-}\infty,u]})=\nu-\nu^{\pi_
%\lc}_u$ is equal to the shadow of $\mu_{]u,x]}$ in $\nu_{[u,+\infty[}$
%or in $\nu-\nu^\pi_u$. In particular this measure is concentrated on
%the set $[u,+\infty[$.

%
%re5.4 #&#
\begin{rem} % \marginpar{Perhaps we do not need this remark.}
The two cases in the proof are actually not very different. In both of
them, $\pi|_{]{-}\infty,x]\times\R}$ and ${\pi_\lc}|_{]{-}\infty
,x]\times
\R}$ (roughly speaking the transport plans restricted to $\mu
|_{]{-}\infty
,x]}$) are concentrated on
\[
\bigl(]{-}\infty,u]\,\times\,]{-}\infty,u] \bigr)\cup
 \bigl(]u,+\infty [\,\times\, [u,+\infty[ \bigr)
\]
and this lies at the core of the argument.
\end{rem}
%
%s5.1 #&#
\subsection{Structure of the monotone martingale coupling}
It remains to establish Corollary~\ref{FamousCoro} which states that if
$\mu$ is continuous, then $\pi_\lc$ is concentrated on the graph of two
functions. We need the following lemma.

%le5.5 #&#
\begin{lem}\label{Borel_lemma}
Assume that $\Gamma\subseteq\R^2$ is a Borel set such that for each
$x\in\R$ we have $|\Gamma_x|\leq2$. Then $S=\proj^x(\Gamma)$ is a
Borel set and there exist Borel functions $T_1,T_2\dvtx S\to\R$ with
$T_1\leq T_2$ such that
\[
\Gamma= \graph(T_1) \cup\graph(T_2).
\]
\end{lem}

\begin{pf}
This is a consequence of \cite{Ke95}, Theorem~18.11.
\end{pf}
We can now complete the proof.
\begin{pf*}{Proof of Corollary~\ref{FamousCoro}}
Consider the left-curtain coupling $\pi_\lc$ between measures $\mu
\leqc
\nu$, where $\mu$ is continuous. As $\pi_\lc$ is left-monotone there
exists a Borel monotonicity set $\Gamma$ as in Definition~\ref
{defi_monoto}. Note that if $\mu(A)=0$, the set $\Gamma\setminus
(A\times\R)$ is still a monotonicity set. This applies in particular
to all countable sets since $\mu$ is continuous.

With the notation of Lemma~\ref{AccumulatingBadPoints} let us show that
$A=\{x\in\R\dvtx |\Gamma_x|\geq3\}$ is countable. If not, we can apply
this lemma and obtain $x\in\R$ with three points $y^-<y<y^+$ in the set
$\Gamma_x$ that can be approximated from the right-hand side. In particular,
there exists $(x',y')\in\Gamma$ with $x'>x$ and $y'\in\,]y^-,y^+[$,
which is the forbidden configuration \eqref{BadConfiguration}.
Therefore, $A$ is countable so that we can assume that $|\Gamma_x|\leq
2$ for every $x$.
%If necessary we can replace $S=\proj_x(\Gamma)$ by a Borel set $S'
%\subset S$ with $\mu(S'\setminus S)=0$. In that case $\Gamma'=\Gamma
%\cap(S'\times\R)$ is still a Borel monotonicity set.
Applying Lemma~\ref{Borel_lemma}, we obtain the desired assertion. %
%\marginpar{Why is that true? Can you please repair this brocken
%paragraph?}
\end{pf*}
The following lemma permits to obtain uniqueness of the optimal
martingale transport plan, provided that we know that every optimal
martingale transport is concentrated on the graphs of two mappings (see
Section~\ref{sec:general}). We can apply it to the martingale transport
plans when $\mu$ is continuous and recover the uniqueness of the
monotone transport plan in this particular case.

%le5.6 #&#
\begin{lem}\label{half-sum}
Let $\mu$ and $\nu$ be in convex order and $\mathcal{E}$ a nonempty
convex set of martingale transport plans. Assume that every $\pi\in
\mathcal{E}$ is concentrated on some $\Gamma^\pi\subset\R^2$ with
$|\Gamma^\pi_x|\leq2$ for every $x\in\R$. Then the set $\mathcal{E}$
consists of a single point.
\end{lem}

\begin{pf}
Let $\pi$ and $\pi'$ be elements of $\mathcal{E}$. We consider $\bar
\pi
=\frac{\pi+\pi'}{2}\in\mathcal{E}$ and $\Gamma^{\bar\pi}$,
which can be
seen as the graph of two functions according to Lemma~\ref
{Borel_lemma}. The measures $\pi$ and $\pi'$ are also concentrated on
$\Gamma^{\bar\pi}$. For two disintegrations $(\pi_x)_{x\in\R}$ and
$(\pi'_x)_{x\in\R}$ with respect to $\mu$, we know that $\mu
$-a.s. $\pi
_x$ and $\pi'_x$ are probability measures concentrated on $\Gamma
^{\bar
\pi}_x$ and with the same barycenter, namely $x$. It follows that $\pi
'_x=\pi_x$, $\mu$-a.s. so that $\pi'=\pi$.
\end{pf}

%s6 #&#
\section{Optimality properties of the monotone martingale transport}
\label{sec:optimal}
In this section, we prove that $\pi_\lc$ is the unique optimal coupling
for the martingale optimal transport problem \eqref{PrimalMart}
associated to two different kinds of cost functions. The special case
$c(x,y)=\exp(y-x)$ is in the intersection of these two families of cost
functions.

%th6.1 #&#
\begin{them}\label{OptForExp}
Assume that $c(x,y)=h(y-x)$ for some differentiable function $h$ whose
derivative is strictly convex and that $c$ satisfies the sufficient
integrability condition. If there exists a finite martingale transport
plan, then $\pi_\lc$ is the unique optimizer.
\end{them}

\begin{pf} %Assume that $h''$ is strictly convex.
We have to show that every finite optimizer $\pi$ is monotone.
Pick a set $\Gamma$ such that $\pi(\Gamma)=1$ and $\Gamma$ resists
improvements by barycenter preserving reroutings as in Lemma~\ref
{GlobalLocal}. Pick $(x,y^-),(x,y^+),(x',y')\in\Gamma$. Striving for a
contradiction we assume that they satisfy \eqref{BadConfiguration}. Let
us define a transport $\alpha$ on these edges and a competitor $\alpha
'$ of it. We pick $\lambda\in\,]0,1[$ such that $\lambda y^++(1-\lambda)
y^-=y'$. The measure $\alpha$ puts mass $\lambda$ on $(x,y^+)$, mass
$1-\lambda$ on $(x,y^-)$ and mass $1$ on $(x',y')$. Our candidate for
$\alpha'$ will assert mass $1-\lambda$ on $(x',y^-)$, mass $\lambda$ on
$(x',y^+)$ and mass $1$ on $(x,y')$. Clearly, $\alpha'$ is a competitor
of $\alpha$. It leads to smaller costs if and only if
\[
\lambda c\bigl(x,y^+\bigr)+ (1-\lambda) c\bigl(x,y^-\bigr)+c
\bigl(x',y'\bigr) > \lambda c\bigl(x',y^+
\bigr)+ (1-\lambda) c\bigl(x',y^-\bigr)+c\bigl(x,y'
\bigr).
\]
A sufficient condition for this is that
%
%e15 #&#
\begin{equation}
d(t):=\lambda c\bigl(t,y^+\bigr)+ (1-\lambda) c\bigl(t,y^-\bigr)-c
\bigl(t,y'\bigr)
\end{equation}
is strictly decreasing in $x$. In terms of $h$, the function $d$ can be
written as
\[
d(t)=\lambda h\bigl(y^+-t\bigr)+(1-\lambda) h\bigl(y^--t\bigr)-h
\bigl(y'-t\bigr).
\]
To have it decreasing, it is sufficient that
\begin{eqnarray*}
0&>&d'(t)\\
&=&-\lambda h'\bigl(y^+-t\bigr)-(1-\lambda)
h'\bigl(y^--t\bigr)+h'\bigl(y'-t\bigr)
\\
&=&h'\bigl(\lambda\bigl(y^+-t\bigr)+(1-\lambda) \bigl(y^--t\bigr)
\bigr)-\bigl[\lambda h'\bigl(y^+ -t\bigr)+(1-\lambda)
h'\bigl(y^--t\bigr)\bigr].
\end{eqnarray*}
Finally, it is sufficient to know that $h'$ is strictly convex which
holds by assumption.
\end{pf}
%
%re6.2 #&#
\begin{rem}
%The left-curtain transport plan is also a solution to the problem of
%minimizing the essential supremum of $y-x$ among all martingale
%%transport plans with the same marginals. For this fact we point out
%that the function $h_n:x\mapsto x^{2n+1}$ has a strictly %convex
%derivative for any $n$ and that $n\in\mathbb{N}\mapsto(
%\int(y-x)^{2n+1} \,\d\pi(x,y)))^{1/(2n+1)}$ goes to %$\mathrm{essup}_\pi
%(y-x)$ as $n$ tends to $+\infty$ for every martingale transport plan $
%\pi$.\footnote{We thank Fillipo %Santambrogio for pointing this out to
%us.}
The left-curtain transport plan is also a solution to the problem of
minimizing the essential supremum of $y-x$ among all martingale
transport plans with the same marginals. To see this, note that the
function $h_n\dvtx x\mapsto\exp(n x)$ has a strictly convex derivative for
every $n>0$ and that $\frac{1}{n}\ln(\int\exp(n(y-x)) \,\d\pi(x,y))$
tends to $\mathrm{essup}_\pi(y-x)$ as $n\to+\infty$ for every
martingale transport plan $\pi$.\footnote{We thank Fillipo Santambrogio
for pointing this out to us.}
\end{rem}

We mention another class of cost functions for which the monotone
martingale transport plan $\pi_\lc$ is optimal.

%th6.3 #&#
\begin{them}\label{MonIsOptII}
Let $\psi$ be a nonnegative strictly convex function and $\phi$ a
nonnegative decreasing function. Consider the cost function $c(x,y)=
\phi(x)\psi(y)\geq0$. For two finite measures $\mu$ and $\nu$ in convex
order, the left-curtain coupling $\pi_\lc$ is the unique optimal transport.
\end{them}

One could show that optimal martingale couplings are monotone in a very
similar way as in the proof of Theorem~\ref{OptForExp}. We prefer to
give an alternative proof relying on the order properties of the
left-curtain coupling.
\begin{pf*}{Proof of Theorem \ref{MonIsOptII}}
Let $\pi$ be optimal for the problem and assume that $\int c \,\d\pi
<+\infty$. We want to prove $\int c \,\d\pi_\lc\leq\int c \,\d\pi$ with
equality if and only if $\pi=\pi_\lc$. First of all note that for
positive measurable functions $f$
\[
\int f(x) \phi(x) \,\d\mu(x)=\int_0^{+\infty} \biggl(
\int\I _{]{-}\infty, \phi
^{-1}(t)]} f(x) \,\d\mu(x) \biggr)\,\d t,
\]
where $\phi^{-1}(t)$ means $\sup\{x\in\R\dvtx t\leq\phi(x)\}$. Taking
$f(x)=\int\psi(y) \,\d\pi_x(y)$, we obtain
%
%e16 #&#
\begin{equation}
\label{avant} \int c(x,y) \,\d\pi(x,y)=\int_{0}^{+\infty}
\biggl(\int\psi(y) \,\d \nu^\pi |_{\phi^{-1}(t)}(y) \biggr)\,\d t,
\end{equation}
where $\nu^\pi_u$ denotes $\proj^y_\#\pi|_{]{-}\infty,u]}$ as in the
\hyperref[sec:intro]{Introduction} or in Section~\ref{sec:unique}. In particular, $\nu^{\pi
_\lc}_u$ equals $S^\nu(\mu_{]{-}\infty,u]})$. Of course the
representation \eqref{avant} remains true if we replace all occurrences
of $\pi$ by $\pi_\lc$.

The measures $\nu^{\pi_\lc}_u$ and $\nu^\pi_u$ are in convex order and
$\psi$ is strictly convex. Thus, $\int\psi\,\d\nu^{\pi_\lc}_u\leq
\int
\psi\,\d\nu^{\pi}_u$ and equality holds if and only if the two measures
coincide. This follows from Strassen's theorem (Theorem~\ref
{BasicExistence}) and the equality case in Jensen's inequality.
Finally, it follows from \eqref{avant} that $\pi$ is the left-curtain
coupling.
\end{pf*}

%s7 #&#
\section{Other cost functions---other optimal martingale couplings}
\label{sec:general}

In this section, we use Lemma~\ref{GlobalLocal} to derive results that
appeal to general cost functions.

%s7.1 #&#
\subsection{Cost functions of the form $c(x,y)=h(y-x)$}

%th7.1 #&#
\begin{them}\label{LessThanK}%\marginpar{I think that some examples are
%given in the introduction, are not they}
Assume that the cost function $c(x,y)$ is given by $h(y-x) $ for some
function $ h$ which is twice continuously differentiable. If affine
functions $x\mapsto ax+b$ meet $h'(x)$ in at most %less than (M: Ich
%have das gerade geaendert, ich denke das passt jetzt.)
$k$ points and $\pi$ is an optimal transport plan, then there exists a
disintegration $(\pi_x)_{x\in\R}$ such that for any $x\in\R$ at least
one of the two following statements holds:
\[
\mu\bigl(\{x\}\bigr)>0 \quad\mbox{or}\quad \card\bigl(\Spt(\pi_x)\bigr)\leq k.
\]
In particular, if $\mu$ is continuous then $\card(\Spt(\pi_x))\leq k$
is satisfied $\mu$-a.s. for any disintegration of $\pi$.
\end{them}

\begin{pf}%[Proof of Theorem~\ref{LessThanK}.]
Let $\pi$ be optimal and $\Gamma$ according to Lemma~\ref
{GlobalLocal}. If there are only countably many continuity points of
$\mu$ such that $\card(\Gamma_x)\geq k+1$, then we can remove them.
Assume for contradiction that there are uncountably many.
Consider the set
\[
\tilde\Gamma=\bigl\{(x,y)\in\Gamma\dvtx \mu\bigl(\{x\}\bigr)=0\bigr\}
\]
to obtain $a\in\R$ and $b_0< \cdots< b_{k}\in\Gamma_a$ verifying the
assertions of Lemma~\ref{AccumulatingBadPoints}.

Let $a'\in\R$, $\lambda\in\,]0,1[$ and set $b_\lambda= (1-\lambda)
b_0 +\lambda b_{k}$. We will compare
%
%e17 #&#
\begin{equation}
\label{cost1} h(b_\lambda-a)+\lambda h\bigl(b_k-a'
\bigr)+(1-\lambda)h\bigl(b_0-a'\bigr)
\end{equation}
and
%
%e18 #&#
\begin{equation}
\label{cost2} h\bigl(b_\lambda-a'\bigr)+\lambda
h(b_k-a)+(1-\lambda)h(b_0-a).
\end{equation}
As $a'$ tends to $a$, $b_i-a'$ tends to $b_i-a$. Considering a Taylor
expansion of $h$ at $b_i-a$, we find some $\eps>0$ such that
$|a-a'|<\eps$ implies
\[
\bigl\llvert \bigl[h\bigl(b_i-a'\bigr)-h(b_i-a)
\bigr]-h'(b_i-a)\cdot\bigl(a-a'\bigr)
\bigr\rrvert \leq \bigl|h''(b_i-a)\bigr|
\bigl(a-a'\bigr)^2
\]
for $i\in\{0,\lambda,k\}$. Hence, if we subtract (\ref{cost1}) from
(\ref{cost2}) we obtain
%
%e19 #&#
\begin{equation}
\label{comb} \bigl(h'(b_\lambda-a)- \bigl[(1-\lambda)
h'(b_0-a)+\lambda h'(b_k-a)
\bigr] \bigr) \bigl(a'-a\bigr)
\end{equation}
up to an error of
\[
\bigl[(1-\lambda)\bigl|h''(b_0-a)\bigr|+
\lambda\bigl|h''(b_k-a)\bigr|+\bigl|h''(b_\lambda
-a)\bigr| \bigr]\cdot\bigl(a-a'\bigr)^2.
\]
But $h'$ is not linear so that \eqref{comb} is not identically zero.
Moreover, according to the assumption on $h'$ and the affine functions
there is an index $i\in\{1,\ldots,k-1\}$ such that if $b_\lambda=b_i$
and $a'\neq a$ then \eqref{comb} is not zero. More precisely, as $h''$
is continuous there exists some $\eps_1<\eps$ such that if
$|b_i-b_\lambda|<\eps_1$ and $0<|a-a'|<\eps_1$ then the difference of
\eqref{cost1} and \eqref{cost2} is not zero and its sign is determined
by the one of $a-a'$.

% Let $B_i\ni y_i$ be a rational semi-open interval such that (
%\ref{comb}) is not $0$ on $\bar{B_i}$. Using the uniform continuity of
%$\varphi''$ we can find $\eps_2>0$ so that (\ref{cost1}) and $(
%\ref{cost2})$ are different for all $(x',b_\lambda)$ with $x'\in\,]a-
%\eps_2,a+\eps_2[\setminus\{a\}$ and $b_\lambda\in\,]b_i-\eps_1, b_i+
%\eps_1 [$. The sign of $x'-a$ can also be chosen to make (\ref{cost1})
%smaller than (\ref{cost2}).

Since $a, b_0, \ldots, b_k$ were chosen according to Lemma~\ref
{AccumulatingBadPoints}, we may
pick $a'$ and $b_\lambda\in\Gamma_{a'}$ such that $(a',b_\lambda)$ is
sufficiently close to $(a,b_i)$ and $a'$ is on the correct side of $a$,
making (\ref{cost1}) smaller than (\ref{cost2}).

Setting
\begin{eqnarray*}
\alpha&=& \lambda\delta_{(a, b_k)} + (1-\lambda) \delta_{(a,b_0)} +
\delta_{(a',b_\lambda)},
\\
\alpha'&=& \lambda\delta_{(a', b_k)} + (1-\lambda)
\delta_{(a',b_0)} + \delta_{(a,b_\lambda)},
\end{eqnarray*}
we have thus found a competitor $ \alpha'$ which has lower costs than
$\alpha$, contradicting the choice of $\Gamma$.
\end{pf}
\subsection{The cost function $h(y-x)$ in the usual setup}
%We are not aware that a result in the spirit of Theorem~\ref{LessThanK} has been noticed in the literature on optimal
%transport. It %therefore seems worthwhile to mention the following
%variant of Theorem~\ref{LessThanK}.

It seems worthwhile to mention that Theorem~\ref{LessThanK} is the
martingale variant of a result that belongs to the theory of the
classical problem \eqref{big_cost}. We mention it below in Theorem~\ref
{ClassicLessThanK} because we are not aware that it has been recorded
in the literature in this form. In fact for a family of special costs
we can bound the number of parts the mass can split in if it is
transported optimally. Note that this number is not attained for every
pair $(\mu,\nu)$ (see \cite{RuUc00}). The similarity with Theorem~\ref
{LessThanK} lies in the fact that we want to count the number of
intersection points of $\graph(h')$ with affine lines in the martingale
case, and with horizontal lines in the classical setup.

%th7.2 #&#
\begin{them}\label{ClassicLessThanK}
Let $k$ be a positive integer and let $h\dvtx \R\to\R$ be a twice % for
%some
%people smooth means just the same as one time differentiable.
continuously differentiable function such that the cost function
$c\dvtx (x,y)\mapsto h(y-x)$ satisfies the sufficient integrability
condition with respect to probability measures $\mu$ and $\nu$. Assume
also that $C(\mu,\nu)<+\infty$.

% If constant functions $x\mapsto b$ meet $h'(x)$ in less than $k$
%points and $\pi$ is an optimal transport plan --we are interested in
%the intersections of $\graph(h')$ and the horizontal lines--

If the equation $h'(x)=b$ has at most $k$ different solutions for
$b\in\R$, then
there exists a disintegration $(\pi_x)_{x\in\R}$ such that for any
$x\in\R$ at least one of the two statements
\[
\mu\bigl(\{x\}\bigr)>0\quad \mbox{or}\quad \card\bigl(\Spt(\pi_x)\bigr)\leq k
\]
holds. In particular, if $\mu$ is continuous then $\card(\Spt(\pi
_x))\leq k$ is satisfied $\mu$-a.s. for any disintegration.
\end{them}

%%%%For the proof see after end{document}%%%%

%s7.3 #&#
\subsection{(Counter)examples based on the cost function $c(x,y)=(y-x)^4$}

In this section, we give two counterexamples that distinguish the
general behavior from the one of the curtain transport plan: the
optimizer is in general not unique and it may very well split into more
than two parts even if the starting distribution is continuous (see
Corollary~\ref{FamousCoro}, resp., Theorem~\ref{LessThanK}). Throughout
this subsection, we consider the cost function $c(x,y)=(y-x)^4$.

%s7.3.1 #&#
\subsubsection{Example of nonuniqueness of the transport} \label{non_uniq}
Let $\mu$ be uniformly distributed on $\{-1;1\}$ and $\nu$ uniformly
distributed on $\{-2;0;2\}$. We denote $-1$ and~$1$ by $(x_i)_{i=1,2}$
and $-2, 0$ and $2$ by $(y_j)_{j=1,2,3}$. To any matrix $A=(a_{i,j})$
of two rows and three columns satisfying $\sum_j a_{i,j}=1/2$ and
$\sum_i a_{i,j}=1/3$, we associate the transport plan defined by $\pi(\{
(x_i,y_j)\})=a_{i,j}$. For such a transport plan, the accumulated costs equal
\begin{eqnarray*}
\sum_{i,j}a_{i,j}\cdot
|x_i-y_j|^4&=&(a_{1,1}+a_{1,2}+a_{2,2}+a_{2,3})+3^4
\cdot (a_{1,3}+a_{2,1})
\\
&=&1+80(a_{1,3}+a_{2,1}).
\end{eqnarray*}
The matrices associated to a martingale transport plan are
\[
A_\lambda=\pmatrix{1/4 & 1/4 & 0
\vspace*{2pt}\cr
1/12 & 1/12 & 1/3}
+\lambda\pmatrix{1/12 & -1/6 & 1/12
\vspace*{2pt}\cr
-1/12 & 1/6 & -1/12}, %
\]
where $\lambda\in[0,1]$. Therefore, the martingale transport plan
associated to the parameter $\lambda$ gives rise to total costs of
$1+80(\lambda/12+1/12-\lambda/12)=23/3$, independently of $\lambda$. We
conclude that every martingale transport plan is optimal.

%s7.3.2 #&#
\subsubsection{Example of splitting in exactly three points in the
continuous case} \label{threepoints}
Roughly speaking, we have proved in Theorem~\ref{LessThanK} that if
$\mu
$ is continuous, $\dd\mu(x)$-mass elements split in at most three
points. Indeed, $t\mapsto t^4$ has derivative $t\mapsto4t^3$ which is
of degree $3$. In this paragraph, we give a numerical example showing
that this upper bound is sharp. The construction is inspired by the
dual theory of the martingale transport problem mentioned in Section
\ref{dual_pro}. Briefly, Figure~\ref{graddrei} depicts a family of
curves indexed by $x$. These curves touch three envelope curves at
three moving points $y_1, y_2$ and $y_3$ close to $-1, 0$ and $1$. The
optimal martingale transport plan that we construct is supported by the
union of the graphs $\Gamma_i=\{(x,y_i(x))\in\R^2\dvtx x\in\,]0,1/5[\}$ for
$i=1,2,3$.

%f4 #&#
\begin{figure}

\includegraphics{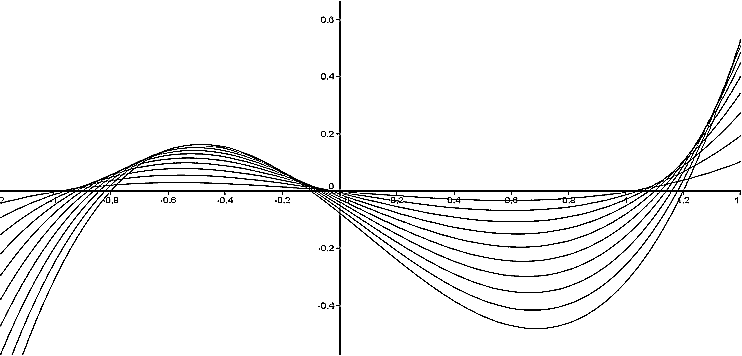}

\caption{Graphs and envelope of the functions $y\mapsto F(x,y)$ for
$x\in[0,1/5]$.}\label{graddrei}
\end{figure}

%\begin{figure}[ht]
%\begin{center}
%\includegraphics[width=10cm]{graddrei}
%
%\end{center}
%\end{figure}
%\marginpar{Notations: here $\psi$ is not a convex function. Is that a
%problem?}
Let $\psi\dvtx \R\to\R$ be defined by
%
%e20 #&#
\begin{equation}
\label{defsup} \psi(y)=y^4-\max_{x\in[0,1/2]} \biggl\{4x
\biggl(y+\frac
{x}2\biggr) (y+1-x) (y-1-x) \biggr\}.
\end{equation}
Hence, for any $(x,y)\in[0,1/2]\times\R$
\[
y^4-\psi(y)\geq4xy^3-6x^2
y^2+a_1(x)y+b_1(x),
\]
where $a_1(x)=4x-4x^2-4x^3$ and $b_1(x)=2x^2-2x^4$. But
$y^4=(y-x)^4+4xy^3-6x^2y^2+a_2(x)y+b_2(x)$ so that
%
%e21 #&#
\begin{equation}
\label{dualeq} (y-x)^4\geq a_3(x)+b_3(x)y+
\psi(y)
\end{equation}
for $a_3=a_1-a_2$ and $b_3=b_1-b_2$. %\marginpar{Include this sentence?
%Notice that we recover duality condition \eqref{DualCond} with $
%\varphi(x)=a_3(x)+b_3(x)x$ and $\Delta(x)=b_3(x)$.}
Here, (\ref{dualeq}) is an equality at the point $(x_0,y_0)$ if and
only $\psi(y_0)$ is realized in (\ref{defsup}) by $x=x_0$. Integrating
(\ref{dualeq}) against a transport plan $\pi$, one obtains
\[
\iint(y-x)^4 \,\dd\pi(x,y)\geq\int a_3(x) \,\dd\mu(x)+\iint
b_3(x)y \,\dd\pi (x,y)-\int\psi(y) \,\dd\nu(y) %
\]
and the equality holds if and only if $\pi$ is concentrated on
\[
\bigl\{(x,y)\in[0,1/2]\times\R\dvtx (y-x)^4=a_3(x)+b_3(x)y+
\psi(y)\bigr\}.
\]
Moreover, as we are considering a martingale transport plan we have
\[
\iint(y-x)^4 \,\dd\pi(x,y)\geq\int a_3(x) \,\dd\mu(x)+\int
b_3(x)x \,\dd\mu (x)+\int\psi(y) \,\dd\nu(y). %
\]
Here, the lower bound on the right-hand side is the same for every
martingale transport plan $\pi$. It follows that martingale transport
plans concentrated on
$\{(x,y)\in[0,1/2]\times\R\dvtx (y-x)^4=a_3(x)+b_3(x)y+\psi(y)\}$ are
optimal with respect to their marginals.
We set $F(x,y)=4x(y+\frac{x}2)(y+1-x)(y-1-x)$ so that \eqref{defsup} is
$\psi(y)=y^4-\sup_{x\in[0,1/2]}F(x,y)$. In Figure~\ref{graddrei}, one
can see the graphs of $F(x,\cdot)$ for values of $x$ between $0$ and $1/5$.

\step{We will prove that for $y\in\,]{-}1,0[\, \cup\,]1,2[$, $F(\cdot
,y)\dvtx [0,1/2]\to\R$ has a unique global maximum in $]0,1/2[$.} Actually,
$F(\cdot,y)$ has main term $2x^4$. Therefore, it is sufficient to prove
that $\partial_x F(\cdot,y)$ is positive for $x=0$ and negative for
$x=1/2$. Indeed this means that we are analyzing the variation of the
polynomial function $F(\cdot,y)$ of degree $4$ on an interval where its
variations are different from the asymptotic ones. In particular
$F(\cdot,y)$ will have a unique maximum on $]0,1/2[$. This turns out to
be true. Indeed,
%
%e22 #&#
\begin{equation}
\partial_x F(x,y)=4 \bigl((x+y)\bigl[(x-y)^2-1
\bigr]+x(x+2y) (x-y) \bigr),
\end{equation}
so that for any parameter $y$ in $]{-}1,0[ \,\cup\,]1,2[$, the function
$\partial_x F(\cdot,y)$ is positive in $x=0$ since it equals
$y\mapsto
4 (y(y^2-1) )$. For $x=1/2$, straightforward considerations
show that $\partial_x F(1/2,y)$ is negative for all $y\in\,]{-}\infty,2]$.

\step{We will now show that for a given parameter $x\in\,]0,1/5[ $, $x$
is the maximum of $F(\cdot,y)$ on $[0,1/2]$ for exactly three elements
$y$ of $]{-}1,0[\, \cup\,]1,2[$.} For this purpose, we consider $y\mapsto
\partial_x F(x,y)$. We prove that it vanishes exactly three times on
$]{-}1,0[ \,\cup\,]1,2[ $. For fixed $x\in\,]0,1/5[$, this function is indeed
negative in $0$ and $-1$ while it is positive in $-1/2$. The sign is
also different for $y=1$ and $y=2$ so that we have found the three
zeros of $y\mapsto\partial_x F(x,y)$. But as explained in the previous
step, for $y\in\,]{-}1,0[\, \cup\,]1,2[$ being a maximum of $F(\cdot,y)$ is
exactly the same as having zero derivate.

\step{Therefore, any $x\in\,]0,1/5[$ gives rise to the maximum of
$F(\cdot,y)$ for three different $y\in[-1,0]\cup[1,2]$.} Hence, there
are $y_1,y_2,y_3$ such that $\psi(y_i)=y_i^4-F(x,y_i)$ for $i=1,2,3$.
Notice that $x$ is in the convex hull of these points because $y_1$ is
close to $-1$, $y_2$ is close to $0$ and $y_3$ close to $1$. Hence,
there exists a martingale transport plan $\pi$ concentrated on
$[0,1/5]\times([-1,0]\cup[1,2])$ such that $\pi_x$ is supported on
$\{
y_1,y_2,y_3\}(x)$ with positive $\mu$-probability. Moreover, it follows
from the explanations above that this martingale transport plan is
optimal. Namely, \eqref{defsup} holds $\pi$-a.s. Hence, we have proved
that the bound $k=3$ of Theorem~\ref{LessThanK} is sharp in the case
$c(x,y)=(y-x)^4$.
%\begin{figure}[ht]
%\begin{center}
%\includegraphics[width=10cm]{uniform}
%\caption{Illustration.}\label{uniform_fig}
%\end{center}
%\end{figure}

%s7.4 #&#
\subsection{The Hobson--Neuberger cost function and its converse}

As mentioned in the \hyperref[sec:intro]{Introduction}, Hobson and Neuberger \cite{HoNe11}
study the case $c(x,y)= -|y-x|$, motivated by applications in
mathematical finance. They identify the minimizer $\pi_{\hn}$ based on
a construction of the maximizers for the dual problem. Here, some
conditions on the underlying measures are necessary; an example in
\cite{BeHePe11},
Proposition~5.2, shows that the dual maximizers need not
always exist. %\marginpar{The text here is mediocre, I will repair this
%later.}
Based on Lemma~\ref{GlobalLocal} we partly recover their result.
Throughout this part, we will only deal with the case of a continuous
starting distribution $\mu$ (see Remark~\ref{not_ok} on this hypothesis).

%th7.3 #&#
\begin{them}\label{HobsonNeuberger}
Assume that $\mu$ and $\nu$ are in convex order and that $\mu$ is continuous.
There exists a unique optimal martingale transport plan $\pi_{\hn}$ for
the cost function $c(x,y)= -|y-x|$.

Moreover, there exist two nondecreasing functions $T_1,T_2\dvtx \R\to\R$
such that $T_1(x)\leq x\leq T_2(x)$ and $\pi_{\hn}$ is concentrated on
the graphs of these functions.
%\footnote{We remark that given measures $\mu$, $\nu$ in Convex Order
%there may exist more then one coupling which has the same monotonicity
%properties as the Hobson-Neuberger coupling $\pi_{\hn}$: Let $\mu=
%\tfrac12 \delta_{-1}+ \tfrac12\delta_{+1}$ and let $\nu=\tfrac14
%\lambda_{|[-2,2]}.$. Then there are many ways to pick a partition of
%$[-2, 2] $ into disjoint intervals $I_1< I_2< I_3<I_4$ such that $-1
%\in I_2, +1\in I_4, \lambda(I_1\cup I_3)=\lambda(I_2\cup I_4)=2$ and
%that $$ \pi= 1/4 \delta_{-1} \times\lambda_{|I_1\cup I_3} \cup1/4
%\delta_{+1} \times\lambda_{|I_2\cup I_4}$$ is a martingale coupling
%between $\mu$ and $\nu$. Each such coupling has the same monotonicity
%properties as the Hobson-Neuberger coupling. Using the same idea it is
%not hard to see that this nonuniqueness pertains if $\mu$ is replaced
%by the continuous measure $\tfrac12\lambda_{|[-1,1]}$. }
\end{them}

A similar behavior holds for the cost function $c(x,y)=|y-x|$ built on
the absolute value $h\dvtx x\mapsto|x|$. We have learned about the
structure of the optimizer for this cost function from D. Hobson and
M. Klimmek \cite{HoKl12}.
Recall that $\Gamma_x=\{y\dvtx (x,y)\in\Gamma\}$ for $\Gamma\subseteq\R^2$.

%th7.4 #&#
\begin{them}\label{RHN}
Assume that $\mu$ and $\nu$ are in convex order and that $\mu$ is continuous.
There exists a unique optimal martingale transport plan $\pi_{\abs}$
for the cost function $c(x,y)= |y-x|$.

Moreover, there is a set $\Gamma$ such that $\pi_{\abs}$ is
concentrated on $\Gamma$ and $|\Gamma_x|\leq3$ for every $x\in\R$.
More precisely, $\pi_{\abs}$ can be decomposed into $\pi_{\mathrm{stay}}+\pi
_{\mathrm{go}}$ where $\pi_{\mathrm{stay}}=(\id\otimes\id)_\#(\mu\wedge\nu)$ (this
measure is concentrated on the diagonal of $\R^2$) and $\pi_{\mathrm{go}}$ is
concentrated on $\graph(T_1)\cup\graph(T_2)$ where $T_1, T_2$ are real
functions.
\end{them}

The ``combinatorial core'' of the proofs to Theorems \ref
{HobsonNeuberger} and \ref{RHN} is contained in the following
lengthy but simple lemma.

%le7.5 #&#
\begin{lem}\label{HNTechCore}
Let $x,y^-,y,^+,y'\in\R$ such that $y^-< x,y'< y^+$. Pick $\lambda$
such that $\lambda y^++(1-\lambda) y^-= y'$.
For $x'\in\R$ we want to compare the quantities
\begin{eqnarray*}
A &:=&\lambda\bigl|x-y^+\bigr|+ (1-\lambda) \bigl|x-y^-\bigr|+ \bigl|x'-y'\bigr|,
\\
B&:=&\lambda\bigl|x'-y^+\bigr|+ (1-\lambda) \bigl|x'-y^-\bigr|+
\bigl|x-y'\bigr|.
\end{eqnarray*}
\begin{longlist}[(1)]
\item[(1)] Assume that $y'< x$. Then there exists $x_0\in\,]y^-,y'[$ such
that $(A-B)$ seen as a function of $x'$ exactly vanishes at $x_0$ and
$x$, is strictly positive outside $[x_0,x]$ and strictly negative in $]x_0,x[$.
\[
\begin{array} {c|c@{\hspace*{4pt}}c@{\hspace*{4pt}}c@{\hspace*{4pt}}c@{\hspace*{4pt}}
c@{\hspace*{4pt}}c@{\hspace*{4pt}}c}
x' &-\infty& y^-
&x_0&y'&x&&+\infty
\\
\hline
(A-B) \bigl(x'\bigr)& &+&0 &
-&0&+&
\end{array}.
\]

%\begin{eqnarray}
%x'\in\,]x_0,x[& \ \Longleftrightarrow\ A < B\\
%x'\in\{x_0,x\}& \ \Longleftrightarrow\ A = B\\
%x'\notin[x_0,x]& \ \Longleftrightarrow\ A > B.
%\end{eqnarray}.

\item[(2)] Assume that $y'> x$. Then there exists $x_1\in\,]y', y^+[$ such
that $(A-B)$ vanishes if $x'\in\{x_1,x\}$, is strictly positive outside
$[x,x_1]$ and strictly negative in $]x,x_1[$:
\[
\begin{array} {c|c@{\hspace*{4pt}}c@{\hspace*{4pt}}c@{\hspace*{4pt}}
c@{\hspace*{4pt}}c@{\hspace*{4pt}}c@{\hspace*{4pt}}c} x' &-\infty& &
x&y'&x_1&y^+&+\infty
\\
\hline(A-B) \bigl(x'\bigr)& &+&0&-\phantom{'}&0\phantom{_1}&+\phantom{^+}&
\end{array} %
. %
\]

%\begin{eqnarray}
%x'\in\,]x,x_1[& \ \Longleftrightarrow\ A < B\\
%x'\in\{x,x_1\}& \ \Longleftrightarrow\ A = B\\
%x'\notin[x,x_1]& \ \Longleftrightarrow\ A > B.
%\end{eqnarray}

\item[(3)] Assume that $y'= x$. Then $(A-B)$ is nonnegative and vanishes
exactly in~$x$.
\[
\begin{array} {c|c@{\hspace*{4pt}}c@{\hspace*{4pt}}c@{\hspace*{4pt}}c
@{\hspace*{4pt}}c@{\hspace*{4pt}}c@{\hspace*{4pt}}c} x' &-\infty&y^- &
x=y'&y^+&+\infty
\\
\hline(A-B) \bigl(x'\bigr)& &+\phantom{^-}&0 &+\phantom{^+}&
\end{array} %
. %
\]
%\begin{tikzpicture}
%\tkzTabInit{$x'$/1,$A-B (x')$/1}{$-\infty$,$x$,$+\infty$}
%\tkzTabLine{,+,z,+,}
%\end{tikzpicture}
%\begin{eqnarray}
%x' = x& \ \Longleftrightarrow\ A = B\\
%x'\neq x& \ \Longleftrightarrow\ A > B.
%\end{eqnarray}
\end{longlist}
\end{lem}

\begin{pf}
Consider the function
\[
f(t)= \lambda\bigl|t-y^+\bigr|+(1-\lambda) \bigl|t-y^-\bigr|-\bigl|t-y'\bigr|.
\]
Then $A> B$ is equivalent to $f(x) > f(x')$ and $A= B$ is equivalent to
$f(x) = f(x')$.

The behavior of the function $f$ is easy enough to understand. On the
intervals $]{-}\infty, y^-]$, $[y^+,\infty[$, the function is zero. On
the interval $[y^-, y']$ it increases linearly from $0$ to $2\lambda
(1-\lambda) (y^+-y^-)$. On the interval $[ y', y^+]$ it decreases
linearly from $2\lambda(1-\lambda) (y^+-y^-)$ to $0$.

The above assertions are simple consequences of this behavior.
Moreover, it is easy to calculate $x_0, x_1$ explicitly. For instance,
in the case $y'< x$ pick $t\in\,]0,1[$ such that $x=y'+t(y^+-y')$. Then
$x_0= y' + t(y^--y')$.
\end{pf}

\begin{pf*}{Proof of Theorem~\ref{HobsonNeuberger}}
Pick $\Gamma$ according to Lemma~\ref{GlobalLocal} and $(x,y^-)$, $
(x,y^+)$, $(x',y')\in\Gamma$, with $y^-< y' < y^+$. Then it cannot
happen that
%
%e23 #&#
\begin{equation}
\label{2BadConf} y'\leq x' < x\quad \mbox{or}\quad
x<x'\leq y'.
\end{equation}
%
%Without loss of generality we explain why the second case is not
%possible.
Indeed, choosing $\lambda\in\,]0,1[$ and $\alpha$, respectively,
$\alpha
'$ as in the proof of Theorem~\ref{OptForExp}, we find that an
improvement is possible if
\begin{eqnarray*}
&&-\lambda\bigl|x-y^+\bigr|- (1-\lambda) \bigl|x-y^-\bigr|- \bigl|x'-y'\bigr| > -
\lambda\bigl|x'-y^+\bigr|- (1-\lambda) \bigl|x'-y^-\bigr|\\
&&\qquad{}-\bigl|x-y'\bigr|.
\end{eqnarray*}
This inequality holds in the just mentioned cases by Lemma~\ref{HNTechCore}.

Consider the set $A$ of points $a$ such that $\Gamma_a$ contains more
than two points and assume by contradiction that this set is
uncountable. According to Lemma~\ref{AccumulatingBadPoints}, there is
an accumulation effect at some $a\in A$ together with $b^-, b, b^+\in
\Gamma_a$ in the order $b^-< b< b^+$. (Without loss of generality, one
may assume $b\leq a$.) In particular, Lemma~\ref{AccumulatingBadPoints}
provides $(a_0, b_0^-),(a_0, b_0^+)\in\Gamma$ such that $a< a_0<
b_0^+$ and $b_0^-< b$. We have settled the first forbidden situation of
\eqref{2BadConf} for $(x,y^-)=(a_0,b_0^-)$, $(x,y^+)=(a_0,b_0^+)$ and
$(x',y')=(a,b)$, which provides the desired contradiction. Hence, $A$
is countable and $\mu(A)=0$. It follows that one can assume $|\Gamma
_a|\leq2$ for every $a\in\R$.

We may thus assume that there exist $T_1$ and $T_2$ from $\proj
^x(\Gamma
)$ to $\R$ such that $\Gamma_x=\{T_1(x), T_2(x)\}$ where $T_1(x)\leq
x\leq T_2(x)$ for $\mu$-almost every $x\in\proj^x(\Gamma)$. It remains
to show that $T_1$ and $T_2$ are monotone. Let $x,x'\in\R$ with $x<
x'$. We necessarily have $T_2(x)\leq T_2(x')$ since the opposite
inequality leads to the second forbidden inequality in \eqref{2BadConf}
taking $y^-= T_1(x), y'= T_2(x')$ and $y^+= T_2(x)$. The monotonicity
of $T_1$ is established in the same way. %Finally the nondecreasing
%functions $T_1$ and $T_2$ can be extended from the domain $\proj^1(
%\Gamma)$ to $\R$ keeping the monotonicity property

%

It remains to show that the optimizer is unique. Due to the linear
structure of the optimization problem the set of solutions is convex.
Hence, Lemma~\ref{half-sum} applies.
\end{pf*}
%
%\marginpar{It should be a separate comment (probably in the
%introduction) that ``two-splitting'' implies uniqueness of the
%optimizer.}

%re7.6 #&#
\begin{rem}\label{not_ok}
If $\mu$ is not continuous, there may be more than one minimizer. This
is the case, for example, if $\mu$ and $\nu$ are chosen as in Section
\ref{non_uniq}. In fact, if $h$ is an even function then for the cost
function $c(x,y)=h(y-x)$ (e.g., $x\mapsto-|y-x|$) every martingale
transport plan is optimal. Hence, it seems that it is not directly
possible to define the Hobson--Neuberger transport plan for a general
starting distribution $\mu$
in an unambiguous way.
\end{rem}

\begin{pf*}{Proof of Theorem~\ref{RHN}}
Let $\pi$ be an optimal martingale transport plan.
Pick $\Gamma$ according to Lemma~\ref{GlobalLocal} and $(x,y^-)$, $
(x,y^+)$, $(x',y')\in\Gamma$, with $y^-< y' < y^+$. Then it cannot
happen that
%
%e24 #&#
\begin{equation}
\label{3BadConf} x' < x\leq y' \quad\mbox{or}\quad
y'\leq x<x' \quad\mbox{or}\quad x' \notin
\bigl[y^-, y^+\bigr].
\end{equation}
%
%Without loss of generality we explain why the second case is not
%possible.
Indeed, choosing $\lambda\in\,]0,1[$, $\alpha$ and $\alpha'$ as in the
proof of Theorem~\ref{OptForExp} above we find that an improvement of
$\alpha$ by $\alpha'$ is possible if
\[
\lambda\bigl|x-y^+\bigr|+ (1-\lambda) \bigl|x-y^-\bigr|+\bigl |x'-y'\bigr| >
\lambda\bigl|x'-y^+\bigr|+ (1-\lambda) \bigl|x'-y^-\bigr|+\bigl|x-y'\bigr|.
\]
Indeed, this inequality holds in the just mentioned cases by Lemma~\ref
{HNTechCore}. Note in particular that one of the forbidden cases of
\eqref{3BadConf} occurs if $x\neq x'$ and $x= y'$. This will be crucial
in the following argument which establishes that as much mass as
possible is transported by the identity mapping. (Roughly speaking, the
following is forbidden: Some mass goes from $x$ to $y^-$ and $y^+$
while some mass goes from $x'$ to $y'=x$.)

Set $\pi_0= \pi|_{\Delta}$, where $\Delta$ is the diagonal $\{
(x,y)\in\R
^2\dvtx x=y\}$ and $\bar\pi= \pi-\pi_0$, let $\rho$ be the projection of
$\pi_0$ onto the first (or the second) coordinate. As $\rho\leq\mu$
and $\rho\leq\nu$, we have $\rho\leq\mu\wedge\nu$. We want to prove
that $\rho=\mu\wedge\nu$, that is, $\pi_0$ is $(\id\otimes\id
)_\#(\mu
\wedge\nu)$. Let us define the reduced measures $\bar\mu= \mu-\rho
, \bar
\nu= \nu-\rho$ and $\kappa= \mu\wedge\nu-\rho$. Note that $\bar
\pi\in
\M(\bar\mu,\bar\nu)$ and that $\bar\pi$ is concentrated on $\bar
\Gamma=
\Gamma\setminus\Delta$.
Hence, we have the following:
\begin{itemize}
\item For $\bar\mu$-almost every $a$, there exist $b^-$ and $ b^+$ such
that $a\in\,] b^- ,b^+[$ and $(a,b^-), (a, b^+)\in\bar\Gamma$.
\item For $\kappa$-almost every $b$, there exists some $a\neq b$ such
that $(a, b) \in\bar\Gamma$.
\end{itemize}
As $\kappa\leq\bar\mu$, we conclude that $\kappa$-almost every real
number satisfies both of these conditions. Thus, for $\kappa$-almost
every $x$ there exist $y^-, y^+$ and $x'$ such that the points
$(x,y^-), (x,y^+)$ and $(x',x)$ are included in $\bar\Gamma$ and one
has $x'\neq x$ and $x\in\,]y^-,y^+[$. This coincides with one of the
forbidden situations of \eqref{3BadConf}. Hence, $\kappa$ has mass $0$
and $\pi_0=(\id\otimes\id)_\#(\mu\wedge\nu)$ as claimed above.

Our next goal is to establish that, removing countably many points if
necessary, we have $|\bar\Gamma_x|\leq2$ for every $x\in\R$. Indeed,
if this is not true, then there exist $a, b', b^-$ and $b^+$ with $b^-<
b <b^+\in\bar\Gamma_a$ to which the assertion of Lemma~\ref
{AccumulatingBadPoints} applies. We know that $b< a $ or $a< b$; assume
without loss of generality that $a<b$. But then there exist $a'$ with
$b^-< a'<a$ and $b'$ with $a < b' < b$ such that $(a', b')\in\Gamma$.
This contradicts \eqref{3BadConf} (with $x=a$, $y^-=b^-$, $y^+=b^+$,
$x'=a', y'=b'$).

It remains to establish that there exists at most one optimizer. For
optimal transports $\pi$, the static part $\pi_0=\pi|_\Delta$ equals
$(\id\otimes\id)_\#(\mu\wedge\nu)$. Hence, the reduced measure
$\bar\pi
=\pi-\pi_0$ is a minimizer of the martingale transport problem between
$\bar\mu=\mu-\mu\wedge\nu$ and $\bar\nu=\nu-\mu\wedge\nu$.
Note that
$\bar\mu\wedge\bar\nu=0$ so that the optimal martingale couplings are
concentrated on two Borel graphs. We conclude by Lemma~\ref{half-sum}.
\end{pf*}

%re7.7 #&#
\begin{rem}
Exactly as in Remark~\ref{not_ok}, the hypothesis that $\mu$ is
continuous is needed to prove uniqueness of the optimizer; $\pi_{\abs}$
is not well defined otherwise. %Possibly one of the optimal transport
%plan can be distinguished adding a property we wish for $\pi_{\abs}$.
\end{rem}
%sA #&#
\begin{appendix}\label{app}
%sB #&#
\section{A converse to the variational lemma}\label{appA}

In this section, we prove that the optimality criterion given in the
variational Lemma~\ref{GlobalLocal} is not only necessary but also
\emph
{sufficient} provided that the cost function is assumed to be bounded
and continuous. We conjecture that these regularity assumptions can be
relaxed. Before we state the variational lemma, let us give a
definition. %\marginpar{N: The definition is actually also useful for
%the first part of the variational lemma. We could also introduce a
%notion of ``globally optimal''\ldots}

%deB.1 #&#
\begin{defi}
Let $c$ be a cost function with values in $\R$. We say that a Borel set
$\Gamma$ is \emph{finitely optimal} for $c$ if for every measure
$\alpha
$ on $\R\times\R$ with $|\operatorname{spt}(\alpha)|<\infty$ and
$\operatorname{spt}(\alpha)
\subseteq\Gamma$ and every competitor $\alpha'$ of $\alpha$ we have
$\int c \,\d\alpha\leq\int c \,\d\alpha'$.
\end{defi}

As $c$ only takes finite values, the integrals exist.

%leB.2 #&#
\begin{lem}[(Variational lemma, part II)]\label{LocalGlobal}
Assume that $\mu, \nu\in\mathcal{P}$ are in convex order and that
$c\dvtx \R^2\to\R
$ is a continuous bounded cost function.
Let $\pi\in\M(\mu, \nu)$. It there exists a finitely optimal set
$\Gamma$ such that $\pi(\Gamma)=1$, then $\pi$ is an optimal martingale
transport plan.
\end{lem}

The strategy of our proof will be to establish \emph{dual maximizers}
(see Section~\ref{dual_pro}). Such dual maximizers do not exist in
general as follows from \cite{BeHePe11}, Proposition~4.1. However, the
following simple lemma allows us to reduce the martingale transport
problem to ``irreducible components.'' It turns out that on each of
these components it \emph{is} possible to construct the desired dual
maximizers.\footnote{Roughly speaking, the construction given in \cite{BeHePe11},
Proposition~4.1, uses an infinite number of such irreducible
components. While it is possible to construct optimizers on each
component, it turns out to be impossible to glue them together.}

%sB.1 #&#
\subsection{Irreducible decompositions}
Let us now introduce some of the necessary vocabulary.

%deB.3 #&#
\begin{defi}\label{def_irred}
Let $\mu,\nu$ be elements of $\mathcal{M}$ such that $\mu\leqc\nu
$. We say
that $(\mu,\nu)$ is irreducible if there exists an open interval $I$
(bounded or not) such that $\mu(I)$ and $\nu(\bar{I})$ have the total
mass and $u_\mu<u_\nu$ on $I$.
\end{defi}

Note that on $\R\setminus I$ we have $u_\mu=u_\nu$ so that $I$ is
exactly $\{u_\mu<u_\nu\}$.

%thB.4 #&#
\begin{them}[{[Decomposition of $(\mu,\nu)$ into irreducible
components]}]\label{yo1}%\marginpar{In fact $\nu_k=S^\nu(\mu_k)$. It is
%certainly possible to shorten the proof with what we know on shadows
%\ldots
%
%M: I would prefer to keep the proof self-contained. This decomposition
%is really relatively simple and people might be interested to
%understand it without learning to much about shadows. }
Let $\mu,\nu$ be elements of $\mathcal{M}$ such that $\mu\leqc\nu
$. Let
$(I_k)_k$ be the (in essence unique) sequence of disjoint open
intervals such that $\bigcup_k I_k=\{u_\mu<u_\nu\}$ and write $F$ for
the closed set $\R\setminus\bigcup_k I_k$.
Set $\mu_k=\mu|_{I_k}$ and define $\eta=\mu|_F$ such that $\mu
=(\sum_k
\mu_k)+\eta$.

There exists a unique decomposition $\nu=(\sum_k \nu_k)+\upsilon$ such
that $\mu_k\leqc\nu_k$ for each $k$ and $\eta\leqc\upsilon$.

For this decomposition $\eta=\upsilon$ and $(\mu_k,\nu_k)$ is
irreducible with $\{ u_{\mu_k} < u_{\nu_k} \}=I_k$. Moreover, any
martingale transport plan $\pi\in\M(\mu, \nu)$ can be decomposed in
the form
%
%eB.1 #&#
\begin{equation}
\label{PiDec} \pi= \biggl(\sum_k
\pi_k \biggr)+\pi_F,
\end{equation}
where $\pi_k$ is a martingale transport from $\mu_k$ to $\nu_k$. This
decomposition is unique and $\pi_F=(\id\otimes\id)_\#\eta$.
\end{them}
Note that the measure $\eta\wedge\nu_k$ does not necessarily vanish.
\begin{pf*}{Proof of Theorem \ref{yo1}}
%\marginpar{see Referee's report about Lemma~8.5 and find a good
%English sentence.}
To establish the uniqueness part, we need two auxiliary results.
%The following lemma from \cite[Lemma~4.2]{BeHePe11} and Lemma~\ref{Irr1Step} will be useful both for the uniqueness part.

%leB.5 #&#
\begin{lem}\label{SepProblems}
Assume that $\mu, \nu$ are elements of $\mathcal{P}$ and let $\pi
\in\M(\mu,\nu
), s\in\R$.
The following are equivalent:
\begin{longlist}[(ii)]
\item[(i)]$\pi ( ]{-}\infty, s[ \,\times\,]{-}\infty, s] \cup\{(s,s)\}
\,\cup\,]s,\infty[ \,\times\,[s, \infty[  )=1$.
\item[(ii)]$u_\mu(s)=u_\nu(s)$.
\end{longlist}
Consequently, as \textup{(ii)} does not depend on $\pi$, if \textup{(i)} holds for one
measure in $\M(\mu, \nu)$, then it applies to \emph{all} elements
of $\M
(\mu, \nu)$.
\end{lem}

\begin{pf} This is essentially \cite{BeHePe11}, Lemma~4.2; the only
difference is that the formulation in \cite{BeHePe11} refers to the
function $u_\mu^+(x):= \int(y-x)_+ \,\d\mu(y)$ rather than to $u_\mu$.
However, the proof goes through in the same way if $(\cdot )_+$ is replaced
by $|\cdot|$.
\end{pf}

We record the following consequence.

%leB.6 #&#
\begin{lem}\label{Irr1Step}
Let $I$ be an open interval such that $u_\mu=u_\nu$ on the boundary
of $I$.
Let $\mu_I$ be $\mu|_I$ and $\pi$ be a transport plan of $\M(\mu
,\nu)$.
Set also $\nu_I:=\proj^y_\#(\pi|_{ I\times\R})$.

The measure $\nu_I$ is concentrated on $\bar{I}$ and does not actually
depend on the particular choice of $\pi$. Moreover, we have $u_{\nu
_I}-u_{\mu_I}=0$ on $\R\setminus I$ and $u_{\nu_I}-u_{\mu_I}=u_\nu
-u_\mu$ on $I$.
\end{lem}

\begin{pf}
Pick $\pi\in\Pi_M(\mu, \nu)$ and apply Lemma~\ref{SepProblems} to
every $s\in\partial I$. Then
%
%eB.2 #&#
\begin{equation}
\label{SomeRestriction} \pi \bigl((I\times\bar I) \cup(\R\setminus I)^2
\bigr)=1.
\end{equation}
Set $\pi_I:=\pi|_{I\times\R}$. Relation \eqref{SomeRestriction}
asserts that no mass of $\mu$ is moved from $\R\setminus I$ to $I$ and
that the mass of $I$ is transported into $\bar{I}$. Thus, $\mu_I\leqc
\nu
_I=\proj^y_\# \pi_I$ (so that the two measures have the same integral
against linear functions) and $\nu_I$ is concentrated on $\bar I$. It
follows directly from the definition of the potential functions that
$u_{\nu_I}=u_{\mu_I}$ on $\R\setminus I$. Applying similar arguments
to $\mu|_J$ and $\nu_J=\proj^y_\#\pi|_{J\times\R}$ for every (closed)
connected component $J$ of $\R\setminus I$ and recalling that $\alpha
\mapsto u_\alpha$ is linear, we obtain
$u_{\mu-\mu_I}=u_{\nu-\nu_I}$ on $I$. Hence, $u_{\nu_I}-u_{\mu
_I}=u_\nu
-u_\mu$ holds on this interval.
\end{pf}
We first prove the existence of some decomposition of $\nu$. We fix
some $\pi\in\M(\mu,\nu)$ and for every $k$, we define $\mu_k$ and
$\nu
_k$ as the marginals of $\pi_k:=\pi|_{I_k\times\R}$. Denote by
$\eta
,\upsilon$ the marginals of $\pi_F:=\pi|_{F\times\R}$. The transport
plans $\pi_k$ and $\pi_F$ are martingale transport plans so that $\mu
_k\leqc\nu_k$ and $\eta\leqc\upsilon$.

For the uniqueness part, we take for $i=1,2$ a decomposition $(\nu
^i_k)_k,\upsilon^i$ of $\nu$ such that $\mu_k\leqc\nu^i_k$ and
$\eta
\leqc\upsilon^i$. According to Example~\ref{delta_b}, there exists a
martingale transport plan $\pi^i$ that transports every $\mu_k$ on
$\nu
^i_k$ and $\eta$ on $\upsilon^i$. But the $\mu_k$'s are concentrated on
disjoint intervals so that $\nu^i_k=\proj^y_\#\pi^i|_{I_k\times\R}$
and $\upsilon^i=\proj^y_\#\pi^i|_{F\times\R}$. It follows from Lemma~\ref{Irr1Step} that $\proj^y_\#\pi|_{I_k\times\R}$ does not depend on
the particular choice of $\pi\in\M(\mu,\nu)$. Hence, $\nu^1_k=\nu^2_k$
for every $k$ and $\upsilon^1=\nu-\sum_k \nu^1_k=\upsilon^2$.

Let us now prove the properties listed in the second part of Theorem~\ref{yo1}. We continue to use the notation of the existence part ($\pi,
\pi_k, \pi_F, \mu_k, \nu_k, \eta$ and $\upsilon$). As a
consequence of
Lemma~\ref{Irr1Step} (applied to $\mu,\nu$ and $I_k$), we have the following:
\begin{longlist}[(ii)]
\item[(i)] $\nu_k$ is concentrated on $\bar{I_k}$;

\item[(ii)]$u_{\nu_k}-u_{\mu_k}$ is $0$ on $\R\setminus I_k$ and
$u_\nu
-u_\mu$ on $I_k$.
\end{longlist}

As the $I_k$'s are disjoint, we have
\begin{eqnarray*}
u_{\sum\nu_k}-u_{\sum\mu_k}=\sum_k(u_{\nu_k}-u_{\mu_k})=
\cases{u_\nu-u_\mu,&\quad $\mbox{on $\displaystyle\bigcup
_k I_k$,}$\vspace*{2pt}
\cr
0=u_\nu-u_\mu,&\quad $
\mbox{on $F=\displaystyle\bigcap_k \bar{I_k}$.}$}
\end{eqnarray*}
Hence,
\[
u_\upsilon=u_\nu-u_{\sum\nu_k}=u_\mu-u_{\sum\mu_k}=u_\eta
\]
on the whole real line. Thus, we have $\upsilon=\eta$. The fact that
$(\mu_k, \nu_k)$ is irreducible and $\{u_{\mu_k}<u_{\nu_k}\}=I_k$
follows directly from Definition~\ref{def_irred} and what has been
proved so far. Finally, concerning $\pi$, note that $\pi=(\sum_k \pi
_k)+\pi_F$ where $\pi_k$ has marginals $\mu_k$ and $\nu_k$. As $\pi_F$
is a martingale transport plan from $\eta$ to $\upsilon=\eta$ it is the
identical transport plan $(\id\otimes\id)_\#\eta$. The uniqueness of
the decomposition \eqref{PiDec} follows from the fact that the $\mu
_k$'s are concentrated on disjoint intervals.
\end{pf*}

As a consequence of Theorem~\ref{yo1}, we have the following straightforward corollary:

%coB.7 #&#
\begin{cor}[(Reducing the transport problem)]\label{yo2}
Let $\mu,\nu$ be elements of $\mathcal{M}$ and $\mu\leqc\nu, \pi
\in\M(\mu, \nu
)$ with decompositions $(\mu_k)_k,(\nu_k)_k,\eta$, $\pi=(\sum_k
\pi
_k)+(\id\otimes\id)_\#\eta$ as in Theorem~\ref{yo1}.
Let $c$ be a cost function such that the martingale transport problem
satisfies the sufficient integrability condition and leads to finite
costs. Then the transport $\pi$ is optimal if and only if every $\pi_k$
is optimal for the transport problem between $\mu_k$ and $\nu_k$.
\end{cor}

%We have thus proved the following:
%\begin{pro}
%Let $\mu$ and $\nu$ be two finite measures in convex order and
%$(I_k)_k$ an at most countable family of disjoint open intervals with
%$u_\mu< u_\nu$ on $O=\bigcup_k I_k$ and $u_\mu=u_\nu$ on $F=\R
%\setminus O$.
%Then every martingale transport plan $\pi$ admits a unique
%decomposition $\pi_F+\sum_k \pi_k$, such that $\pi_k:=\pi_{I_k\times
%I_k}$ is a martingale transport plan that transports $\mu_k:=
%\mu_{I_k}$ to $\nu_k:=\nu_{I_k}$ and $\pi_F$ is a marginal transport
%plan from $\mu_F$ to $\nu_F$. Moreover $\pi_{F}$ equals $(\id\otimes
%\id)_\#\mu_F$.

%Assume that there exist martingale transport plan which leads to
%finite costs.
%Then a measure $\pi$ is optimal for the martingale transport problem
%if and only if every $k$ the measure $\pi_k:=\pi_{I_k\times I_k}$ is a
%martingale transport plan that transports $\mu_k$ to $\nu_k$ in an
%optimal way.
%\end{pro}

Recall that in Lemma~\ref{LocalGlobal}, the main result of this
section, one is assuming that some particular finitely optimal set
exists for the cost $c$. We will need several times to assume that this
set satisfies some additional properties that we introduce in the next
definition. Recall for the sequel that for a set $G\subseteq\R^2$ we
write $G_x=\{y\dvtx (x,y)\in G\}$ and denote the projections of $G$ by
$X_G$ and $Y_G$, respectively.

%deB.8 #&#
\begin{defi}%\marginpar{Is my interpretation of \Reg and \Red correct?
%I would write (Reg.) and (Irred.). Maybe we can write Red for
%``Reduced''}
Let $I$ be an open interval. A set $G$ satisfies the regularity
property on $I$ if $G\subseteq I\times\bar{I}$ and for every
$x\in I$ we have
$G_x=\varnothing$ or $G_x=\{x\}$ {or} $x\in\,]\inf G_x,\sup G_x[$.

A set $G$ satisfies the irreducibility property on $I$ if
$G\subseteq I\times\bar{I}$ and for every $y\in I$ there exist $x\in
I$ and $y^-, y^+\in G_x$ so that $y^-<y< y^+$.
\end{defi}

Note that if $G$ is irreducible on $I$, we can apply this property to
points $y\in I$ close to the boundary of $I$. Therefore, we have
$I=\accentset{\circ}{\Conv(Y_G)}$.

%leB.9 #&#
\begin{lem}\label{yo3}
Let $\mu,\nu$ be elements of $\mathcal{P}$ such that $(\mu,\nu)$
is irreducible
with $I=\{u_\mu<u_\nu\}$. Let $c$ be a cost function. Let moreover $G$
be a finitely optimal set and $\pi$ a martingale transport plan with
$\pi(G)=1$. Then there exists a Borel set $G'\subseteq G\cap(I\times
\bar{I})$ that is regular and irreducible on $I$ and such that $\pi
(G')=1$. Moreover, $G'$ is finitely optimal.
\end{lem}

%Using the decomposition \eqref{PiDec} provided in Theorem~\ref{yo1} we
%may work under the assumption that the martingale %transport problem
%associated to $\mu$ and $\nu$ gives rise to only one irreducible
%component, i.e.\ an open interval $I$ satisfying
%\begin{eqnarray}
%\label{OldRed}%\tag{irred.}
%\begin{split}
%&\mu(I)=\nu(I) =1\\
%& u_\nu-u_\mu>0 \mbox{on $I$}, u_\nu-u_\mu= 0 \mbox{on $\partial I$.}
%\end{split}
%\end{eqnarray}
%Assume that $\pi\in\M(\mu, \nu)$ satisfies the assumptions of Lemma~\ref{LocalGlobal} and pick a monotonicity set (see Definition~\ref{defi_monoto}) $\Gamma$ on which $\pi$ is concentrated.

\begin{pf}
Let $G$ and $\pi$ be as in the statement. Since $\pi$ is a martingale
transport plan we find that for $\mu$-almost all $x\in I$
\[
x\in\accentset{\circ} {\Conv(G_x)}\quad \mbox{or}\quad \{x\}= G_x.
\]
Erasing a negligible set if necessary, we can assume that the regularity property is
satisfied on~$I$. Let $G'$ be the resulting set.
Assume by contradiction that $G'$ does not satisfy the irreducibility property on $I$. Hence,
there exists $y\in I$ such that for every $x\in I$, the set $G_x$ is
included in $]{-}\infty,y]$ or in $[y,+\infty[$. By regularity,
$G_x\subseteq\,]{-}\infty,y]$ if $x\leq y$ and $G_x\subseteq[y,+\infty[$
otherwise. Hence, $\pi(]{-}\infty,y]^2\cup[y,+\infty[^2)=1$ so that
$u_\mu
(y)=u_\nu(y)$, according to Lemma~\ref{SepProblems}. But $y\in I=\{
u_\mu
<u_\nu\}$, which yields a contradiction. Therefore, the set $G'$
is regular and irreducible on $I$. Each subset of $G$ is finitely
optimal, hence so is $G'$.
\end{pf}

%sB.2 #&#
\subsection{Existence of dual maximizers \texorpdfstring{$\varphi, \psi, \Delta$}{varphi,psi,Delta} on an
irreducible component}\label{DualMaxSection}

In this paragraph, we aim to prove Proposition~\ref{DualMax}. The cost
function $c$, the sets $\Gamma\subseteq\R^2$ and $I$ are fixed
accordingly throughout Sections~\ref{DualMaxSection} and \ref{yaha}.

%prB.10 #&#
\begin{pro}\label{DualMax}
Assume that $c\dvtx \R\to\R$ is continuous and
let $\Gamma$ be a finitely optimal set that is regular and irreducible on some open interval $I$.

Then there exist upper semicontinuous functions $\phi\dvtx I\to[-\infty,
\infty[ ,\psi\dvtx J=\Conv(Y_\Gamma) \to[-\infty, \infty[$ and a measurable
function $\Delta\dvtx I\to\R$ such that
\[
\phi(x)+\psi(y) + \Delta(x) (y-x) \leq c(x,y)
\]
for all $x\in I, y\in J$, with equality holding whenever $(x,y)\in
\Gamma$.
\end{pro}

We emphasize that the functions appearing in Proposition~\ref{DualMax}
can be interpreted as a sort of maximizer for the dual problem
described in Section~\ref{dual_pro}.

%Some more preparations are necessary before we can give the proof of
%Proposition~\ref{DualMax}.
Throughout Section \ref{DualMaxSection}, we will work under the
assumptions of Proposition~\ref{DualMax}; some preparations will be
necessary to establish the result.

%deB.11 #&#
\begin{defi}
Let $\psi$ be a function from a subset of $\R$ into $\R$ and let $G$ be
a subset of $\R\times\R$ such that $\psi$ is defined on $Y_G=\proj^y
(G)$. The function $\psi$ is called \emph{$G$-good} if the following
holds true:

For every $x\in X_G=\proj^x (G)$, there exists an affine function $y
\mapsto a_x(y)$ such that
%
%eB.3 #&#
\begin{equation}
a_x(y) \leq-\psi(y)+c(x,y)
\end{equation}
for all $y\in Y_G$ with equality holding true if $y\in G_x=\{y\in\R\dvtx (x,y)\in G\}$.
\end{defi}

Note that the function $a_x $ is uniquely determined if $|G_x|\geq2$.
Clearly, a function $\psi$ is $G$-good if and only if there exist
functions $\phi$, $\Delta$ (defined on some set containing $X_G$)
such that
\[
\phi(x)+\psi(y) + \Delta(x) (y-x) \leq c(x,y)
\]
for all $x\in X_G$ and $y\in Y_G$ with equality being satisfied
whenever $(x,y)\in G$.

Subsequently, we will show that in Proposition~\ref{DualMax} there
exists a $\Gamma$-good function $\psi$. We want to explain already at
this stage that for a given $\Gamma$-good function $\psi$, suitable
functions $\phi$ and $\Delta$ can be defined rather explicitly in terms
of the function $\psi$: Fix $x\in X_\Gamma$. By the regularity property, there exist $y^-,y^+$ with
$y^-< x < y^+, (x,y^-), (x,y^+)\in\Gamma$ and a unique affine function
$a_x $ such that $a_x(y^-)= -\psi(y^-)+c(x,y^-)$ and $a_x(y^+)= -\psi
(y^+)+c(x,y^+)$; moreover, $a_x$ lies below the function $y\mapsto
-\psi(y)+c(x,y)$. Writing $g(\cdot)^{**}$ for the convex hull of a
function $y\mapsto g(y)$, we find further that $a_x(y)$ is also smaller
or equal than $(-\psi(\cdot)+c(x,\cdot))^{**}(y)$, with equality holding true
for all $y\in[y^-,y^+]$.
This implies that $a_x(y)=\phi(x)+\Delta(x)(y-x)$,
where
%
%eB.4 #&#
\begin{equation}
\label{FromHull1}\phi(x):= \bigl(-\psi(\cdot)+c(x,\cdot)\bigr)^{**}(x),
\end{equation}
%
%\marginpar{What do the two stars mean?}
and $\Delta(x)$ denotes the derivative of $y\mapsto(-\psi
(\cdot)+c(x,\cdot))^{**}(y)$ at the point $y=x$.

% Let us note already at this stage that for a given $\Gamma$-good
%function $\psi$, suitable functions $\phi$ and $\Delta$ can be defined
%rather explicitly in terms of the function $\psi$. Due to the property
%\Reg\ we can define $\phi$ on the set $X_\Gamma$ by
%\begin{eqnarray}\label{FromHull1}\phi(x):= (-\psi(\cdot)+c(x,\cdot))^{**}(x),
%\end{eqnarray}
% where we write $g^{**}$ for the convex hull of a function $g$.
%\marginpar{at this place or in the proof of Proposition~8.10, can you
%please say more about the definition of $\phi$ and $\Delta$?}

%
%If \Reg\ holds for all $x\in X_G$, there is a useful explicit way to
%choose $\phi$ and $\Delta$. Write $I_G$ for $\Conv(Y_G)^\circ$.
% For a real function $g(\cdot) $ we write $g(\cdot)^{**}$ for its convex hull.
% Then we may define $\phi$ on the set $I_G\supset X_G$ by
%\begin{eqnarray}\label{FromHull}\phi(x):= (-\psi(\cdot)+c(x,\cdot))^{**}(x).
%\end{eqnarray}
%Moreover $\Delta$ can be chosen as the right derivative of the convex
%function $y\mapsto(-\psi(\cdot)+c(x,\cdot))^{**}(y)$ at the point $x$.%
%\marginpar{why not directly the convex function $\phi$? It is the
%same? Can $\phi$ be defined everywhere?Usually there will not be a
%convex function $\phi$. You are right $\phi$ could be defined
%everywhere but it is not important to define it somewhere else. For
%instance if there are many irreducible components each should have
%different $\phi$.}\marginpar{$I_G$ is not defined. M: Now it does not
%appear anymore, no?}

The first step toward the existence of a $\Gamma$-good function in
Proposition~\ref{DualMax} is the following auxiliary result.

%leB.12 #&#
\begin{lem}\label{itsgood}
Let $G\subseteq\Gamma$ be a finite set. Then there exists a $G$-good function.
\end{lem}

\begin{pf}
As $\Gamma$ is regular, there exists a finite set $\tilde G$,
$G\subseteq\tilde G\subseteq\Gamma$ such that $\tilde G$ is regular.
As a consequence of the regularity property, there exists a probability measure
$\alpha$ which has support $\tilde G$ and is a martingale transport
plan between its marginals, that is, satisfies $\alpha\in\M(\mu_0,
\nu_0)$ for $\mu_0:=\proj^x_\# \alpha, \nu_0:=\proj^y_\# \alpha
$. As
$\Gamma$ is finitely optimal, every competitor of $\alpha$ leads at
least to the same amount of costs as $\alpha$, that is, $\alpha$~is an
optimal martingale measure. By the duality theorem of linear
programming, %\marginpar{M:...[I] would just put a quick reference to
%it. Hence I wrote this little footnote. But probably it would be
%better to put into the proof of the variational lemma where I also
%make reference to the simplex algorithm.}
there exist functions $\phi, \Delta\dvtx X_{\tilde G} \to\R,\psi
\dvtx Y_{\tilde
G} \to\R$ such that
\[
\phi(x)+\psi(y) + \Delta(x) (y-x) \leq c(x,y)
\]
for all $(x,y)\in X_{\tilde G}\times Y_{\tilde G}$ with equality
holding for all elements of the set ${\tilde G}$. In particular, $\psi
$ is a $G$-good function.
\end{pf}

The following technical lemma will give us some control over the
variety of different $G$-good functions which can exist for a specified
set $G$.

%leB.13 #&#
\begin{lem}\label{BoundedWobble}
Let $G=\{(x_i,y_i^-),(x_i,y_i^+)\dvtx i=1,2\}$, where $y_i^-< x_i< y_i^+$.
Assume that $]y_1^-,y_1^+[\, \cap\,]y_2^-,y_2^+[ \,\neq\varnothing$.
Given bounded intervals $K_1^\pm$ there exist bounded intervals
$K_2^\pm
$ such that the following holds:
If $\psi$ is $G$-good and $\psi(y_1^\pm)\in K_1^\pm$, then $\psi
(y_2^\pm
)\in K_2^\pm$.

Let $G=\{(x_1,y_1^-), (x_1,y_1^+),(x_2,y_2)\}$, where $y_1^-< x_1< y_1^+$.
Assume that $y_2\in\,]y_1^-,y_1^+[$.
Given bounded intervals $K_1^\pm$ there exists a bounded interval $K_2$
such that the following holds:
if $\psi$ is $G$-good and $\psi(y_1^\pm)\in K_1^\pm$, then $\psi
(y_2)\in K_2$.
\end{lem}

\begin{pf}
We will only prove the first part of the lemma, the second is similar.
Moreover, we will assume that $y_1^- < y_2^- < y_2^+
< y_1^+$. If these numbers are ordered in a different way, the argument
can be adapted easily.
Since $\psi$ is $G$-good, there is an affine function $a_{x_1}$ such that
%
%eB.5 #&#
%eB.6 #&#
%eB.7 #&#
%eB.8 #&#
\begin{eqnarray}
\label{NM1} a_{x_1}\bigl(y_1^-\bigr)&=& - \psi\bigl(
y_1^-\bigr)+c\bigl(x_1,y_1^-\bigr)
\in-K_1^-+c\bigl(x_1,y_1^-\bigr),
\\
\label{NM2}a_{x_1}\bigl(y_1^+\bigr)&=& - \psi\bigl(
y_1^+\bigr)+c\bigl(x_1,y_1^-\bigr) \in
-K_1^++c\bigl(x_1,y_1^+\bigr),
\\
\label{NM3}a_{x_1}\bigl(y_2^-\bigr) &\leq&- \psi\bigl(
y_2^-\bigr)+c\bigl(x_1,y_2^-\bigr),
\\
\label{NM4}a_{x_1}\bigl(y_2^+\bigr) &\leq&- \psi\bigl(
y_2^+\bigr)+c\bigl(x_1,y_2^+\bigr).
\end{eqnarray}
From \eqref{NM1} and \eqref{NM2}, we have a good control over the
possible positions of the affine function $a_{x_1}$. By \eqref{NM3} and
\eqref{NM4}, %\marginpar{Twice the same?}
this translates to a lower bounded for the value of $-\psi(y_2^-) $
[resp., $-\psi(y_2^+) $]. More precisely, we obtain that there exists
a real number $q$ which depends on $K_1^\pm, x_1,y_1^\pm, y_2^\pm$ and
$c$ [but not on the particular values of $\psi(y_2^\pm)$] such that $q
\leq- \psi( y_2^\pm)$.

On the other hand, there exists an affine function $a_{x_2}$ such that
\begin{eqnarray*}
a_{x_2}\bigl(y_1^-\bigr)&\leq&- \psi\bigl(
y_1^-\bigr)+c\bigl(x_1,y_1^-\bigr) \in
-K_1^-+c\bigl(x_2,y_1^-\bigr),
\\
a_{x_2}\bigl(y_1^+\bigr)&\leq&- \psi\bigl(
y_1^+\bigr)+c\bigl(x_1,y_1^-\bigr) \in
-K_1^++c\bigl(x_2,y_1^+\bigr),
\\
a_{x_2}\bigl(y_2^-\bigr)& =&- \psi\bigl( y_2^-
\bigr)+c\bigl(x_2,y_2^-\bigr),
\\
a_{x_2}\bigl(y_2^+\bigr)& =& - \psi\bigl( y_2^+
\bigr)+c\bigl(x_2,y_2^+\bigr).
\end{eqnarray*}
This implies the existence of a constant $p$ such that $p\geq-\psi
(y_2^\pm)$. Summing up, we may choose $K_2^+=K_2^-= [-p,-q]$.
\end{pf}

%leB.14 #&#
\begin{lem}\label{CompactnessNonsense}
There exists a $\Gamma$-good function $\psi$.
\end{lem}
\begin{pf} %\marginpar{You will not like this sentence, but it is true
%and right now I do not have the time to write more. Is it OK? N: Ok,
%you can erase this note}
In Lemma~\ref{itsgood}, we have already seen that for every finite set
$G\subseteq\Gamma$ there exists a $G$-good function. The idea of the
proof is thus to pass to some sort of limit of these functions. To do
so, we aim to confine (properly chosen) $G$-good functions to a compact
subset of the space $\R^{Y_G}$. The existence of this compact set will
be a consequence of Lemma~\ref{BoundedWobble} and Tychonoff's theorem.

We claim that there exist compact intervals $ (K_y)_{y\in Y_\Gamma}$
such that for any finite set $G\subseteq\Gamma$ there is a $G$-good
function $\psi$ such that $\psi(y)\in K_y$ for $y\in Y_G$.

We give the proof under the assumption that $Y_\Gamma\subseteq I$ is
such that $\Conv(Y_{\Gamma})$ is open [such that $\Conv(Y_{\Gamma
})=I$], the other cases are similar.
The irreducibility and regularity properties imply that for every $y\in I$ there exist
$(x, y^-),(x,y^+)\in\Gamma$ such that $y^-< y < y^+$ and $y^-< x <
y^+$. That is, $I$ is the union of intervals of the form $]y^-, y^+[$,
where $(x, y^-),(x,y^+)\in\Gamma$ and $y^-< x < y^+$.
Using that the set $I$ can be written as a countable union of compact
sets, it is straightforward that there exist sequences $(x_k)_{k\in\N
}$, $(y^-_k)_{k\in\N}$, $(y^+_k)_{k\in\N}$ such that the points
$(x_k,y_k^-)$ and $(x_k,y_k^+)$ are in $\Gamma$, we have $y_k^-< x_k<y_k^+$,
\[
\bigcup_{i=0}^k\, \bigl]y_i^-,y_i^+\bigr[
\cap\bigl]y_{k+1}^-,y_{k+1}^+\bigr[ \neq \varnothing ,\qquad k\in\N\quad
\mbox{and}\quad \bigcup_{k\in\N}\, \bigl]y_k^-,
y_k^+\bigr[ =I.
\]
%
%\marginpar{I do not see why $(x_k)$ can be chosen increasing. M: And I
%do not want to put this condition! Good point!}
Given an arbitrary set $G$, a $G$-good function $\psi$ and an affine
function $a$, the function $\psi'=\psi-a$ is again a $G$-good function.
Thus, for all finite $G$ satisfying $(x_0,y_0^-),(x_0,y_0^+)\in G$,
there is a $G$-good function $\psi$ such that $\psi(y_0^-)=\psi(y_0^+)
= 0$.
Iterating (the first part of) Lemma~\ref{BoundedWobble} for $k\in\N$
we find the desired intervals $K_y$ for $y\in\{y^-_k,y^+_k\dvtx k \in\N\}$.

For every $y\in Y_\Gamma$, there exist $x\in\R$ and $k\in\N$ such
that $(x,y) \in\Gamma$ and $y\in(y_k^-,y_k^+)$. Hence, (the second
part of) Lemma~\ref{BoundedWobble} yields the existence of the desired
interval $K_y$ for $y\in Y_\Gamma\setminus\{y^-_k,y^+_k\dvtx k \in\N\}$.

We can view the set $\mathcal{K}:=\prod_{y\in Y_\Gamma} K_y$ as a
subset of the space of all functions from $Y_G$ to $\R$. In the
topology of pointwise convergence, the set $\mathcal{K}$ is compact by
Tychonoff's theorem.

For every finite $G\subseteq\Gamma$, the set
\[
\Psi_G:= \{\psi\in\mathcal{K}\dvtx \psi\mbox{ is $G$-good} \}
\]
is a nonempty closed subset of the set $\mathcal{K}$.
Moreover, the family $(\Psi_G)_{G}$ has the finite intersection
property. For instance, given finite sets $G_1, G_2\subseteq\Gamma$
the intersection of $\Psi_{G_1}$ and $\Psi_{G_2}$ contains $\Psi
_{G_1\cup G_2}$ and is therefore nonempty. By compactness of $\mathcal
{K}$, the intersection
\[
\bigcap_{G\subseteq\Gamma, |G|<\infty} \Psi_G=:
\Psi_\Gamma
\]
of all these sets is nonempty as well. Obviously, any element $\psi\in
\Psi_\Gamma$ is $\Gamma$-good.
\end{pf}

\begin{pf*}{Proof of Proposition~\ref{DualMax}}
By Lemma~\ref{CompactnessNonsense}, there exists a $\Gamma$-good
function $\psi$. We have to show that $\psi$ can be replaced by an
upper semicontinuous function and that there exist appropriate
functions $\phi$ and $\Delta$. We start with the latter task.

Recall that we write $J=\Conv(Y_\Gamma)$ and note that $I\subseteq
J\subseteq\overline I$. %\marginpar{you should say what is the goal of
%this proof. Whith a $\Gamma$-good function we are done, or what?}

For fixed $x\in X_\Gamma$, consider the function $y\mapsto g_x(y)=
-\psi
(y)+ c(x,y), y\in Y_\Gamma$.
% is well-defined on the set $Y_\Gamma$.
For any $x\in I$, let $g_x^{**}\dvtx \R\to[-\infty, +\infty]$ be the
largest convex function which is smaller than $g_x$ on the set
$Y_\Gamma
$ for $x\in X_\Gamma$ and $g_x^{**}=+\infty$ if $x\in I\setminus
X_\Gamma$. For $x\in X_\Gamma$, there exists an affine function which
is smaller than $g_x$. Hence, $g_x^{**}$ does not take the value
$-\infty$ in this case. %\marginpar{You should define $g^{**}$ with
%values in $[-\infty,+\infty]$ and prove that it takes values in $]-
%\infty,+\infty]$.}\marginpar{You also say that $g^{**}$ can be $-
%\infty$ in $I$. Contradiction?}\marginpar{What does ``well-defined''
%mean in this context? Is it ``finite''?}

Since $I= \accentset{\circ}{\Conv(Y_\Gamma)}$ the function $g_x^{**}$
is continuous and finitely valued on the set $J$ for $x\in X_\Gamma$.
As a function on the set $\R$, $g_x^{**}$ may possibly assume the value
$+\infty$.
Moreover, if $x\in I\setminus X_\Gamma$ then $g_x^{**}$ can take the
value $-\infty$.

We now define the function $H\dvtx I\times\R\to[-\infty, \infty]$ by
\[
H(x,y):= \bigl(-\psi(\cdot)+ c(x,\cdot)\bigr)^{**}(y)
\]
and emphasize that $H$ takes finite values on $X_\Gamma\times J$.
Thus, the function $\phi\dvtx I \to[-\infty, \infty[$, defined by
%
%eB.9 #&#
\begin{equation}
\label{FromHull}\phi(x):= \bigl(-\psi(\cdot)+c(x,\cdot)\bigr)^{**}(x)=H(x,x)
\end{equation}
takes finite values on the set $X_\Gamma$.

To prove that $\phi$ is upper semicontinuous, consider for $n\in\N$
the function
\[
H_n(x,y):= \bigl(\bigl(-\psi(\cdot)\vee(-n)\bigr)+ c(x,\cdot)
\bigr)^{**}(y).
\]
It is straightforward to prove that $H_n$ is \emph{continuous} on the
set $I\times J$. Thus, $H= \inf_{n\in\N} H_n$ is upper semicontinuous,
and hence $\phi$ is upper semicontinuous as well.

For each $x\in I$, denote by $\Delta(x)$ the right-derivative of the
convex function $y\mapsto H(x,y)$ in the point $x$ if $H(x,x) >-\infty$
and set $\Delta(x)=0$ otherwise. %Then the function $\Delta$ is upper
%semi-continuous.

%Next we prove that $\phi$ is upper semi-continuous.
%For $n\in\N$ define $\psi_n $ by rest icing $\psi$ to the $\dom\psi
%\cap[-n,n]$. Pick any $x\in X_G$. Since $y\mapsto-\psi(y)+ c(x,y)$ is
%bounded from below by an affine function and $y\mapsto c(x,y)$ is
%uniformly bounded on $[-n,n]$ we obtain that $-\psi_n$ is uniformly
%bounded from below. From this it is straightforward to show that $H\dvtx I
%\times\R\to\R$,
%$$H_n(x,y):= (-\psi_n(\cdot)+ c(x,\cdot))^{**}(y)$$
%is a \emph{continuous function}.
%
%It is then straight forward to verify that the function
%$$H(x,y):= (-\psi(\cdot)+ c(x,\cdot))^{**}(y)$$
%is \emph{continuous} on $X_G\times I$.
%It follows that the function
%\begin{eqnarray}\label{FromHull}\phi(x):= (-\psi
%(\cdot)+c(x,\cdot))^{**}(x)=H(x,x)\end{eqnarray}
%is continuous on the set $X_G$.
%For each $x\in I$ denote by $\Delta(x)$ the right-derivative of the
%convex function $y\mapsto H(x,y)$ in the point $x$. Then the function $
%\Delta$ is upper semi-continuous. %To see this one may argue as
%follows. Extend $H$ to an upper semi-continuous function $\tilde H$ on
%$I \times\R$ by assigning the value $\infty$ outside $I\times J$.
%Then
%$$\Delta(x) = \inf_{k\in\N} \frac{ H(x,x+1/k)- H(x,x)}{1/k},$$
%i.e.\ $\Delta$ is the infimum of upper semi-continuous functions and
%hence upper semi-continuous.
By construction, we then have
\[
\phi(x)+\psi(y) +\Delta(x) (y-x) \leq c(x,y),
\]
for all $(x,y) \in X_\Gamma\times J$. Moreover, as $\psi$ was assumed
to be $\Gamma$-good, equality holds for all $(x,y)\in\Gamma$. [See the
discussion preceding \eqref{FromHull1}.] %\marginpar{Here I need the
%additional property that $x\in\mbox{conv} \Gamma_x.$}

Next, we define a function $\tilde\psi$ by
\[
\tilde\psi(y) = \inf_{x} c(x,y) -\bigl[ \phi(x)+\Delta(x)
(y-x)\bigr].
\]
For every $x$, the function $y\mapsto c(x,y) -[ \phi(x)+\Delta(x)
(y-x)]$ is continuous, hence $\tilde\psi$ is upper semicontinuous. %
%\marginpar{everybody is upper-continuous? I thing it is wrong.
%Something should be l.s.c. ?! M: Note that for fixed $x$, $\phi(x)$ is
%just a number!}
As above, $\phi(x)+\tilde\psi(y) +\Delta(x) (y-x) \leq c(x,y)$ holds
by construction and since $\tilde\psi(y)$ is greater or equal to
$\psi
(y) $ for all $y\in I$ we conclude that the inequality is indeed an
equality on the set $\Gamma$.
\end{pf*}

%sB.3 #&#
\subsection{Integrating the duality relation between \texorpdfstring{$\varphi$, $\psi$,
$\Delta$}{varphi,psi,Delta} and $c$ on the irreducible components}\label{yaha}
Section \ref{DualMaxSection} was a first step in the direction of the
proof of Lemma~\ref{LocalGlobal}. Unfortunately, the functions $\phi,
\psi$ constructed in Proposition~\ref{DualMax} are measurable but not
necessarily integrable. The following lemma will provide a remedy for
this. %\marginpar{Ok like that? M: Nice like that!}

%leB.15 #&#
\begin{lem}
Let $\chi$ be a convex or concave function on some (possibly unbounded)
interval $I$ and assume that $\mu, \nu$ are in convex order and
concentrated on~$I$. Then
%
%eB.10 #&#
\begin{equation}
\label{Quantities}\int \biggl[ \int\chi(y) \,\d\pi _x(y)-\chi(x) \biggr]
\,\d\mu(x)=\int \biggl[ \int\chi(y) \,\d\tilde \pi _x(y)-\chi(x) \biggr]
\,\d\mu(x)
\end{equation}
for all measures $\pi,\tilde\pi\in\M(\mu,\nu)$.
\end{lem}

\begin{pf} We will give the proof in the case where $I=\R$ and $\chi$
convex, the other cases being similar. Note that, leaving integrability
issues aside, the left as well as the right-hand side of \eqref
{Quantities} equal $\int\chi \,\mathrm{d}\nu-\int\chi \,\mathrm{d}\mu$ and in particular we
expect them to be equal. To give a formal proof, we approximate $\chi$
by functions which grow at most linearly so that all involved integrals
do exist.

Denote by $\chi_n$ the smallest convex function which agrees with
$\chi
$ on the interval $[-n,n]$. (So $\chi_n$ is affine on the complement of
$[-n,n]$.)
We have to show that for each $\pi\in\M(\mu, \nu)$.
\[
\int \biggl[ \underbrace{\int\chi(y) \,\d\pi_x(y)-\chi
(x)}_{=:f(x)} \biggr] \,\d \mu(x)=\lim_n\int \biggl[
\underbrace{ \int\chi_n(y) \,\d\pi _x(y)-\chi
_n(x)}_{=:f_n(x)} \biggr] \,\d\mu(x)
\]
Applying Jensen's inequality to the functions $\chi, \chi_n$, we see
that $f, f_n\geq0$ and applying Jensen's inequality to the convex
function $\chi_n-\chi_m$, we see that $f_n\leq f_m$ for $n\leq m$.
Hence, the desired equality follows from the monotone convergence theorem.
\end{pf}
As a consequence of this lemma, the following definition is unambiguous.

%deB.16 #&#
\begin{defi} %\marginpar{N: I have changed nothing about
%concave/convex, but think of making a modification. M: Once you make a
%decision, I am of course happy with it.}
Assume that $\phi,\psi$ are measurable functions and that $\mu, \nu$
are in convex order. Let $\chi$ be a convex\footnote{Of course, the
assertion is also true in the case where $\chi$ is concave, but we do
not need this.} function such that $\phi_0=\phi+\chi$, $\psi_0=\psi
-\chi
$ are uniformly bounded.
Then we set
\[
\int\phi\,\d\mu+ \int\psi\,\d\nu:= \int\phi_0 \,\d\mu+ \int\psi
_0 \,\d \nu+\int \biggl[ \int\chi(y) \,\d\pi_x(y)-\chi(x)
\biggr] \,\d\mu(x),
\]
where $\pi$ is some martingale transport plan.
\end{defi}

%coB.17 #&#
\begin{cor}\label{SimpleDualForBadFunctions}
Assume that we are given measurable functions $\phi,\psi, \Delta$
and a
convex function $\chi$ such that
%
%eB.11 #&#
\begin{equation}
\label{SmallerAgain}\phi(x)+\psi(y) + \Delta(x) (y-x) \leq c(x,y)
\end{equation}
for all $x,y\in I$ and such that $\phi$ and $-\psi$ differ from $\chi$
only by some bounded functions.
Then we have
\[
\int\phi\,\d\mu+\int\psi\,\d\nu\leq\int c \,\d\pi
\]
for any martingale transport plan $\pi$. Furthermore, if equality holds
$\pi$-a.s. in \eqref{SmallerAgain}, then $\int\phi\,\d\mu+\int
\psi\,\d
\nu= \int c \,\d\pi$.
\end{cor}

We are now finally in the position to establish the main result of this section.

%sB.4 #&#
\subsection{Proof of Lemma \texorpdfstring{\protect\ref{LocalGlobal}}{A.2}}
We will first give the proof assuming that $(\mu,\nu)$ is irreducible
on the open interval $I$ (bounded or not). According to Lemma~\ref
{yo3}, we may assume that the finitely optimal set $\Gamma$ is included
in $I\times\bar{I}$ and is regular and irreducible on $I$. It follows
from Proposition~\ref{DualMax} that there exist upper semi-continuous
functions $\phi,\psi\dvtx I\to\,]{-}\infty, \infty]$ and a measurable function
$\Delta\dvtx I\to\R$ such that
\[
\phi(x)+\psi(y) + \Delta(x) (y-x) \leq c(x,y)
\]
for all $x,y\in I$, with equality holding for $(x,y)$ in $\Gamma$.
Recall that the function $\psi$ constructed in Proposition~\ref
{DualMax} is of the form
\[
\inf_{x} c(x,y) -\bigl[ \phi(x)+\Delta(x) (y-x)\bigr].
\]
This leads us to define the convex function $\chi\dvtx I\to\R$ by
\[
\chi(y) = \sup_x \phi(x)+\Delta(x) (y-x).
\]
Since $c$ is assumed to be bounded, it follows that $\psi$ differs
from $-\chi$ only by a bounded function (i.e., $\psi+\chi$ is
bounded). Replacing $\phi$ by
\[
\bigl(-\psi(\cdot)+c(x,\cdot)\bigr)^{**}(x),
\]
it follows also that $\phi$ differs from $\chi$ only by a bounded
function (i.e., $\phi-\chi$ is bounded). Thus, Corollary~\ref
{SimpleDualForBadFunctions} implies that $\pi$ is an optimal transport plan.

Consider now the general case and the decomposition $\pi=(\sum_k \pi
_k)+\eta$ of Theorem~\ref{yo1}, \eqref{PiDec}, where $(\proj^x \pi
_k,\proj^y \pi_k)=:(\mu_k,\nu_k)$ is irreducible. % with open
%interval
%$I_k$.
But $\Gamma$ has full measure for $\pi_k$ [if not $\pi(\Gamma)$ would
be smaller than $1$] and it is finitely optimal for the cost $c$.
According to the first part of the proof, $\pi_k$ is an optimal
martingale transport plan from $\mu_k$ to $\nu_k$. By Theorem~\ref
{yo2}, $\pi$ is optimal and this completes the proof of Lemma~\ref
{LocalGlobal}.

%sC #&#
\section{A self-contained approach to the variational lemma}\label{appB}

In this appendix, we provide a self-contained proof of the variational
lemma (Lemma~\ref{GlobalLocal}, established in Section~\ref
{sec:gamma}). Indeed, we obtain a somewhat stronger conclusion in
Theorem~\ref{cyclic} below. The benefit of this second version is that
Theorem~\ref{cyclic} does not rely on the Choquet's capacability
theorem and that the new approach provides an explicit set $\Gamma$. A
drawback is that we have to assume that the cost function is
continuous. Compared to the approach given in Section~\ref{sec:gamma},
another disadvantage is that the argument does not seem to be adaptable
from $\R\times\R$ to more general product spaces.

%\marginpar{read carefully until the proof of Lemma~9.2. M: I am
%assuming you mean only w.r.t. English but I am afraid that is wishful
%thinking from my side. Do you want me to also understand the proof of
%Lemma~9.2? N: Ich glaube es w\"are gut. Wenn nicht muss ich es machen.
%Aber es ist gemogelt.}
%sC.1 #&#
\subsection{Preliminaries based on Lebesgue's density theorem}
Our aim is to establish Corollary~\ref{utile} which may be viewed as an
avatar of Lemma~\ref{AccumulatingBadPoints}, the uncountable set of
points $a$ being replaced by a set $A$ of positive measure. We start
with the well-known Lebesgue density theorem. It asserts that for an
integrable function $f$ on $[0,1]$ we have
%
%eC.1 #&#
\begin{equation}
\lim_{\varepsilon\to0} \frac{1}{2\varepsilon}\int_{s-\varepsilon
}^{s+\varepsilon}
\bigl\llvert f(s)-f(t)\bigr\rrvert \,\dd t=0
\end{equation}
for almost every $s\in\,]0,1[$. In sloppy language, almost every point
is a ``good'' point. Those points will be called \emph{regular points}
of $f$.
In those regular points $s$, we also have
%
%eC.2 #&#
\begin{equation}
\label{Lg2} \lim_{n \to+\infty} \frac{1}{\lambda(M_n)}\int
_{M_n}\bigl\llvert f(s)-f(t)\bigr\rrvert \,\dd t=0
\end{equation}
for every sequence $(M_n)$ of measurable sets satisfying $M_n\subset
[s-\varepsilon_n,s+\varepsilon_n]$ with $\frac{\lambda
(M_n)}{\varepsilon
_n}$ bounded from below and $\varepsilon_n\to0$. Particular admissible
choices are $M_n=[s,b_n]$ or $]s,b_n]$ and $M_n=[a_n,s]$ or $[a_n,s[$.
As a consequence of (\ref{Lg2}), we have that
%
%eC.3 #&#
\begin{equation}
\lim_{n \to+\infty} \frac{1}{\lambda(M_n)}\int_{M_n}f(t)
\,\dd t=f(s).
\end{equation}

Intervals $B=\, ]q,q']$ or $]{-}\infty,q' ]$ with $q,q'\in\Q\cup\{
-\infty
,+\infty\}$ will be called
\textit{rational semiopen intervals}. By Fubini's theorem, \eqref{Lg2}
implies the following result.

%leC.1 #&#
\begin{lem}\label{littlelemma}
Let $\pi$ be a probability measure on $\R\times\R$ with first marginal
$\lambda_{[0,1]}$. Fix a disintegration $(\pi_x)_{x\in[0,1]}$. There
exists a set $R\subset[0,1], \lambda(R)=1$ such that for $s\in R$, any
rational semiopen interval $B$ and any two sequences $(a_n)_n, (b_n)_n$
satisfying $a_n,b_n\to s$ as well as $a_n\leq s< b_n$ or $a_n< s\leq
b_n$, we have
\[
\lim_{n \to+\infty}\frac{1}{b_n-a_n}\int_{a_n}^{b_n}
\bigl\llvert \pi _t(B)-\pi _s(B)\bigr\rrvert \,\dd
\lambda(t)=0.
\]
\end{lem}

We now extend this lemma to the case where the first marginal of $\pi$
is a general measure $\mu$, not necessarily equal to $\lambda
|_{[0,1]}$. Recall from Section \ref{summary} that $G_\mu$ denotes
the quantile function of $\mu$ and $F_\mu$ the cumulative distribution
function. See Figure~\ref{cumulative} for the graphs of $F_\mu$ and
$G_\mu$ in an example: Here, $\mu$ satisfies $\mu(\{1\})=1/3$ and is
uniform of mass $2/3$ on $[0,1]\cup[2,3]$ (the axis are not scaled in
the same way). Recall that the measure $\mu$ can be written as $(G_\mu
)_\#\lambda$.

%\marginpar{N: The picture is back! Do you like it? See bottom of p6 of
%the report. M: I like it}
%
%f5 #&#
\begin{figure}

\includegraphics{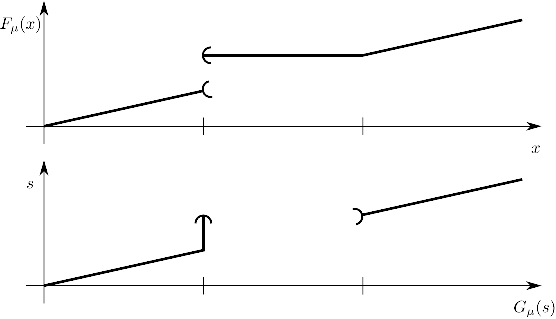}

\caption{The quantile and cumulative distribution functions.}\label
{cumulative}
\end{figure}

%\begin{figure}[ht]
%\begin{center}
%\def\svgwidth{14cm}
%%\input{cumulative.pdf_tex}
%
%\begingroup
% \makeatletter
% \providecommand\color[2][]{%
% \errmessage{(Inkscape) Color is used for the text in Inkscape, but
%the package 'color.sty' is not loaded}
% \renewcommand\color[2][]{}%
% }
% \providecommand\transparent[1]{%
% \errmessage{(Inkscape) Transparency is used (nonzero) for the text
%in Inkscape, but the package 'transparent.sty' is not loaded}
% \renewcommand\transparent[1]{}%
% }
% \providecommand\rotatebox[2]{#2}
% \ifx\svgwidth\undefined
% \setlength{\unitlength}{597.125pt}
% \else
% \setlength{\unitlength}{\svgwidth}
% \fi
% \global\let\svgwidth\undefined
% \makeatother
% \begin{picture}(1,0.480381)%
% \put(0,0){\includegraphics[width=\unitlength]{cumulative.pdf}}%
% \put(0.03866443,0.19658782){\color[rgb]{0,0,0}\makebox(0,0)[lb]{
%\smash{$s$}}}%
% \put(0.77552857,0.0090224){\color[rgb]{0,0,0}\makebox(0,0)[lb]{
%\smash{$G_\mu(s)$}}}%
% \put(0.80232363,0.25017794){\color[rgb]{0,0,0}\makebox(0,0)[lb]{
%\smash{$x$}}}%
% \put(-0.00152816,0.43774336){\color[rgb]{0,0,0}\makebox(0,0)[lb]{
%\smash{$F_\mu(x)$}}}%
% \end{picture}%
%\endgroup
%\caption{The quantile and cumulative distribution functions.}
%\label{cumulative}
%\end{center}
%\end{figure}

The map $G_\mu$ is increasing on $[0,1]$, and hence continuous on the
complement of a countable set $D$, the set of $s\in[0,1]$ such that
$F_\mu^{-1}(s)$ is a nontrivial interval. For such a $s\in D$, the
$\mu
$-measure of $F_\mu^{-1}(s)$ is zero so that $\mu(G_\mu(D))\leq \mu
(F_\mu^{-1}(D))=0$.

Consider a random variable $(U,G_\mu(U),Y)$ on $[0,1]\times\R\times
\R$
such that the law of $U$ is $\lambda$ and the law of $(G_\mu(U),Y)$ is
$\pi$. Let $\tilde{\pi}$ be the law of $(U,Y)$ and $(\tilde{\pi
}_s)_{s\in[0,1]}$ a disintegration with respect to $\lambda$, that
is,
$\tilde{\pi}_s$ is the conditional law of $Y$ given the event $\{U=s\}
$. Apply Lemma~\ref{littlelemma} to this disintegration of $\tilde\pi$
to obtain a set $R$. Let $S\subset\R$ be the set $G_\mu(R\setminus D)$
and let us call $S$ the \emph{set of regular points}.

Note that $S$ has full measure and that it may depend on the
disintegration of $\tilde\pi$.

%leC.2 #&#
\begin{lem}\label{biglemma}
Let $\pi$ be a probability measure on $\R^2$ with first marginal $\mu$
and $(\pi_x)_{x\in\R}$ a disintegration of $\pi$. There exists a set
$S\subset\R$ of measure $\mu(S)=1$ satisfying the following: for any
$x\in S$ and any rational semiopen interval $B$ the limit
\[
\lim_{n \to+\infty}\frac{1}{\mu(N_n)}\int_{N_n}
\bigl\llvert \pi_t(B)-\pi _x(B)\bigr\rrvert \,\dd\mu(t)
\]
is zero for any sequence $N_n=[x-\varepsilon_n,x+\varepsilon_n]$ with
$\eps_n\downarrow0$. %If moreover $\mu(x)=0$ then the sequences
%$N_n=]x,x+\varepsilon_n]$ and $N_n=[x-\varepsilon_n,x[$ are also
%admissible.
\end{lem}
\begin{pf}
We note that if the statement of the lemma holds for one particular
disintegration of $\pi$, then it automatically carries over to any
other disintegration.

Therefore, we will consider a disintegration of $\pi$ which is
convenient for the proof.
Let $\tilde\pi$ and $S$ be as in the discussion preceding Lemma~\ref
{biglemma} and set for $x\in\R$
%
%eC.4 #&#
\begin{eqnarray}
\pi_x= \cases{ \tilde{\pi}_{F_\mu(x)},&\quad $\mbox{if }\mu\bigl(\{x\}
\bigr)=0$,\vspace*{2pt}
\cr
\displaystyle\frac{1}{\mu(\{x\})}\int_{G_\mu^{-1}(\{x\})}\tilde{
\pi}_s \,\dd s,&\quad $\mbox{if }\mu\bigl(\{x\}\bigr)>0$.}
\end{eqnarray}
Let $x$ be a point in $S$ and $N_n=[x-\eps_n,x+\eps_n]$. To prove that
the limit is zero, we distinguish two cases depending on whether or not
$x$ is an atom of $\mu$. The first case is quite straightforward. In
the second case, we will apply Lemma~\ref{littlelemma}.
\begin{itemize}
\item Assume $\mu(\{x\})>0$. As $\bigcap_{n\in\N} N_n=\{x\}$ we have
$\mu(N_n)\downarrow\mu(\{x\})$ as $\eps_n\to0$. Hence,
\begin{eqnarray*}
\frac{1}{\mu(N_n)}\int_{N_n}\bigl\llvert
\pi_t(B)-\pi_x(B)\bigr\rrvert \,\dd\mu (t)&= &
\frac{1}{\mu(N_n)}\int_{\{x\}}\bigl\llvert \pi_t(B)-
\pi_x(B)\bigr\rrvert \,\dd \mu(t)
\\
&&{}+ \frac{1}{\mu(N_n)}\int_{N_n\setminus\{x\}}\bigl\llvert
\pi_t(B)-\pi _x(B)\bigr\rrvert \,\dd\mu(t).
\end{eqnarray*}
The first part of the sum equals $0$ and the second part tends to $0$
since $\llvert \pi_t(B)-\pi_x(B)\rrvert \leq2$ and $[\mu(N_n)-\mu
(x)]/\mu
(N_n)\downarrow0$ as $\eps_n\to0$.

\item Assume $\mu(\{x\})=0$. As $x\in S=G_\mu(R\setminus D)$ there
exists a regular $s_0$ [w.r.t. the disintegration $(\tilde{\pi}_s)_s$]
such that $x=G_\mu(s_0)$ and $G_\mu$ is continuous in $s_0$. As $x$ is
in the interior of $N_n$, $s_0$ is in the interior of $M_n:=G_\mu
^{-1}(N_n)$. Hence, $\lambda(M_n)=\mu(N_n)$ is positive.

We can separate the push-forward measure $\mu=(G_\mu)_\#\lambda$ into
its atomic and its continuous part and integrate accordingly, and thus obtain
%
%eC.5 #&#
\begin{eqnarray}
\label{eq}
&&
\frac{1}{\mu(N_n)}\int_{N_n}\bigl
\llvert \pi_t(B)-\pi_x(B)\bigr\rrvert \,\dd\mu(t)\nonumber\\
&&\qquad=
\frac{1}{\lambda(M_n)}\int_{M_n}\bigl\llvert \pi_{G_\mu(s)}(B)-
\pi _x(B)\bigr\rrvert \,\dd\lambda(s)
\\
&&\qquad\leq\frac{1}{\lambda(M_n)}\int_{M_n}\bigl\llvert \tilde{
\pi}_s(B)-\tilde {\pi }_{s_0}(B)\bigr\rrvert \,\dd\lambda(s).\nonumber
\end{eqnarray}
Here, we used the following properties: (i) if $\mu(\{t\})>0$: Jensen's
inequality for the integration on $\{s\dvtx G_\mu(s)=t\}$, (ii) if $\mu(\{
t\}
)=0$: $G_\mu(s)=t$ implies that $F_\mu(t)=s$ or $s$ is a discontinuity
point of $G_\mu$, so that $F_\mu(t)=s$ almost surely.

But $\lambda(M_n)=\mu(N_n)\to\mu(\{x\})=0$ as $n$ tends to infinity.
Note also that $M_n$ is an interval because $G_\mu$ is nondecreasing.
Hence, we can apply Lemma~\ref{littlelemma} [with the point $s_0$, the
disintegration $(\tilde{\pi}_s)_s$ and the sequence $M_n$] to equation~\eqref{eq}.
Summing up, we obtain that the limit equals zero as required.\quad\qed
% As $\mu(x)=0$ we may replace the sequence of intervals $]x,x+\eps_n]$
%by $N_n=[x,x+\eps_n]$. Then we can conclude as before since in this
%case $M_n=[s,F_\mu(x+\eps_n)]$. The case $[x-\eps_n,x[$ can be dealt
%with analogously.
\end{itemize}
\noqed\end{pf}

We remark that for $\pi\in\mathcal{P}(\R^2)$, if $y\in\Spt(\pi
_x)$, it is not
always true that $(x,y)\in\Spt(\pi)$. We have introduced $S$ in order
to obtain this conclusion for $x\in S$. More precisely, we have the following.

%coC.3 #&#
\begin{cor} \label{utile}
Let $S$ be a set of regular points associated to $(\pi_x)_x$ as in
Lemma~\ref{biglemma} and let $x\in S$. Let $B_1,\ldots, B_k$ be a
family of pairwise disjoint rational semiopen intervals such that $\pi
_x(B_j)>0$ for $j=1,\ldots,k$.

For every $\eps>0$, there exists $A\subset S\cap[x-\varepsilon
,x+\varepsilon]$ such that $\mu(A)>0$ and $\pi_t(B_j)>0$ for
$(j,t)\in\{
1,\ldots,k\}\times A$.% Moreover if $x$ is not an atom of $\mu$, then
%the set $A$ can be chosen as a subset of $]x,x+\varepsilon]$ (resp.\
%as a subset of $[x-\varepsilon,x[$).
\end{cor}

\begin{pf}
Let $\pi, x, \varepsilon$ and the sets $B_j$ be given. Let $(\eps_n)_n$
be a decreasing sequence of positive numbers tending to $0$. For every
$j$, we have
\[
\lim_{n\to+\infty}\frac{1}{\mu(N_n)}\int_{N_n}
\bigl\llvert \pi _x(B_j)-\pi _t(B_j)
\bigr\rrvert \,\dd\mu(t)=0,
\]
where $N_n$ is $[x-\eps_n,x+\eps_n]$ or, in the case $\mu(\{x\})=0$,
one of the intervals $]x,x+\eps_n]$, respectively, $[x-\eps_n,x[$. This implies
\[
\mu\bigl( \bigl\{t\in N_n\dvtx \bigl|\pi_x(B_k)-
\pi_t(B_k)\bigr|>\pi_x(B_k)/2 \bigr\}
\bigr)=o\bigl(\mu(N_n)\bigr).
\]
Therefore,
\[
\mu\bigl( \bigl\{t\in N_n\dvtx \exists j\in\{1,\ldots,k\},\bigl |
\pi_x(B_k)-\pi _t(B_k)\bigr|>
\pi_x(B_k)/2 \bigr\}\bigr)=o\bigl(\mu(N_n)
\bigr)
\]
and
\[
\mu\bigl( \bigl\{t\in N_n\dvtx \exists j\in\{1,\ldots,k\},
\pi_t(B_k)=0 \bigr\} \bigr)=o\bigl(\mu(N_n)
\bigr).
\]
Hence, for $n$ sufficiently large the set
\[
A= \bigl\{t\in N_n\dvtx \forall j\in\{1,\ldots,k\},
\pi_t(B_k)>0 \bigr\}
\]
has positive measure. For $n$ large enough, we also have $\eps
_n<\varepsilon$, which completes the proof.
\end{pf}

%sC.2 #&#
\subsection{Construction of a better competitor when \texorpdfstring{$\Gamma$}{Gamma} supports
a finite nonoptimal coupling}

Let $\V$ be the set of signed measures $\sigma$ on $\R^2$ with Hahn
decomposition $\sigma=\sigma^+-\sigma^-$ such that the following
conditions are satisfied:
\begin{itemize}
\item The total mass of $\sigma$ is $0$.
\item The marginals $\proj^x_\#\sigma$ and $\proj^y_\#\sigma$ vanish
identically.
\item The measure $\proj^y_\#(|\sigma|)=\proj^y_\#\sigma^++\proj
^y_\#
\sigma^-$ has finite first moment.
\item$\sigma$ has a disintegration $(\sigma_x)_x$ such that $\proj
^x_\#
|\sigma|$-a.s., the positive and the negative parts of $\sigma_x$ have
the same mean.
\end{itemize}
If only the three first conditions are satisfied, $\sigma$ will be an
element of $\V'$.

Here, the letter $\V$ is reminiscent to the term \emph{variation}.
Indeed, observe that if $\alpha$ is a positive measure on $\R^2$ such
that $\proj^y_\#\alpha$ has finite first moment and $\beta=\alpha
-\sigma
$ is a positive measure, then $\beta$ is a competitor of $\alpha$ in
the sense of Definition~\ref{def:competitor}. Conversely, for a pair of
competitors $(\alpha,\beta)$, the measures $\alpha-\beta$ and
$\beta
-\alpha$ are elements of $\V$. A notable element of $\V$ is $(\delta
_x-\delta_{x'})\otimes(\lambda\delta_{y^+}+(1-\lambda)\delta
_{y-}-\delta_{\lambda y^++(1-\lambda) y^-})$, the kind of measure that
we have used repeatedly in Sections~\ref{sec:optimal} and \ref
{sec:general}. An element of $\V$ will be called a \emph{variation}. A
variation $\sigma$ is \emph{positive} (resp., \emph{negative}) if
$\int c(x,y) \,\dd\sigma(x,y)> 0$ (resp., $< 0$).

For a cost function satisfying the sufficient integrability condition,
it is not difficult to prove that the following statements are equivalent:
\begin{longlist}[(1)]
\item[(1)] The martingale transport plan $\alpha$ is optimal for the cost $c$.
\item[(2)] For any $\sigma\in\V$ such that $\sigma^+\leq\alpha$, one has
$\int c(x,y) \,\d\sigma(x,y)\leq0$.
\end{longlist}

We can now state the main result of this appendix.% (a kind of optimal $
%\Rightarrow$ cyclical monotonicity statement)

%thC.4 #&#
\begin{them}\label{cyclic}
Assume that $\mu, \nu$ are probability measures in convex order and
that $c\dvtx \R^2\to\R$ is a continuous cost function satisfying the
sufficient integrability condition. Assume that $\pi\in\M(\mu,\nu
)$ is
an optimal martingale transport plan which leads to finite costs. Let
$(\pi_x)_x $ be a disintegration of $\pi$ and $S\subset\R$ a set of
regular points associated to $(\pi_x)_x$ in the sense of Lemma~\ref
{biglemma}. We set
\[
\Gamma=\bigl\{(x,y)\in\R^2\dvtx x\in S\mbox{ and }y\in\Spt(
\pi_x)\bigr\}.
\]
If $\alpha$ is a martingale transport plan such that:
\begin{itemize}
\item the support $\Spt(\alpha)$ of $\alpha$ is finite and
\item the support $\Spt(\alpha)$ is included in $\Gamma$,
\end{itemize}
then the martingale transport plan $\alpha$ is optimal for $c$ between
$\proj^x_\#\alpha$ and $\proj^y_\#\alpha$.

Furthermore, if $\sigma$ is a measure of finite support in $\V$ with
$\Spt(\sigma^+)\subset\Gamma$, it is a nonpositive variation.
\end{them}

\begin{pf}
Let $\alpha$ be as in the theorem and assume for contradiction that
there exists a competitor $\beta$ that leads to smaller costs. We will
prove that $\pi$ cannot be optimal, thus establishing the desired
contradiction. In other words, assume that there is a variation $\sigma
\in\V$ with $\Spt\sigma^+\subset\Spt\alpha$ and $\int c(x,y) \,\d
\sigma(x,y)>0$. We will construct $\tilde\sigma\in\V$ by applying
modifications to $\sigma$ so that $\tilde\sigma^+\leq\pi$ and
$\int
c(x,y) \,\d\tilde\sigma(x,y)>0$. This yields a contradiction since the
competitor $\pi-\tilde\sigma$ is cheaper than $\pi$ with respect to
the cost function $c$.

The argument is based on two lemmas and Proposition~\ref{durdur}, whose
proof is postponed to the next subsection. Let us introduce some
notation. Assume first that $\Spt|\sigma|$ is included in $\{
x_1,\ldots
,x_n\}\times\{y_1,\ldots, y_m\}$ and define for $\eps>0$ the rectangle
$R_{ij}(\eps)=[x_i-\eps,x_i+\eps]\times[y_j-\eps,y_j+\eps]$.

%leC.5 #&#
\begin{lem}\label{boxes}
There exists $\varepsilon>0$ such that the sets $R_{ij}(\eps)$ are
disjoint and any measure $\sigma' \in\V$ satisfying:
\begin{itemize}
\item$|\sigma'|$ is concentrated on $\bigcup_{i,j} R_{ij}(\eps)$ and
\item for $(i,j)\in\{1,\ldots,n\}\times\{1,\ldots,m\}$
\[
\bigl|\sigma-\sigma'\bigr|(R_{ij})\leq\eps,
\]
\end{itemize}
is a positive variation.
\end{lem}

\begin{pf}%[Proof of Lemma~\ref{boxes}]
The argument relies on the continuity of $c$ and is straightforward.
\end{pf}
Let us call $\V(\sigma,\eps)$ the subset of the measures $\sigma
'\in\V$
such that $\sigma'$ satisfies the conditions of the above lemma.
Elements of $\V(\sigma,\eps)$ are positive variations and so are the
elements of the cone $\mathcal{CV}(\sigma,\eps)=\{w\sigma'\in\V
\dvtx w>0\mbox
{ and }\sigma'\in\V(\sigma,\eps)\}$. We want to find a measure
$\sigma
'\in\V(\sigma,\eps)$ and $v$ such that $w\sigma'^+\leq\pi$. For this
purpose, we will use the fact that $\sigma^+$ is concentrated on
$\Gamma$.

Using the notation of Corollary~\ref{utile}, let $A_i$ be the set $A$
associated to $x_i$ and consider an arbitrary family of rational
semiopen intervals $B_k$ with $y_j\in B_j\subset[y_j- \eps,y_j+\eps]$
and $\pi_{x_i}(B_j)>0$ for each $j$. Moreover, we take $A_i\subset
S\cap
[x_i-\eps,x_i+\eps]$ for every $i$.

%prC.6 #&#
\begin{pro}\label{durdur}
Let $\eps>0$.
There are sets $A_1,\ldots, A_n$ with $\mu(A_i)>0$ and $A_i\subset
[x_i-\eps,x_i+\eps]$ such that for $(t_1,\ldots,t_n)\in A_1\times
\cdots
\times A_n$ there is a measure $\sigma_{t_1,\ldots,t_n}\in E$
satisfying the following:
\begin{itemize}
\item We have $\sigma_{t_1,\ldots,t_n}\in\mathcal{CV}(\sigma,\eps)$.
\item The first marginal of $|\sigma_{t_1,\ldots,t_n}|$ has support
$\{
t_1,\ldots,t_n\}$.
\item$\sigma_{t_1,\ldots,t_n}^+\leq\sum_{i=1}^n \mu(A_i)\times
(\delta
_{t_i}\otimes\pi_{t_i})$.
\end{itemize}
\end{pro}

We postpone the proof of Proposition~\ref{durdur} to the next subsection.

Note that $\sigma_{t_1,\ldots,t_n}$ is not the measure $\tilde\sigma$
we are looking for. Nevertheless, it satisfies almost all the
conditions. It is in $\V$ and even in $\mathcal{CV}(\sigma,\eps)$ so
that according to Lemma~\ref{boxes} it is a positive variation. The
only missing condition it that $\sigma_{t_1,\ldots,t_n}^+$ is not
smaller than $\pi$. We provide a remedy in the following lemma.

%leC.7 #&#
\begin{lem}[(A variation $\tilde\sigma$ leading to the
contradiction)]\label{arrival}
The measure
\[
\tilde\sigma=\frac{1}{\mu(A_1)\times\cdots\times\mu(A_n)}\iiint _{A_1\times
\cdots\times A_n}\sigma_{t_1,\ldots,t_n} \,\dd
\mu(t_1)\otimes\cdots \otimes\,\dd\mu(t_n)
\]
is in $\mathcal{CV}(\sigma,\eps)$ and satisfies both $\iint c(x,y)
\,\d
\tilde\sigma(x,y)>0$ and $\tilde\sigma^+\leq\pi$. Hence, $\pi
-\tilde
\sigma$ gives rise to smaller costs than $\pi$.
\end{lem}

\begin{pf}
As all $\sigma_{t_1,\ldots,t_n}$ are in $\mathcal{CV}(\sigma,\eps)$,
they are positive variations. Hence, $\tilde\sigma$ which is an average
of these measures in $\V$ is also a positive variation. Let us prove
that $\tilde\sigma^+\leq\pi$. Observe that $\tilde\sigma^+$ is again
the average of the positive parts $\sigma^+_{s_1,\ldots,s_n}$. By
Proposition~\ref{durdur}, this is smaller than
\begin{eqnarray*}
&&\frac{1}{\mu(A_1)\times\cdots\times\mu(A_n)}\iiint_{A_1\times\cdots
\times
A_n}\sum_{i=1}^n
\mu(A_i) (\delta_{t_i}\otimes\pi_{t_i}) \,\dd\mu
(t_1)\otimes\cdots\otimes\,\dd\mu(t_n)
\\
&&\qquad=\sum_{i=1}^n\int_{A_i}
\biggl(\frac{\iiint(\delta_{t_i}\otimes
\pi
_{t_i}) \,\dd\mu(t_1)\otimes\cdots\otimes\widehat{\,\d\mu
(t_i)}\otimes\cdots
\otimes\,\d\mu(t_n)}{\mu(A_1)\times\cdots\times\widehat{\mu
(A_i)}\times
\cdots\times\mu(A_n)} \biggr)\,\d\mu(t_i)
\\
&&\qquad=\sum_{i=1}^n\int_{A_i}
(\delta_{t_i}\otimes\pi_{t_i} )\,\dd\mu (t_i)=
\pi|_{\bigcup_{i=1}^n A_i\times\R}. %\qedhere
\end{eqnarray*}
\upqed\end{pf}
Up to Proposition~\ref{durdur}, we have thus proved Theorem~\ref{cyclic}.
\end{pf}

%sC.3 #&#
\subsection{Proof of Proposition \texorpdfstring{\protect\ref{durdur}}{B.6}}
%Let us start the proof of Proposition~\ref{durdur}. We will decompose
%it in several steps and lemmas.

Recall the definitions and notation of Theorem~\ref{cyclic} and
Proposition~\ref{durdur}. In particular, $\sigma$ has finite support
included in $\Gamma$. It is also included in some product set $\{
x_1,\ldots,x_n\}\times\{y_1,\ldots,y_m\}$ where we choose $m$ and $n$
as small as possible. For $\tau\in\V$, we denote the support of
$\proj
^x_\#(|\tau|)$ by $X(\tau)$ and the support of $\proj^y_\#(|\tau|)$ by
$Y(\tau)$ so that $\{x_1,\ldots,x_n\}=X(\sigma)$ and $\{y_1,\ldots
,y_m\}
=Y(\sigma)$. Let $d\leq n\cdot m$ be the cardinality of $\operatorname
{spt}(\sigma
^+)$ and denote its elements by $p_1,\ldots,p_d$.

For measures of finite support, the conditions for being in $\V$ can be
simplified. A measure $\tau$ is in $\V$ if:
\begin{longlist}[(1)]
\item[(1)] for every $y\in Y(\tau)$, $L_y(\tau)$ defined as $\sum_{x\in
X}\tau(x,y)$ is zero,
\item[(2)] for every $x\in X(\tau)$, $C_x(\tau)$ defined as $\sum_{y\in
Y}\tau(x,y)$ is zero,
\item[(3)] for every $x\in X(\tau)$, $M_x(\tau)$ defined as $\sum_{y\in
Y}\tau(x,y)\times y$ is zero.
\end{longlist}
Moreover, the measure $\tau$ is an element of $\V'$ if the conditions
(1) and (2) are satisfied.

We introduce some further notation. For every $\tau\in\V'$ of finite
support, we introduce a relation between the points of $X(\tau)$. We
write $x\to x'$ if there are $y,y'$ such that $y>y'$ and $\tau(x,y),
\tau(x',y')$ are not zero. If $x\to x'$ and $x'\to x$ we write
$x\leftrightarrow x'$ and will say that $x$
\textit{double-touches} $x'$. If $\tau\in\V$, for any point $x\in
X(\tau)$ an important consequence of condition (3) is that there exist
three distinct points $y, y', y''$ such that $\tau(x,y), \tau(x,y')$
and $\tau(x,y'')$ are not zero. Hence, $x\leftrightarrow x$ if $x\in
X(\tau)$. However the relation $\leftrightarrow$ is not transitive. If
$x\in X$ double-touches both $x'$ and $x''$, we say that
\textit{$x$ is a bridge over $x'$ and $x''$}. In particular, if
$x\leftrightarrow x'$ the point $x$ is a bridge over $x'$ and $x$ itself.

Roughly speaking for $\tau\in\V'$, the relation $x\to x'$ means that it
is possible to replace $\tau$ (in a continuous manner) by a signed
measure $\tau'\in\V'$ such that $\tau^+$ and $\tau'^+$ have the same
support. Applying this modification $\tau\mapsto M_x(\tau)$ increases
while $\tau\mapsto M_{x'}(\tau)$ decreases (and their sum remains
constant). More precisely, consider $y,y'$ such that $y>y'$ and $\tau
(x,y),\tau(x',y')$ are both nonzero. Let $m$ be the measure $(\delta
_{x}-\delta_{x'})\otimes(\delta_{y}-\delta_{y'})$. Notice that $m$ is
an element of $\V'\setminus\V$. Considering $\tau^h=\tau+h\cdot m$ and
$h>0$, we have
\[
M_{x}\bigl(\tau^h\bigr)-M_{x}(\tau)=h\cdot
M_{x}(m)=h\cdot\bigl(y-y'\bigr)>0.
\]
We only consider positive $h$ in order to keep the same support for
$(\tau^h)^+$ and $\tau^+$. In particular this prohibits that $\tau
(x,y')>0$ and $\tau(x',y)>0$. For the same reason, we choose $h\in
[0,h_0[$ where $h_0=\max(|\tau(x,y)|,|\tau(x',y')|)$. Indeed, if
$\tau
(x,y)<0$ then the same applies to $\tau^h(x,y)$.

If we want to make $M_{x}$ and $M_{x'}$ vary in the opposite direction,
we may consider the relation $x'\to x$ in place of $x\to x'$. Thus, $x
\leftrightarrow x'$ allows to make small variations of $M_{x}$ and
$M_{x'}$ in the one or the other direction. If there is a bridge
$x''\in X(\tau)$ over $x$ and $x'$, we have exactly the same freedom as
if $x\leftrightarrow x'$. The next lemma is a tool for finding bridges
between points when $\tau\in\V$.

%leC.8 #&#
\begin{lem}\label{deco}
Let $\tau$ be a finitely supported element of $\V$ and $(x,y)\in
X(\tau
)\times Y(\tau)$ such that $\tau(x,y)> 0$. Let $G \subset X(\tau)$ be
the subset of points $x'$ such that:
\begin{itemize}
\item there exists a bridge over $x$ and $x'$,
\item$\tau(x',y)<0$.
\end{itemize}
Then
\[
\tau(x,y)+\sum_{x'\in G }\tau\bigl(x',y
\bigr)\leq0.
\]
\end{lem}
\begin{pf}
Condition (1) implies that if every $x'\in X(\tau)$ satisfying $\tau
(x',y)< 0$ is connected with $x$ by a bridge, we are done. Conversely,
assume that there exists $x'\in X(\tau)$ such that $\tau(x',y)<0$ and
there is no bridge between $x$ and $x'$. Then for any $x_0\in X(\tau)$
the measure $|\tau|$ restricted to $\{x_0\}\times\R$ is concentrated
on $\{x_0\}\times[y,+\infty[$ or $\{x_0\}\,\times\,]{-}\infty,y]$ (if not it
would be a bridge between $x$ and $x'$). Let $X^1\sqcup X^2$ be the
partition of $X(\tau)$ induced by this remark and $\tau^i$ the
restriction of $\tau$ to $X^i\times\R$ for $i=1,2$. Without loss of
generality, we can assume $x\in X^1$. Let us prove that $\tau^1$ and
$\tau^2$ are in $\V$. Actually, they coincide with $\tau$ on vertical
lines so that they satisfy conditions (2) and (3). %on $C_x$ and $M_x$.
The total mass of $\tau$ on the horizontal lines that are not equal to
$\R\times\{y\}$ is zero as well. Thus, as $\tau^i(\R^2)=0$, we obtain
$\tau^i(X^i\times\{y\})=0$ for $i=1,2$. This yields condition (1) %on
%$L_y$
for $\tau^1$ and $\tau^2$. Hence, these measures are in $\V$.

As $\tau^1\in\V$, applying condition (1) we obtain that any $x'_1\in
X^1$ such that $\tau(x'_1,y)< 0$ is connected with $x$ by a bridge.
Indeed with condition (2) and the definition of $X^1$, we know that
there are $y'$ and $y''$ in $]y,+\infty[$ such that $\tau(x,y')\neq0$
and $\tau(x'_1,y'')\neq0$. Hence, we have $x\leftrightarrow x'_1$. So
we can apply the first remark to $\tau^1$ in place of $\tau$. Indeed,
$G$ is the set of points of $x_1\in X(\tau^1)$ such that $\tau
(x_1,y)=\tau^1(x_1,y)<0$.
\end{pf}

%For $\{p_1,\ldots,p_d\}\subset\R^2$ let $E(\alpha)$ the set of points
%$(q_1,\ldots,q_d)$ where each $q_k$ has the same first coordinate as
%$p_k$ for each $k\in\{1,\ldots,d\}$.

%leC.9 #&#
\begin{lem}\label{move}
Let $\tau$ be a finitely supported positive variation and consider
$\operatorname{spt}
(\tau^+)=\{p_1,\ldots,p_d\}\subset\R\times\R$. There exists $\eps>0$
such that if $q_k\in\R^2$ has the same first coordinate as $p_k$ and
$|p_k-q_k|<\eps$ for every $k\in\{1,\ldots,d\}$, then there exists a
sequence of positive variations $(\tau_k)_{k=1}^d$ such that $|\tau_k|$
has finite support and $\tau_k^+$ has support $\{q_1,\ldots
,q_k,p_{k+1},\ldots,p_d\}$.
%(The measures $\alpha_k$ can actually be chosed continuously as
%functions of $(q_1,\ldots,q_d)$ in $V$).
\end{lem}

\begin{pf}
Let $\eps$ be a positive real number. Let us denote by $X$ the support
of $\proj^x_\#(|\tau_k|)$ for some $k\in\{1,\ldots,d\}$ (which does not
depend on $k$). We explain how to build $\tau_k$ from $\tau_{k-1}$.
Roughly speaking, we are moving $p_k=(a,b)$ to a position $q_k=(a,b')$,
where $|b'-b|<\eps$. Doing this, we have to take care to stay in $\V$.
The conditional measure $\tau_k|_x$ can easily be forced to preserve
mass zero [condition~(2)] during this operation but there are two
difficulties: for each $y$ the conditional measures $\tau_k|_y$ must
have mass zero [condition (1)]. The second problem is that for each
$x\in X$ the positive and the negative part of $\tau_k|_x$ must have
the same mean [condition (3)].

Let us go into details. We define $\tau_k$ from $\tau_{k-1}$ in two
steps: the first step is a vertical translation. Applying Lemma~\ref
{deco} to $p_k=(a,b)$, we obtain a measure $m$ concentrated on $X(\tau
)\times\{b\}$ that satisfies the following conditions:
\begin{itemize}
\item$m(\R^2)=0$,
\item$m^+$ is concentrated on the point $p_k=(a,b)$ and $m(a,b)=\tau
_{k-1}(a,b)$,
\item$m^-$ is concentrated on a set $G\times\{b\}$ such that any
$x\in
G$ is connected with $a$ by a bridge and $m^-\leq\tau_{k-1}^-$.
\end{itemize}
Let us denote $m$ by $\zeta\otimes\delta_{b}$. We replace $\tau_{k-1}$
by $\tau'_{k-1}=\tau_{k-1}+\zeta\otimes(\delta_{b'}-\delta_{b})$. Doing
this, we preserve conditions (1) and (2), that is, the measure is
still in $\V'$, but condition (3) is possibly violated. Recall that
$\zeta$ has mass zero. %Notice that $\zeta$ can be written $\zeta=
%\sum_{x\in A} \zeta(x)\cdot(\delta_x-\delta_a)$.
It follows that
\[
M_a\bigl(\tau'_{k-1}\bigr)+\sum
_{x \in G}M_x\bigl(\tau'_{k-1}
\bigr)=0.
\]
Using the bridges between $a$ and the elements of $G$ (these bridges
are available for $\tau'_{k-1}$ as they were for $\tau_{k-1}$ assuming
that $\eps$ is sufficiently small), we can modify the measure and make
$M_a$ and $M_x$ for $x\in G$ equal to $0$. Call $\tau_k$ the result of
this procedure. Observe that if the variations are sufficiently small
then the points of positive mass are exactly $q_1,\ldots
,q_k,p_{k+1},\ldots,p_d$ as we want. As in Lemma~\ref{boxes}, we also
obtain that the variations $(\sigma_k)_{k=1}^d$ are positive provided
that $\eps>0$ is sufficiently small.
\end{pf}

We can now prove Proposition~\ref{durdur}. Let $\sigma\in\V$ of finite
support as in the proof of Theorem~\ref{cyclic}. Observe that $\sigma$
can be written as a sum
\[
\sum_{k=1}^d \zeta_k\otimes
\delta_{y_k},
\]
where for $k\in\{1,\ldots,d\}$ the signed measure $\zeta_k$ has its
positive part concentrated in one point. Given $k$, let $\omega_k$ be a
probability measure on $\R$ with expectation $y_k$ (the same as
$\delta
_{y_k}$). We consider
\[
\sum_{k=1}^d \zeta_k\otimes
\omega_k
\]
and easily convince ourselves that this measure is an element of $\V$.
We will apply this transformation not directly to $\sigma$ but to a
measure $\sigma_d\in\mathcal{V}(\sigma,\eps)$, that we build in the
following paragraph.

The proof of the proposition proceeds as follows. Consider the family
of points $(r_1,\ldots, r_d)$ of the support of $\sigma^+$ and pick
$\eps$ as in Lemma~\ref{boxes}. For each point $r_k=(a,b)$, we consider
a rational semiopen interval $B_k\ni b$ of diameter smaller than $\eps
$. Using Corollary~\ref{utile}, we obtain a family $(A_i)_{1\leq i\leq
n}$ and we can assume that these sets are included in $[x_i-\eps
,x_i+\eps]$. We fix a point $(t_1,\ldots,t_n)$ of $A_1\times\cdots
\times A_n$. For each $k\in\{1,\ldots,d\}$ we can write $r_k$ in the
form $(x_i,b)$. We have $\pi_{t_i}(B_k)>0$. Let now $p_k=(t_i,b)$ and
$q_k=(t_i,\tilde{y})$ where $\tilde{y}=\frac{1}{\pi_{t_i}(B_k)}\int_{B_k}
y \,\dd\pi_{t_i}(y)$. Apply Lemma~\ref{move} to the measure $\sigma
_0\in
\V$ obtained from $\sigma$ by translating horizontally the mass
concentrated on the line $\{x_i\}\times\R$: The measure $\sigma
|_{x_i}$ equals precisely $\sigma_0|_{t_i}$. The other parameters
$(p_1,\ldots,p_d)$ and $(q_1,\ldots,q_d)$ have just been constructed.
Applying Lemma~\ref{move}, we obtain a measure $\sigma_d\in\V
(\sigma
,\eps)$ concentrated on $\{t_1,\ldots,t_n\}\times\R$ and
$\operatorname{spt}\sigma
_d^+=\{q_1,\ldots,q_d\}$. Next, we perform the transformation explained
above where each $\omega_k$ has the form $\frac{1}{\pi_{t_i}(B_k)}\pi
_{t_i}|_{B_k}$ for some $(i,k)$. The measure $\overline{\sigma_d}$ we
obtain is in $\V(\sigma,\eps)$ but it may not satisfy the condition
$\overline{\sigma_d}^+\leq\sum_{i=1}^n \mu(A_i)\delta
_{t_i}\otimes\pi
_{t_i}$. However, this inequality does hold for $w\overline{\sigma
_d}^+\in\mathcal{CV}(\sigma,\eps)$ if $w$ is a sufficiently small
positive constant.
\end{appendix}

\section*{Acknowledgments}
The authors wish to thank Michel \'Emery, Martin Goldstern, Claus
Griessler, Martin Keller-Ressel, Vincent Vigon and the participants of
the Winter school 2012 in Regen for enlightening discussions on the
topic of this paper. We are also indebted to a particularly careful
referee for numerous valuable suggestions and for pointing out a
mistake in the initial version of this manuscript.

% imsref loaded by akundreckaite, 2015-01-21 13:48:25
%

%\begin{appendix}
%\section{}
%\end{appendix}

%\section*{Acknowledgments}

%\begin{supplement}[id=suppA]
%\slink[doi]{10.1214/00-AOPXXXXSUPP} %[doi,text={...}] - jei reikia
%suskaldyti doi
%\sfilename{aopXXXX\_supp.pdf}
%\end{supplement}

%\begin{thebibliography}{99}
%\bibitem[\protect\citeauthoryear{}{}]{r1}
%\bibitem{r1}
%\end{thebibliography}

\printaddresses
\end{document}